\documentclass[a4paper,12pt]{amsart}
\usepackage{amsmath,amsfonts,amssymb,amsthm}
\usepackage{subfig}
\usepackage{graphicx}
\usepackage{color}
\usepackage[all]{xy}
\usepackage{url}
\usepackage{setspace}
\usepackage{booktabs}
\usepackage{longtable}
\usepackage{latexsym}
\usepackage{comment}
\usepackage[usenames,dvipsnames]{xcolor}
\usepackage{hyperref}
\usepackage{mathtools}
\usepackage[textsize=tiny]{todonotes}
\usepackage{tikz}
\usetikzlibrary{cd}
\usetikzlibrary{patterns}
\usetikzlibrary{math}
\usetikzlibrary{shapes}
\usetikzlibrary{decorations.fractals}

\newtheorem{theorem}{Theorem}[section]
\newtheorem{lemma}[theorem]{Lemma}
\newtheorem{proposition}[theorem]{Proposition}
\newtheorem{corollary}[theorem]{Corollary}

\theoremstyle{definition}
\newtheorem{definition}[theorem]{Definition}
\newtheorem{example}[theorem]{Example}

\newtheorem{remark}[theorem]{Remark}

\definecolor{cKlaus}{rgb}{0.1,0.55,0.03}
\definecolor{cKlausOK}{rgb}{0.6,0.10,0.33}
\definecolor{intOrange}{rgb}{1.0,.310,.0}

\newcommand{\ktrash}[1]{}

\newcommand{\cN}{\mathcal{N}}
\newcommand{\cP}{\mathcal{P}}

\newcommand{\cH}{\mathcal{H}}

\DeclareMathOperator{\rank}{rk}

\DeclareMathOperator{\coker}{coker}

\DeclareMathOperator{\orb}{orb}

\DeclareMathOperator{\Nef}{Nef}
\DeclareMathOperator{\Eff}{Eff}

\DeclareMathOperator{\face}{face}

\DeclareMathOperator{\conv}{conv}

\newcommand{\PP}{\mathbb P}

\newcommand{\R}{\mathbb R}

\newcommand{\T}{\mathbb T}
\newcommand{\V}{\mathbb V}
\newcommand{\Z}{\mathbb Z}

\newcommand{\CC}{{\mathcal C}}
\newcommand{\CD}{{\mathcal D}}
\newcommand{\CE}{{\mathcal E}}
\newcommand{\CF}{{\mathcal F}}

\newcommand{\CH}{{\mathcal H}}
\newcommand{\CI}{{\mathcal I}}

\newcommand{\CL}{{\mathcal L}}
\newcommand{\CO}{{\mathcal O}}
\newcommand{\CP}{{\mathcal P}}

\newcommand{\CR}{{\mathcal R}}
\newcommand{\CS}{{\mathcal S}}
\newcommand{\CT}{{\mathcal T}}

\newcommand{\entspr}{\mathop{\widehat{=}}}  %

\newcommand{\toric}{\T\V}  %
\newcommand{\gH}{\operatorname{H}}

\newcommand{\Pic}{\operatorname{Pic}}
\newcommand{\Div}{\operatorname{Div}}
\newcommand{\Cl}{\operatorname{Cl}}
\definecolor{intOrange}{rgb}{1.0,.310,.0} 
\newcommand{\til}[1]{\widetilde{#1}}

\newcommand{\ifff}{\,\Longleftrightarrow\,}

\newcommand{\gExt}{\operatorname{Ext}}
\newcommand{\RHom}{\operatorname{\R Hom}}

\renewcommand{\div}{\operatorname{div}}

\newcommand{\kst }{\,|\;}

\newcommand{\ku}{\underline}

\newcommand{\surj}{\rightarrow\hspace{-0.8em}\rightarrow}

\newcommand{\kss}{\scriptscriptstyle}
\newcommand{\kbb}{{\kss \bullet}}
\newcommand{\ko}{\overline}

\newcommand{\kk}{K}

\newcommand{\normal}{\cN}

\newcommand{\vect}{\overrightarrow}

\newcommand{\Matrr}[3]{\left(\begin{array}{@{}*{#1}{r}|*{#2}{r}@{}} #3
\end{array}\right)}

\newcommand{\cc}{{c^{\ell_2}}}
\newcommand{\cb}{{\ko{c}}}

\newcommand{\eN}[1]{e^{#1}}
\newcommand{\fN}[1]{f^{#1}}
\newcommand{\uN}[1]{u^{#1}}
\newcommand{\vN}[1]{v^{#1}}
\newcommand{\eM}[1]{e_{#1}}
\newcommand{\fM}[1]{f_{#1}}

\newcommand{\nImm}{\operatorname{\CI}}
\newcommand{\xa}{\alpha} %
\newcommand{\xb}{\beta} %
\newcommand{\xc}{{\gamma}} %
\newcommand{\gammaText}[1]{{#1}}
\newcommand{\noGammaText}[1]{}
\newcommand{\cl}{s} %
\newcommand{\clDown}{s_{\downarrow}} %
\newcommand{\clUp}{s^{\uparrow}} %
\newcommand{\DeltaUp}{\Delta^{\mbox{\tiny up}}}
\newcommand{\Rect}{\CR} %
\newcommand{\kU}{U} %
\newcommand{\kV}{V} %
\newcommand{\hZ}{d_{\mbox{\rm\tiny int}}} %
\newcommand{\tm}{maximal}
\newcommand{\nex}{non-extendable}
\newcommand{\leftR}{V^{\mbox{\rm\tiny left}}}
\newcommand{\rightR}{V^{\mbox{\rm\tiny right}}}
\newcommand{\upperR}{H^{\mbox{\rm\tiny up}}}
\newcommand{\lowerR}{H^{\mbox{\rm\tiny down}}}
\newcommand{\nl}{{l}}
\newcommand{\mes}{maximal exceptional sequence}

\newif\ifcomment

\newcommand{\bbox}{\ko{\cP}_V}

\newcommand{\helix}{\hslash} %
\newcommand{\heLex}{\helix_{\mbox{\tiny lex}}}
\newcommand{\wbox}{ensemble}
\renewcommand{\sharp}{\#}

\newcommand{\ssect}[1]{Subsection~\eqref{#1}}
\newcommand{\ssects}[1]{Subsections~\eqref{#1}}
\newcommand{\thLex}{A}
\newcommand{\thHeight}{B}
\newcommand{\thHelex}{C}
\newcommand{\thClass}{D}
\newcommand{\thFull}{E}
\definecolor{originM}{RGB}{180,0,0}
\definecolor{origin}{RGB}{0,130,0}
\definecolor{regi}{HTML}{677081}
\definecolor{regii}{RGB}{0,0,160}
\definecolor{regiii}{RGB}{0,0,160}
\definecolor{redSeg}{RGB}{160,0,0}
\definecolor{center}{RGB}{180,120,60}
\definecolor{parallelogram}{RGB}{135,206,250}
\definecolor{coverA}{RGB}{180,0,0}
\definecolor{coverB}{RGB}{10,180,0}
\definecolor{coverC}{RGB}{10,0,180}
\definecolor{coverD}{RGB}{80,20,60}
\definecolor{skin}{HTML}{FFECC9}
\definecolor{pumpkin}{HTML}{FEDFA9}
\definecolor{piggy}{HTML}{FFB99D}
\definecolor{fiolet}{HTML}{CD8F9C}
\definecolor{granat}{HTML}{677081}
\definecolor{ciemnyblekit}{HTML}{91A1B8}
\definecolor{oliwkowy}{HTML}{627037}
\definecolor{ciemnazielen}{HTML}{394D2E}
\definecolor{ciemnyfiolet}{HTML}{424444}
\definecolor{mocnyfiolet}{HTML}{717299}
\definecolor{jasnyfiolet}{HTML}{B0ABCC}
\definecolor{bladyfiolet}{HTML}{C9C7DB}
\definecolor{lightblue}{RGB}{135,206,250}
\definecolor{darkblue}{RGB}{0,0,160}
\definecolor{darkgreen}{RGB}{0,160,0}
\definecolor{veryPeri}{RGB}{102,103,171}

\newcommand{\immreg}[4]{
\def\la{#1} %
\def\lb{#2} %
\def\ch{#3} %
\def\cb{#4} %
\draw[color=oliwkowy!40] (-\cb-0.3,-\la/\ch) grid (\la+0.3,\lb+\la/\ch);
\fill[pattern color=parallelogram, pattern=north west lines]
  (0,0) -- (\la,-\la/\ch) -- (\la,0) -- cycle
  (\la-\cb,\lb) -- (-\cb,\lb+\la/\ch) -- (-\cb,\lb) -- cycle;
\draw[thick, color=black]
  (0,0) -- (\la,-\la/\ch) -- (\la,0)
  (\la-\cb,\lb) -- (-\cb,\lb+\la/\ch) -- (-\cb,\lb);
\fill[pattern color=parallelogram, pattern=north west lines]
  (0,0) -- (-\cb,\cb/\ch) -- (-\cb,\lb) --
  (\la-\cb,\lb) -- (\la,\lb-\cb/\ch) -- (\la,0) -- cycle;
\fill[thick, color=origin]
  (0,0) circle (5pt);
\draw[thick, color=origin]
  (\la-\cb,\lb) circle (5pt)
  (-0.5,-0.5) node{$0$};
\foreach \xa in {-1,...,1}
\foreach \xb in {-\cb,...,\la}
\foreach \y in {2,...,\lb} {
  \fill[thick, color=regi]
    (\xa-\cb,\y-1) circle (5pt)
    (\xb+1,\y-1) circle (5pt); }
\pgfmathsetmacro{\hMax}{(\la-2)/\ch}
\foreach \y in {0,...,\hMax} {
  \pgfmathsetmacro{\wMax}{\la-\ch*\y-1}
  \foreach \x in {1,...,\wMax} {
    \fill[thick, color=regii]
      (\la-\x,-\y) circle (5pt)
      (\x-\cb,\y+\lb) circle (5pt); }}
\foreach \x in {2,...,\la} {
  \fill[thick, color=regii]
  (\x-1,0) circle (5pt)
  (\la-\cb-\x+1,\lb) circle (5pt); }
}
\begin{document}
\parindent0mm

\title[Maximal exceptional sequences]
{The structure of exceptional sequences on toric varieties
of Picard rank two}

\author[K.~Altmann]{Klaus Altmann%
}
\address{Institut f\"ur Mathematik,
FU Berlin,
K\"onigin-Luise-Str.~24-26,
D-14195 Berlin
}
\email{altmann@math.fu-berlin.de}
\author[F.~Witt]{Frederik Witt}
\address{
Fachbereich Mathematik,
U Stuttgart,
Pfaffenwaldring 57,
D-70569 Stuttgart
}
\email{witt@mathematik.uni-stuttgart.de}

\begin{abstract}
For a smooth projective toric variety of 
Picard rank two we classify all exceptional 
sequences of invertible sheaves which 
have maximal length. In particular, 
we prove that unlike non-maximal 
sequences, they (a) remain exceptional 
under lexicographical reordering 
(b) satisfy strong spatial constraints 
in the Picard lattice (c) are full, 
that is, they generate the derived 
category of the variety.
\end{abstract}

\maketitle
\setcounter{tocdepth}{1}

\section{Introduction}
\label{sec:Intro}

\subsection{Fullness of exceptional sequences}
\label{subsubsec:ExExSeq}
Let $X$ be a smooth projective variety over 
{an algebraically closed field~$\kk$} 
and $\CD(X)$ the bounded category of coherent sheaves 
on $X$. In his ICM talk~\cite{KuznetsovICM}, 
Kuznetsov posed the following

\medskip

{\bf Fullness Conjecture.} 
If $\CD(X)$ is generated by an exceptional sequence, 
then any exceptional sequence of the same length is full.

\medskip

{Though a counterexample to Kuznetsov's conjecture 
was recently given on a rational surface {~\cite{krah}}, 
this question is still of interest. {For instance, the 
fullness property for $X$} implies the absence of so-called phantom 
categories appearing in~\cite{phantomsGalkin1} and~\cite{Phantom}.}

\medskip

{In this paper we shall address the question of 
fullness in the case of exceptional sequences of 
line bundles on a toric 
variety of Picard rank $2$ defined over an 
algebraically closed field $\kk$.}
Let us first explain the general setting before we
comment on our assumptions. A sequence of elements 
$\CE_1,\ldots,\CE_N$ in $\CD(X)$ is called 
{\em exceptional} {{if} the derived homomorphisms 
satisfy 
$$
\RHom(\CE_i,\CE_i)=\kk\hspace{0.8em}\mbox{and}\hspace{0.8em}
\RHom(\CE_j,\CE_i)=0\;\mbox{for all }j>i.
$$
}
For instance, if the $\CE_i$ are 
given by invertible sheaves and $X$ 
satisfies {$\R\Gamma(X,\CO_X)=\kk$}, the first 
condition holds automatically. Further, the 
second condition becomes $\gExt^\kbb(\CE_j,\CE_i)=0$.

\medskip

An exceptional sequence is {\em full} if it 
generates the derived category. This means that 
$\langle\CE_1,\ldots,\CE_N\rangle$, the smallest 
triangulated full subcategory of $\CD(X)$ 
containing the $\CE_i$, is $\CD(X)$ itself. 
The length $N$ equals the rank of 
the $\operatorname{K}$-group $\operatorname{K}_0(X)$. Any other exceptional
sequence has at most $N$ elements; it is called
{\em maximal} if it attains this bound. 
In particular, full sequences are maximal.

\medskip

The simplest example is Beilinson's
full exceptional sequence on $\PP^{d}$~\cite{beilinson},
namely
$$
\CD(\PP^{d})=\langle\CO_{\PP^n}(0),\ldots,
\CO_{\PP^{d}}({d})\rangle.
$$
More generally, Kawamata proved existence of 
full exceptional sequences consisting of complexes 
of coherent sheaves on any smooth 
projective toric variety $X$ (or even a smooth 
toric DM-stack for that matter)
~\cite{kaw1,kaw2,kaw3}. Note that,
{for $\dim X = d$, }the 
rank of $\operatorname{K}_0(X)$ equals the number of 
$d$-dimensional cones $\sharp\Sigma(d)$ of the 
underlying fan $\Sigma$. 
Equivalently, this coincides with the number of 
vertices of any polytope associated with an 
ample divisor on $X$. 

\medskip

In view of Beilinson's example one could even 
hope of finding full exceptional sequences 
consisting of invertible sheaves 
$\CL_i$ instead of general complexes $\CE_i$, 
but this fails even for toric Fano 
varieties~\cite{efimov}. However, 
existence of such sequences was established 
in~\cite{CostaMiroRoig} if $\Sigma$ is a 
splitting fan. Geometrically, this means that the 
toric variety $X=\toric(\Sigma)$ arises as the 
total space of a sequence of successive fibrations 
via $X_0=\PP^n,X_1,\ldots,X_r=X$ with 
$X_i=\PP(E_{i-1})$ for a sum of invertible sheaves 
$E_{i-1}$ on $X_{i-1}$. 

\medskip

\subsection{Exceptional sequences of line bundles on 
toric varieties of Picard rank two}
\label{subsec:CoSmToVaPic2}
From now on we solely consider exceptional sequences
of line bundles on smooth projective toric varieties 
of Picard rank $2$, the easiest examples { after 
$\PP^{d}$ among toric varieties $X=\toric(\Sigma)$ 
with splitting fan $\Sigma$.} 
The basic invariant of $X$ is 
\begin{center}
the pair $(\ell_1,\ell_2)$ of integers 
$\ell_1$, $\ell_2\geq2$  
\end{center}
which indicates that $X$ fibres over $\PP^{\ell_1-1}$ 
with fibre $\PP^{\ell_2-1}$. {Moreover,
\[
d=\dim X=\ell_1+\ell_2-2,  
\]
and} the defining fan $\Sigma$ 
contains exactly two rays more than $d$; see 
\ssect{pic2:KleinSchmidt} for further details. 
In particular, $\sharp\Sigma(d)=\ell_1\ell_2$.
We refer to the trivial fibration 
$\PP^{\ell_1-1}\times\PP^{\ell_2-1}$ as the 
{\em product case} and to a nontrivial fibration as 
the {\em twisted case}. For the latter, the order 
of the two numbers $\ell_1,\ell_2\geq 2$ really 
matters. In dimension two where $\ell_1=\ell_2=2$, 
we just find the family of Hirzebruch surfaces
but the complexity quickly increases with dimension. 
{The fibration structure also implies that 
we have a canonical identification 
$$
\Pic(X)\cong\Z^2 
$$
given by the primitive generators of the nef cone, 
see Subsection~\eqref{pic2:nef}. Geometrically, 
these generators are given by the pullback of 
$\CO_{\PP^{\ell_1-1}}(1)$ and a relative hyperplane 
section of the fibration.}

\medskip

Since $\gExt^\kbb(\CL_j,\CL_i)=
\mathrm{H}^\kbb(X,\CL_j^{-1}\otimes\CL_i)$
a sequence of line bundles $\CL_0,\ldots,\CL_N$ 
is exceptional if and only if $\CL_j^{-1}\otimes\CL_i$, 
$i<j$, lies in the locus of cohomologically trivial 
line bundles inside the Picard group. This locus is 
explicitly known for toric varieties given by a splitting 
fan~\cite{immaculate}. Second, we also understand 
the extensions provided by nontrivial cohomology, 
cf.~\cite{dop} and~\cite{displayingExt}.

\medskip

\subsection{Maximal exceptional sequences are lexicographic}
\label{subsubsec:LexExSeq}
Since properties of exceptional sequences such 
as fullness only depend on their underlying set,
it is natural to look for a canonical order. 
{We call an exceptional sequence 
{\em vertically} respectively 
{\em horizontally orderable} if it remains 
exceptional for the lexicographic order on 
$\Pic(X)\cong\Z^2$ where priority is given 
either to the ``vertical'' or ``horizontal'' 
direction. In general, lexicographic 
reordering destroys exceptionality of the 
sequence, but remarkably, this does not 
happen for maximal exceptional sequences.}

\medskip

{\bf Theorem {\thLex} [see Theorems 
\ref{thm:MESRelProdCase} and \ref{thm:MESRel}].} 
{\em Let $\cl\subseteq\Z^2$ be a maximal 
exceptional sequence of invertible sheaves on 
a smooth projective toric variety $X$ of Picard 
rank two. {In the product case, $\cl$ 
is either vertically or horizontally orderable. 
In the twisted case, $\cl$ is vertically orderable.}}

\medskip

In contrast, it was shown in~\cite[Example 3.5]{p1p1p1} 
that there are maximal exceptional sequences on
$\PP^1\times\PP^1\times\PP^1$ {that are not 
orderable with respect to any of the six possible 
lexicographic orders on 
$\Pic(\PP^1\times\PP^1\times\PP^1)\cong\Z^3$}.

\medskip

\subsection{Maximal exceptional sequences are densely packed}
\label{subsubsec:PropExSeq}
It is well-known that the shape of exceptional 
sequences also impacts on the derived category, 
see for instance~\cite{resCatLMS},
~\cite{resCatCompositio} or~\cite{Mironov}. 
Our next structure result concerns the 
spatial size of maximal exceptional sequences. 

\medskip

{\bf Theorem {\thHeight} [see Theorems
~\ref{thm:SlimPaPb} and~\ref{cor:HeightConstraint}].} 
{\em Let $\cl\subseteq\Z^2=\Pic(X)$ be a 
maximal exceptional sequence of invertible 
sheaves on a projective toric variety $X$ 
of Picard rank two. {In the twisted case, 
the {\em height}, which is the minimal number of rows 
of a horizontal strip containing $\cl$, is bounded by 
$2\ell_2$. In the product case, either the height or 
the {\em width} (the minimal number of columns of a 
vertical strip containing $\cl$) is bounded by $2\ell_2-1$ 
or $2\ell_1-1$, respectively.}}

\medskip

Again, it is easy to construct counterexamples 
for non maximal sequences. What is striking 
about this result is that it is false for higher 
Picard rank. In~\cite[Example 3.4]{p1p1p1} 
it was shown that maximal exceptional sequences 
on $\PP^1\times\PP^1\times\PP^1$ can spread 
arbitrarily far and simultaneously in all 
three directions.

\medskip

\subsection{The classification of {\mes}s}
\label{subsubsec:HeLexing}
{Let
$$
\Rect_{\ell_1,\ell_2}:=\{(a,b)\in\Z^2\mid 
0\leq a<\ell_1,\,0\leq b<\ell_2\}
$$
be the {\em standard rectangle} associated 
with the pair $(\ell_1,\ell_2)$. 
The sequence given by
$$
\cl_{a+b\ell_1}:=(a,b)\in\Rect_{\ell_1,\ell_2}
$$
is maximal exceptional with respect to the vertical 
lexicographical order, and so is any sequence obtained 
by a global shift, or by shifting 
{each row of points $(\kbb,b)$ independently by some $(a_b,0)$.}
We refer to these \mes s as 
{\em vertically trivial}, see also 
Subsection~\eqref{subsec:TrivMes}. Similarly, 
there are also horizontally trivial sequences in the 
product case. Composing the well-known helix operator 
(e.g.~\cite{rudakov}) with lexicographical reordering 
yields the heLex operator {$\heLex$}, see 
Subsection~\eqref{subsec:HeLexPC}.}

\medskip

{\bf Theorem {\thHelex} 
[see Theorems~\ref{thm:HeLexingPC} and~\ref{thm:HeLexing}].}
{\em Any vertically orderable \mes\ can be transformed 
into a vertically trivial \mes\ by applying $\heLex$ 
at most $\ell_1\ell_2$ times. Mutatis mutandis, the 
statement holds for horizontally \mes s in the 
product case.}

\medskip

{Theorem {\thHelex} can be recast into a 
more constructive version:}

\medskip

{\bf Theorem {\thClass}
[see Theorems~\ref{thm:mesAfterAdShiftPC} 
and~\ref{thm:mesAfterAdShift}].} 
{\em Any {\mes} is determined {and explicitely described}
by a so-called 
{\em admissible set} $X\subseteq(-\xb,\ell_2)
+\Rect_{\ell_1,\ell_2}$ 
{and a complementing partner $Y$} which consist either 
of horizontal or vertical lines of consecutive 
points.}

\medskip

The precise definition of admissible set{s and complementing partners} 
is given in Definition~\ref{def:AdmSetPC} and
~\ref{subsec:mesClassification}. In this way 
we can classify the totality of {\mes}s.

\medskip

\subsection{Fullness of maximal exceptional sequences}
\label{subsubsec:FullnessExSeq}
{Finally, we show that maximality is 
sufficient for fullness. Viewing exceptionality 
as an orthogonality condition in the derived 
category, this is comparable to the fact that 
in a finitely generated vector space any linearly 
independent set of maximal cardinality generates 
the space. This follows either directly from Theorem~\thLex\ 
(admitting that the helix (!) operator preserves fullness 
from the general theory), or from a combinatorial argument 
building on Theorems~{\thLex},~{\thHeight} and~{\thClass}.} 

\medskip

{\bf Theorem {\thFull} 
[see Theorems~\ref{thm:FullnessTheoremPaPb} 
and~\ref{thm:FullnessTheorem}.} 
{\em An exceptional sequence of invertible sheaves 
on a smooth projective toric variety of Picard rank 
two is full if and only if it is maximal.}

\medskip

For toric Fano varieties of Picard rank two and 
dimension less than five this was shown in
~\cite{pic2Fano} by direct calculations in a 
case-by-case analysis {(special cases were 
already considered in~\cite{LiuYangYu})}. 
{In contrast},~\cite{p1p1p1} required aid of a computer to 
prove the same result for $\PP^1\times\PP^1\times\PP^1$, 
the easiest example of Picard rank three.
{
In \cite{BorisovWang}, the claim that maximality implies fullness
was proven (with completely different methods) 
for toric DM stacks $X$ with Picard rank two under the additionel
assumptions that $X$ is Fano and that the sequence is even 
strongly exceptional. See also~\cite{DHLKK} for a much more general approach.

\medskip

In Theorem [5.3 = B] of \cite{exsec} one can find a precise characterization 
of strongly exceptional sequences among all \mes s
in terms of the building pairs $(X,Y)$ consisting of an admissible set and a
complementing partners in the sense of
\ssect{subsec:ClassProdCase} and (\ref{subsec:mesClassification}) 
of the present paper.
}

\medskip

{On the other hand, the structural Theorems {\thLex}-{\thFull} 
go beyond this fullness issue as they provide a completely general 
and conceptional treatment of \mes s for Picard rank two
-- including a complete classification of all \mes s}. As a final 
comment we note that our arguments are purely
combinatorial (at the expense of possible shortcuts, cf.\ 
for instance Remark~\ref{rem:HelixOperator}). 
A more geometrical approach is given 
in the sequel~\cite{exsec}.

\medskip

\subsection{The example $\mathbf{\PP^1\times\PP^1}$}
\label{subsec:GenDerCat}
As an illustration of our theorems we consider 
$\PP^1\times\PP^1$ and show how we generate the 
whole lattice $\Pic(X)=\Z^2$ out of the maximal 
exceptional sequence 
$$
\cl=(\cl_0,\,\cl_1,\,\cl_2,\,\cl_3,)=
\big(0,\,(-3,1),\,(-2,1),\,(1,2)\big),
$$
see also (a) of Figure~\ref{fig:InfecEx}. It 
is vertically ordered and of height $3$ in 
accordance with Theorems {\thLex} and {\thHeight} 
(for horizontally ordered examples of unbounded 
height see~\ref{ex:P1P1}).

\medskip

The main tool is the Beilinson sequence
$$
0\to\CO_{\PP^1}\to\CO_{\PP^1}(1)\oplus\CO_{\PP^1}(1)
\to\CO_{\PP^1}(2)\to0
$$
on $\PP^1$. As we explain in Example~\ref{ex:Beilinson}
it allows us to generate or ``fill'' the whole 
(horizontal or vertical) line whenever it carries two 
consecutive points of $\langle\cl\rangle$.

\medskip

\begin{figure}[ht]
\newcommand{\spaceA}{\hspace*{1.5em}}
\newcommand{\scaleA}{0.3}
\begin{tikzpicture}[scale=\scaleA]
\draw[color=oliwkowy!40] (-3.3,-3.3) grid (3.3,3.3);
\fill[thick, color=origin]
  (0,0) circle (9pt);
\fill[thick, color=black]
  (-3,1) circle (9pt) (-2,1) circle (9pt) (1,2) circle (9pt);
\draw[thick, color=black]
   (0,-5) node{(a)};
\end{tikzpicture}
\spaceA
\begin{tikzpicture}[scale=\scaleA]
\draw[color=oliwkowy!40] (-3.3,-3.3) grid (3.3,3.3);
\draw[thick, color=red]
  (-4,1) -- (4,1);
\fill[thick, color=origin]
  (0,0) circle (9pt);
\fill[thick, color=black]
  (-3,1) circle (9pt) (-2,1) circle (9pt) (1,2) circle (9pt);
\fill[thick, color=red]
  (-3,1) circle (9pt) (-2,1) circle (9pt);
\draw[thick, color=black]
   (0,-5) node{(b)};
\end{tikzpicture}
\spaceA
\begin{tikzpicture}[scale=\scaleA]
\draw[color=oliwkowy!40] (-3.3,-3.3) grid (3.3,3.3);
\draw[thick, color=red]
  (-4,1) -- (4,1) (1,-4) -- (1,4);
\draw[thick, color=blue]
  (-4,1) -- (4,1);
\fill[thick, color=origin]
  (0,0) circle (9pt);
\fill[thick, color=black]
  (-3,1) circle (9pt) (-2,1) circle (9pt) (1,2) circle (9pt);
\fill[thick, color=red]
  (1,1) circle (9pt);
\draw[thick, color=black]
   (0,-5) node{(c)};
\end{tikzpicture}
\spaceA
\begin{tikzpicture}[scale=\scaleA]
\draw[color=oliwkowy!40] (-3.3,-3.3) grid (3.3,3.3);
\draw[thick, color=red]
  (-4,1) -- (4,1) (1,-4) -- (1,4) (0,-4) -- (0,4);
\draw[thick, color=blue]
  (-4,1) -- (4,1) (1,-4) -- (1,4);
\fill[thick, color=origin]
  (0,0) circle (9pt);
\fill[thick, color=black]
  (-3,1) circle (9pt) (-2,1) circle (9pt) (1,2) circle (9pt);
\fill[thick, color=red]
  (0,1) circle (9pt);
\draw[thick, color=black]
   (0,-5) node{(d)};
\end{tikzpicture}
\spaceA
\begin{tikzpicture}[scale=\scaleA]
\draw[color=oliwkowy!40] (-3.3,-3.3) grid (3.3,3.3);
\draw[thick, color=blue]
  (-4,1) -- (4,1) (1,-4) -- (1,4) (0,-4) -- (0,4);
\fill[thick, color=red]
  (0,3) circle (9pt) (1,3) circle (9pt)
  (0,-3) circle (9pt) (1,-3) circle (9pt);
\fill[thick, color=origin]
  (0,0) circle (9pt);
\fill[thick, color=black]
  (-3,1) circle (9pt) (-2,1) circle (9pt) (1,2) circle (9pt);
\draw[thick, color=black]
   (0,-5) node{(e)};
\draw[thick, color=red]
  (-4,3) -- (4,3) (-4,-3) -- (4,-3);
\end{tikzpicture}
\caption{Filling $\Pic(\PP^1\times\PP^1)=\Z^2$ from $\cl$. The green dot in (a) -- (e) marks the origin $0\in\Z^2$.}
\label{fig:InfecEx}
\end{figure}

\medskip

Here and in the sequel, let 
\begin{center}
{$[x=a]$ and $[y=b]$ denote the lines $\{(a,j)\mid j\in\Z\}$ 
and $\{(i,b)\mid i\in\Z\}$ in $\Z^2$.}
\end{center}
{Right from the beginning in (a), we} have a 
consecutive pair of elements in $\cl$ on $[y=1]$ 
so that we can generate or ``fill'' the entire 
line $[y=1]$, cf.\ (b). Hence $[y=1]\subseteq\langle
\cl\rangle$.

\medskip

Next, we fill the vertical line $[x=1]$ 
using the consecutive pair $(1,1)$, $(1,2)$ in 
(c) -- we showed in (b) that $(1,1)\in
\langle\cl\rangle$. Similarly, we can fill 
the line $[x=0]$ since $(0,1)\in\langle\cl\rangle$, 
cf.\ (d).

\medskip

It follows that we can fill any horizontal 
line $[y=b]$ for $(0,b)$, $(1,b)\in\langle\cl\rangle$, 
$b\in\Z$, see (e). Hence $\Pic(\PP^1\times\PP^1)
\subseteq\langle\cl\rangle$. As 
$\Pic(\PP^1\times\PP^1)$ generates 
$\CD(\PP^1\times\PP^1)$, $\cl$ is full. 

\medskip

The proof of Theorem~\thFull\ for 
$\PP^{\ell_1-1}\times\PP^{\ell_2-1}$ is a direct 
generalisation of this example. The key step 
consists in proving existence of sufficiently many 
``horizontal'' and ``vertical'' consecutive points 
inside the sequence, see Section~\ref{sec:ProdCase}. 
The twisted case works differently.
In particular, it requires a suitable ``vertical'' 
Beilinson sequence which reflects the fine structure 
of the twist, see \ssect{subsec:appPicTwo}. 

\medskip

\subsection{Plan of the paper}
\label{subsec:PlanPaper}
The first Sections~\ref{sec:TorGeo} to~\ref{sec:Relax} 
provide the 
necessary background and establish the {main} technical 
tools. 
{Though large parts of the proofs in the product 
and the twisted case are similar and equally technical}
we found it more perspicuous to 
treat them separately with the product case serving as 
guideline. {Consequently, Sections~\ref{sec:ProdCase}
and~\ref{sec:ClassPC} prove Theorems~{\thLex}-{\thFull} 
for the product case, while 
the remaining Sections~\ref{sec:MES} to~\ref{sec:GenDerCat} 
are devoted to proving the twisted versions.}

\medskip

\section{Some background on toric geometry}
\label{sec:TorGeo}
We briefly review some features of toric 
geometry which we shall use in the paper.
For a short introduction to toric geometry, 
see~\cite{fultonToric}.

\medskip

\subsection{Torus invariant divisors}
\label{subsec:TorInvDiv}
Let $X=\toric(\Sigma)$ be a smooth toric variety 
with underlying fan $\Sigma$. The 
$r$-dimensional cones in $\Sigma$ form the 
subset $\Sigma(r)$. 
{Similarily, for any $\sigma\in\Sigma$,
the set $\sigma(r)\subseteq\Sigma(r)$ denotes the set of its
$r$-dimensional faces. All these}
cones live in the real 
vector spaces $N_\R=N\otimes_\Z\R$, where 
$N$ is the lattice of one-parameter groups 
of rank equal to $d=\dim X$. It 
is dual to the character lattice $M$ of $X$. 

\medskip

These lattices link to the group of torus 
invariant Weil divisors $\mathrm{Div}_T(X)$ 
and the class group $\Cl(X)$ of $X$ as follows. 
Any ray, that is, an element $\rho\in\Sigma(1)$, 
corresponds to a unique torus orbit 
$\orb(\rho)$ of codimension one, {namely its closure
$D_\rho=\ko{\orb}(\rho)$.} 
For $m\in M$ we define 
$\rho^\ast(m)=\div(\chi^m)=\sum_\rho\langle m,
\rho\rangle D_\rho$, where $\rho$ denotes 
both the ray and its primitive generator in $N$. 

\medskip

Moreover, let $\nabla$ be a lattice polytope in $M$
which is {\em compatible} with $\Sigma$, i.e., 
$\Sigma$ is a refinement of the normal fan 
$\normal(\nabla)$ of $\nabla$. This comes with an
associated Weil divisor 
$$
D_\nabla=
-\sum_\rho \min\langle\nabla,\rho\rangle\cdot D_\rho.
$$
The induced line bundle 
$\CO(\nabla):=\CO(D_\nabla)$ is globally generated
by the monomials $\chi^m$ with $m\in\nabla\cap M$. 
Further, for any {$\sigma\in\Sigma(d)$} we have an 
associated vertex $v(\sigma)\in\nabla\cap M$ 
which is characterised by
$$
\langle v(\sigma),\rho\rangle=\min\langle \nabla,\rho\rangle
\hspace{0.7em}\mbox{for}\hspace{0.7em}
\rho\in\sigma(1).
$$
It gives rise to a local generator of 
$\CO(\nabla)$ on $\toric(\sigma)$, namely 
$$
\CO(\nabla)|_{\toric(\sigma)}=
\chi^{v(\sigma)}\cdot\CO_{\toric(\sigma)}.
$$
{For non-maximal cones $\sigma\in\Sigma$ this works 
similarly. However, the vertices $v(\sigma)$ are only determined up to $\sigma^\bot$.}

\medskip

Finally, 
{we have the exact sequence}
\begin{equation}
\label{eq:MSeq}
\xymatrix{
0\ar[r]&M\ar[r]^-{\rho^\ast}&
\Big[\mathrm{Div}_T(X)\cong \bigoplus
\limits_{\rho\in\Sigma(1)}\Z D_\rho\Big]
\ar[r]^-{\pi}&\Cl(X)\ar[r]&0
}
\end{equation}
{if the primitive generators span $N_\R$ 
(for instance if $X$ is complete).} In particular, 
$\,\rank\Cl(X)=\sharp\Sigma(1)-\dim X$.

\medskip

\subsection{Exact sequences reflecting polyhedral covers}
\label{subsec:exSeqpolCov}
We recall a method from \cite{displayingExt} to 
transform polyhedral inclusion~/~exclusion sequences 
into exact sequences of sheaves on toric varieties 
$X=\toric(\Sigma)$.

\medskip

We start with a so-called {\em $\Sigma$-family} of 
lattice polytopes $S=\{\nabla_{\! i} 
\mid i=1,\ldots,n\}$ in $M_\R$. This means that for 
all $I\subseteq [n]:=\{1,\ldots,n\}$ the 
intersections
$$
\textstyle
\nabla_{\! I}:=\bigcap_{i\in I}\nabla_{\! i}
\mbox{ for }I\neq\varnothing
$$
as well as 
$$
\textstyle
\nabla_\varnothing:=\bigcup_{i\in [n]}\nabla_{\! i}
$$
are compatible with $\Sigma$,
i.e., $\Sigma\leq\normal(\nabla_{\! I})$.

\medskip

The second ingredient is the standard Koszul complex 
$(\Lambda^\kbb\kk^n,d)$ where
$$
d:\Lambda^{p+1}\kk^n\to\Lambda^{p}\kk^n,\quad 
e_I\mapsto\sum_{i\in I}(-1)^i e_{I\setminus \{i\}}
$$ 
for any $I\subseteq[n]$ with $\sharp I=p+1$,
and $e_I:=\wedge_{i\in I}e_i$ for the standard
basis vectors $e_1,\ldots,e_n\in\kk^n$. Tensoring 
with $\kk[M]$ yields the exact complex
$$
\xymatrix@C=1em{
0\ar[r]& \kk[M] \cdot e_{[n]} \ar[r]&
\renewcommand{\arraystretch}{0.6}
\oplus_{\hspace{-0.9em}
\begin{array}{c}\scriptstyle I\subseteq [n]\\
\scriptstyle\sharp I=n-1\end{array}\hspace{-0.7em}}
\kk[M] \cdot e_{I} \ar[r]& \ldots \ar[r]&
\oplus_{i\in[n]} \kk[M] \cdot e_i \ar[r]&
\kk[M] \cdot e_{\varnothing} \ar[r]& 0.
}
$$
For each $I\subseteq[n]$ the vector space 
$\kk[M]\cdot e_I$ appears as a direct summand 
inside this complex and contains the 
finite-dimensional subspace
$$
S(I):=\kk[\nabla_{\! I}\cap M]\cdot e_I := 
\oplus_{m\in \nabla_{\! I}\cap M} \,\kk\cdot
\chi^m\cdot e_I.
$$
It follows from~\cite[Section 3]{displayingExt} 
that these subvector spaces define an exact 
subcomplex $S^\kbb\subseteq\kk[M]\otimes
\Lambda^\kbb\kk^n$. Moreover, we have the

\begin{theorem}\cite[Theorem 12]{displayingExt}
\label{thm:S_Seq}
$S^\kbb$ is the complex of global sections of 
the equivariant, exact complex $\CS^\kbb$ given 
by the globally generated sheaves 
$$
\CS^k=\bigoplus_{\scriptstyle I\subseteq [n],
\,\sharp I=k}
\CO_X(\nabla_{\! I})
$$ 
on $X=\toric(\Sigma)$, namely
$$
\xymatrix@C=1em{
0\ar[r]& \CO_X(\nabla_{\![n]}) \ar[r]&
\oplus_{i=1}^n \CO_X(\nabla_{\![n]\setminus\{i\}}) \ar[r]&
\ldots \ar[r]& \oplus_{i=1}^n \CO_X(\nabla_{\! i}) \ar[r]&
\CO_X(\nabla_{\!\varnothing}) \ar[r]& 0.
}
$$
\end{theorem}

\begin{example}
\label{ex:HirzOne}
Let $X=\CH_1$ be the first Hirzebruch surface,
see also the picture on the left hand side 
of Figure~\ref{fig:Hirzebruch}. We consider 
the $\Sigma$-family $S=\{\nabla_{\!1},\,
\nabla_{\!2}\}$ provided by the triangle and 
the quadrangle in the middle box of the 
polyhedral exact sequence
$$
\newcommand{\scaleA}{0.5}
\newcommand{\spaceA}{\hspace*{1em}}
0 \hspace{0.5em} \to
\spaceA
\begin{tikzpicture}[scale=\scaleA]
\draw[color=oliwkowy!40] (-0.3,-0.3) grid (1.3,0.3);
\draw[thick, color=black]
(0,0) -- (1,0) -- (0,0);
\fill[thick, color=black]
(0,0) circle (2pt) (1,0) circle (2pt);
\draw[thick, color=origin]
(0,0) circle (3pt);
\end{tikzpicture}
\spaceA
\to
\spaceA
\raisebox{-0.5em}{\fbox{$
\begin{tikzpicture}[scale=\scaleA]
\draw[color=oliwkowy!40] (-0.3,-0.3) grid (1.3,1.3);
\draw[thick, color=black]
(0,0) -- (1,0) -- (0,1) -- (0,0);
\fill[thick, color=black]
(0,0) circle (2pt) (1,0) circle (2pt) (0,1) circle (2pt);
\draw[thick, color=origin]
(0,0) circle (3pt);
\end{tikzpicture}
\raisebox{0.5em}{$\oplus$}
\begin{tikzpicture}[scale=\scaleA]
\draw[color=oliwkowy!40] (-0.3,-0.3) grid (2.3,1.3);
\draw[thick, color=black]
(0,0) -- (2,0) -- (1,1) -- (0,1) -- (0,0);
\fill[thick, color=black]
(0,0) circle (2pt) (1,0) circle (2pt) (2,0) circle (2pt)
(0,1) circle (2pt) (1,1) circle (2pt);
\draw[thick, color=origin]
(0,1) circle (3pt);
\end{tikzpicture}
$}}
\spaceA
\to
\spaceA
\raisebox{-1em}{\begin{tikzpicture}[scale=\scaleA]
\draw[color=oliwkowy!40] (-0.3,-0.3) grid (2.3,2.3);
\draw[thick, color=black]
(0,0) -- (2,0) -- (0,2) -- (0,0);
\fill[thick, color=black]
(0,0) circle (2pt) (1,0) circle (2pt) (2,0) circle (2pt)
(0,1) circle (2pt) (1,1) circle (2pt) (0,2) circle (2pt);
\draw[thick, color=origin]
(0,1) circle (3pt);
\end{tikzpicture}}
\spaceA
\to \hspace{0.5em} 0.
$$
Here, the green dots indicate the position of the 
origin in each of the polyhedra. Using the notation 
from \ssect{pic2:fan} below, the leftmost polyhedron
$\nabla_{\!1}\cap\nabla_{\!2}$ equals $\kU$, and the 
triangle $\nabla_{\!1}$ itself is just $\kV$. The 
sequence may be therefore translated into
\begin{equation}
\label{eq:ExacSeq1}
0\to \CO_{\CH_1}(\kU) \to \CO_{\CH_1}(\kV)\oplus
\CO_{\CH_1}(\kU+\kV)\to\CO_{\CH_1}(2\kV)\to 0.
\end{equation}
{Let us translate this sequence into 
classical language. The right hand side of Figure~\ref{fig:fanH1}
displays the fan of the blow-up $b:\CH_1\to\PP^1$. 

\begin{figure}[ht]
\newcommand{\spaceA}{\hspace*{3.5em}}
\newcommand{\scaleA}{0.8}
\begin{tikzpicture}[scale=\scaleA]
\draw[thick, color=black]
(0,0) -- (2,0) (0,2) -- (0,-1.5) (0,0) -- (-1.5,-1.5);
\draw[thick, color=black]
  (2,0.4) node{$u^1$}
  (-1.5,-1) node{$u^2$}
  (-0.5,1.8) node{$v^1$}
  (0.5,-1.3) node{$v^2$};
\end{tikzpicture}
\spaceA
\raisebox{6ex}{$\stackrel{b}{\longrightarrow}$}
\spaceA
\begin{tikzpicture}[scale=\scaleA]
\draw[thick, color=black]
(0,0) -- (2,0) (0,2) -- (0,0) (0,0) -- (-1.5,-1.5);
\draw[thick, color=black]
  (2,0.4) node{$\rho^1$}
  (-1.5,-1) node{$\rho^0$}
  (-0.5,1.8) node{$\rho^2$};
\end{tikzpicture}
\caption{The }
\label{fig:fanH1}
\end{figure}

The labeling of the rays in $\Sigma(\CH_1)$ is concordant 
with the notation of \ssect{pic2:fan}. In particular, the closed orbit
$\ko{\orb}(v^2)$ equals the exceptional divisor $E\subseteq\CH_1$.
Since the blow-up $b$ contracts $E$ to 
the point $\ko{\orb(\rho^0,\rho^1)}=0\in\PP^2$, the remaining ray
$\rho^2$ encodes the line $L_\infty\subset\PP^2$ at infinity.
Moreover, the restriction 
$b:\ko{\orb(v^1)}\to\ko{\orb(\rho^2)}=E$ is an isomorphism.

\medskip

Therefore, the exact sequence~\eqref{eq:ExacSeq1} is obtained from 
$$
0\to \CO_{\CH_1} \to \CO_{\CH_1}(E)\oplus
\CO_{\CH_1}(L_\infty)\to\CO_{\CH_1}(E+L_\infty)\to 0.
$$
after replacing the polyhedra with toric Weil divisors and twisting 
by $\CO_{\CH_1}(\kU)=\CO_{\CH_1}(\ko{\orb(u^2)})$. This the 
Koszul complex of the exceptional line and the line at infinity.
}
\end{example}

\medskip

\subsection{Dealing with primitive collections}
\label{subsec:pc}
Next we {apply the formalism of 
Subsection~\eqref{subsec:exSeqpolCov} and} 
fix an arbitrary subset $S\subseteq\Sigma(1)$. 
Let $n=\sharp S$ {and choose} an order
on $S$. {We think} of $S$ as a sequence
$\rho_1,\ldots,\rho_n$ and identify $\rho_i\in S$ 
with $i\in[n]$.

\medskip

For subsets $I\subseteq S$ we define integral
tuples $k_I\in\Z^{\Sigma(1)}$ by
$$
(k_I)_\rho:=\left\{
\begin{array}{ll}
1 & \mbox{if } \rho\in S\setminus I\\
0 & \mbox{if otherwise, i.e., }\rho\in I\cup 
(\Sigma(1)\setminus S).
\end{array}\right.
$$
Interpreting $k_I\in\Z^{\Sigma(1)}$ as a
$T$-invariant, effective and reduced Weil divisor 
on $X=\toric(\Sigma)$ we denote by $\CO_X(k_I)$ 
the associated sheaf. Since $X$ is smooth, these 
divisors are automatically Cartier. In particular,
the sheaves $\CO_X(k_I)$ are invertible albeit not
nef in general.

\begin{proposition}
\label{prop:Resol}
Assume that $S\subseteq\Sigma(1)$ is a non-face, 
that is, $S$ does not define a cone in $\Sigma$.
Then, with $I$ running through the subsets of $S$,
the following complex $\CC_S^\kbb$ of invertible 
sheaves with the usual Koszul-like differentials 
is exact:
$$
0\to\CO_X(k_S)\to 
\oplus_{\sharp I=k-1} \CO_X(k_I) \to\ldots\to
\oplus_{\sharp I=1} \CO_X(k_I) \to \CO_X(k_\varnothing)\to 0.
$$
Here, $\CO_X(k_S)=\CO_X$ and $\CO_X(k_\varnothing)=
\CO_X(\ku{1}_S)$ which is the sheaf associated with 
the effective and reduced divisor 
$\sum_{\rho\in S}D_\rho$.
\end{proposition}

\begin{proof}
We choose a sufficiently ample polytope $\Delta$ 
such that all polytopes
$$
\nabla_{\! I}:=\Delta(k_I):=\{m\in M_{\R}\kst
\langle m,\rho\rangle \geq\min
\langle\Delta,\rho\rangle-k_I\}
$$
are at least nef. It follows immediately that
$$
\nabla_{\! I}\cap\nabla_{\! J}=\Delta(k_I)\cap 
\Delta(k_{J})=\Delta(k_{I\cup J})=\nabla_{\!I\cup J}
$$
for all subsets $I$, $J\subseteq[n]$. In particular, 
$\nabla_{\! I}=\bigcap_{i\in I}\nabla_{\! i}$ if
$I\not=\varnothing$. We assert that 
$\bigcup_{i\in[n]}\nabla_{\! i}=\nabla_{\!\varnothing}$
whence $\{\nabla_{\! i}\}$ is a $\Sigma$-family.
This immediately implies the claim of the proposition 
by tensoring the sequence in Theorem~\ref{thm:S_Seq} 
with $\CO(\Delta)^{-1}$.

\medskip

We claim that for sufficiently ample $\Delta$, 
$\bigcup_{i\in[n]}\Delta(k_i)=\Delta(k_{\varnothing})$ 
if and only if $S$ is a non-face.

\medskip

Let $\bigcup_{i\in[n]}\Delta(k_i)=\Delta(k_{\varnothing})$.
Further, assume that $S\subseteq\sigma(1)$ for some 
(smooth) and full-dimensional cone $\sigma\in\Sigma$ 
with set of rays $\sigma(1)$. The vertex 
$v_{\varnothing}(\sigma)$ of 
$\nabla_{\!\varnothing}=\Delta(k_{\varnothing})$
associated with $\sigma$ satisfies
$$
\langle v_{\varnothing}(\sigma),\rho\rangle=
\left\{\begin{array}{ll}
\min\langle\Delta,\rho\rangle-1 & 
\mbox{if }\rho\in S\subseteq\sigma(1)\\
\min\langle\Delta,\rho\rangle & 
\mbox{if }\rho\in \sigma(1)\setminus S.
\end{array}\right.
$$
However, this contradicts the inequality 
$\langle v_{\varnothing}(\sigma),\rho_i\rangle\geq 
\min\langle\Delta,\rho_i\rangle$
of $\Delta(k_i)$ for $\rho_i\in S$.

\medskip

Conversely, assume that $S\notin\Sigma$.
For each $\rho\in\Sigma(1)$ we consider 
the associated facet 
$$
\face(\Delta,\rho):=\{r\in\Delta\kst \langle r,\rho\rangle
=\min\langle\Delta,\rho\rangle\}
$$ 
and define the ``thickened $\rho$-facet'' 
by
$$
F(\Delta,\rho):=\{r\in\Delta\kst \langle r,
\rho\rangle \leq
\min\langle\Delta,\rho\rangle + 1\}.
$$
More generally, these definitions work 
for all cones 
$\sigma\in\Sigma\setminus\{0\}$:
The face associated to $\sigma$ is
$$
\face(\Delta,\sigma):=
\bigcap_{\rho\in\sigma(1)}\face(\Delta,\rho)
$$
and the corresponding ``thickened 
$\sigma$-face'' is 
$$
F(\Delta,\sigma):=\bigcap_{\rho\in\sigma(1)}
F(\Delta,\rho).
$$
The usual one-to-one correspondence between 
faces of $\Delta$ and cones of
$\Sigma=\normal(\Delta)$ implies that 
$\sigma,\sigma'\in\Sigma\setminus\{0\}$
are not contained in some common cone 
$\til{\sigma}\in\Sigma$ if and only if the
$\face(\Delta,\sigma)$ and 
$\face(\Delta,\sigma')$ are disjoint. 
Therefore, the thickened faces 
$F(\Delta,\sigma)$ and $F(\Delta,\sigma')$ are 
also disjoint for sufficiently ample $\Delta$.

\medskip

Applying this to the polytope 
$\nabla_{\!\varnothing}=\Delta(k_{\varnothing})$
shows that for a non-face $S\subseteq\Sigma(1)$
the corresponding thickened facets are disjoint, 
that is, 
$\bigcap_{\rho\in S}F\big(\nabla_{\!\varnothing},
\rho\big)=\varnothing$. On the other hand, 
$$
\nabla_{\!\varnothing} \;\setminus\;
F\big(\nabla_{\!\varnothing},\rho\big)
\;\subseteq\; \nabla_{\!\rho}
$$
implies
$$
\renewcommand{\arraystretch}{1.3}
\textstyle
\bigcup_{\rho\in S}\nabla_{\!\rho}
\;\supseteq\; \bigcup_{\rho\in S}\big[
\nabla_{\!\varnothing} \setminus 
F(\nabla_{\!\varnothing},\rho)\big]
=  \nabla_{\!\varnothing} \setminus
\bigcap_{\rho\in S} F(\nabla_{\!\varnothing},\rho)
=  \nabla_{\!\varnothing}.
\vspace{-4ex}
$$
\vspace{2ex}
\end{proof}

\begin{remark}\label{rem:PC}
(i) Though we we will not use this observation 
in our later arguments, we note in passing
that for a primitive collection $S$ the 
exact sequence of Proposition~\ref{prop:Resol} 
represents the unique extension arising from 
$\gExt^{n-1}\big(\CO_X(k_\varnothing),\CO_X\big)=\kk$.

\smallskip

{(ii) Proposition~\ref{prop:Resol} 
is the homological counterpart to the multiplicative 
Stanley-Reisner presentation of the equivariant 
K-theory ring of a smooth toric variety, see~\cite{vv}.}
\end{remark}

\medskip

\section{Toric varieties of Picard rank two}
\label{sec:TVPic2}
\subsection{Kleinschmidt's classification}
\label{pic2:KleinSchmidt}
Let $X=\toric(\Sigma)$ be a complete and smooth 
toric variety of dimension $d$ and Picard rank 
two. These 
varieties are characterised by the following 
data~\cite{kle88}:

\begin{enumerate}
\item[(i)]
Natural numbers $\ell_1,\ell_2\geq 2$ 
with $\ell_1+\ell_2=d+2$.
\vspace{2ex}
\item[(ii)]
An integer vector $c\in\Z^{\ell_2}$ with 
nonpositive components 
$$
0=c^1\geq c^2\geq\ldots\geq \cc.
$$
\end{enumerate}
We write the corresponding variety 
$$
{X=(\ell_1,\ell_2;c).}
$$
Here are some key properties.

\subsubsection{The class map}
\label{pic2:classMap}
The varieties $(\ell_1,\ell_2;c)$ arise as 
fibre bundles over $\PP^{\ell_1-1}$ with 
typical fibre $\PP^{\ell_2-1}$. 
{Actually, we have
$$
(\ell_1,\ell_2;c)=\PP\Big(\hspace{-0.5em}
\oplus_{j=1}^{\ell_2}\CO_{\PP^{\ell_1-1}}(-c^j)\Big) \to\PP^{\ell_1-1}.
$$
The best known instance is the Hirzebruch surface
$\CH_{\xa}=\big(2,2;(0,-\xa)\big)$, cf.~Example~\ref{ex:resol}
or \ssect{subsec:ExamL1}.}
The fibration 
is trivial ({\em product case}) if and only if 
$c=0$. Identifying $\Div_T(\ell_1,\ell_2;c)$ 
with $\Z^{\ell_1+\ell_2}$ we can rearrange 
this data in terms of the $2\times(\ell_1 + \ell_2)$-matrix
$$
\pi_c:=\Matrr{3}{4}{1 & \ldots & 1 & 0 &  c^2 &\ldots & \cc\\
0  & \ldots & 0 & 1 & 1 & \ldots & 1}:\Z^{\ell_1+\ell_2}\surj\Z^2.
$$
For $X=(\ell_1,\ell_2;c)$ this provides the map 
$\pi:\mathrm{Div}_T(X)\to\Cl(X)$ in the exact 
sequence~\eqref{eq:MSeq} on page~\pageref{eq:MSeq}. 
In particular, $M\cong\ker\pi_c$ and 
$N\cong\coker\pi^\ast_c$ where 
$\pi^\ast_c:\Z^2\hookrightarrow\Z^{\ell_1+\ell_2}$ 
is the transpose. 

\subsubsection{The fan}
\label{pic2:fan}
Let $\{\eM{1},\ldots,\eM{\ell_1},
\fM{1},\ldots,\fM{\ell_2}\}$ and
$\{\eN{1},\ldots,\eN{\ell_1},
\fN{1},\ldots,\fN{\ell_2}\}$ be the mutually dual bases 
of $\mathrm{Div}_T(\ell_1,\ell_2;c)$
and $\mathrm{Div}_T(\ell_1,\ell_2;c)^\ast$, respectively. 
Under $\rho_c$ from the exact sequence
$$
\xymatrix{
0\ar[r]&\Cl(X)^\ast\ar[r]^-{\pi^\ast_c}&
\Big[\mathrm{Div}_T(X)^\ast
\cong\Z^{\Sigma(1)}\Big]\ar[r]^-{\rho_c}&N\ar[r]&0
}
$$
dual to~\eqref{eq:MSeq} the latter vectors are mapped 
to $\{\uN{1},\ldots,\uN{\ell_1},\vN{1},
\ldots,\vN{\ell_2}\}$. These images define the rays 
of the fan $\Sigma=\Sigma (\ell_1,\ell_2;c)$ 
generated by the $d$-dimensional smooth cones 
$$
\sigma_{ij}:=\langle\Sigma(1)\setminus\{\uN{i},\vN{j}\}
\rangle,\quad i=1,\ldots,\ell_1\mbox{ and }j=1,\ldots,
\ell_2.
$$ 
In particular, $\sharp\Sigma(d)=\ell_1\ell_2$ as 
stated before.
{Note that $\sum_{j=1}^{\ell_2}v^j=0$, but
$\sum_{i=1}^{\ell_1}u^i=\sum_{j=1}^{\ell_2}(-c^j)\cdot v^j$.
}

\subsubsection{The nef divisors}
\label{pic2:nef}
Regarded as a map $\Div_T(X)\to\Cl(X)$, 
$\pi_c$ sends the equivariant prime divisors
$\ko{\orb}(\uN{i})$ and $\ko{\orb}(\vN{j})$ 
to their classes. {The effective cone 
$\Eff\subseteq\Cl(X)$ is generated by 
$[\ko{\orb}(\uN{1})]$ and $[\ko{\orb}(\vN{\ell_2})]$. 
The nef cone $\Nef\subseteq\Eff$ is generated 
by $[\ko{\orb}(\uN{1})]$ and 
$[\ko{\orb}(\vN{1})]$. This provides a natural 
identification {$\Cl(X)\cong\Z^2$} by sending 
$[\ko{\orb}(\uN{1})]$ to $(1,0)$ and $[\ko{\orb}(\vN{1})]$ 
to $(0,1)$. In particular, $[\ko{\orb}(\vN{\ell_2})]$ 
is sent to $(\cc,1)$. These classes are also represented 
by the lattice polytopes} 
$$
\kU:=\Delta(\uN{1})=\conv\{\eM{i}-\eM{1}\mid 
i=1,\ldots,\ell_1\}
$$
and 
$$
\kV:=\Delta(\vN{1})=
\conv\{\fM{j}-\fM{1}-c^j\eM{i}\mid
i=1,\ldots,\ell_1;\, j=1,\ldots, \ell_2\}
$$
in $M_\R\subseteq\R^{\ell_1+\ell_2}$. Note 
that $\kU$ equals the $(\ell_1-1)$-dimensional
standard simplex $\Delta^{\ell_1-1}$ while $\kV$ 
can be understood as the Caley product
$$
\kV=(-c^1\cdot \Delta^{\ell_1-1})*\ldots
*(-\cc\cdot \Delta^{\ell_1-1}).
$$

\subsubsection{The anti-canonical divisor}
\label{pic2:can}
The divisor class of $-K_X$ is 
$$
[-K_X]=\pi_c(1,\ldots,1)=(\ell_1-\xb,\,\ell_2)
$$
with 
$$
\xb:=-\sum_j c^j\geq 0.
$$
Hence, $(\ell_1,\ell_2;c)$ is Fano if and 
only if $\,\xb<\ell_1$.

\begin{remark}
\label{rem:effective}
Rather than the full vector $c$ the non-negative 
parameters
$$
\textstyle
\xa:=-\cc
\hspace{1em}\mbox{and}\hspace{1em}
\xb=-\sum_{j=1}^{\ell_2} c^j
$$
are the most important ones for our purposes.
Almost by definition we get the inequalities 
\begin{equation}\label{eq:EffIneq}
0\leq\xa\leq\xb\leq\xa(\ell_2-1).
\end{equation}
The latter will be referred to as the 
{\em basic inequality}.
\gammaText{\ For further simplification 
we also introduce 
$$
\xc:=\xa\ell_2-\xb
$$
which satisfies $\gamma\geq\alpha\geq 0$.}
\end{remark}

\medskip

\subsection{Application of \ssect{subsec:pc} 
to the situation of Picard rank two}
\label{subsec:appPicTwo}
Next we apply Proposition~\ref{prop:Resol} 
to the toric varieties $X=(\ell_1,\ell_2;c)$. 
From the description of their fan in
\ssect{pic2:fan} we derive that there are 
exactly two primitive collections,
namely
$$
p_u=\{\uN{1},\ldots,\uN{\ell_1}\}
\hspace{0.7em}\mbox{and}\hspace{0.7em}
p_v=\{\vN{1}, \ldots,\vN{\ell_2}\}.
$$

\begin{corollary}
\label{coro:resol}
The primitive collections $p_u$ and $p_v$ give 
rise to the ``$\kU$-sequence'' 
$$
0\to\CO_X\to\ldots\to\CO_X(\kU)^{\otimes\ell_1}\to 0
$$
and the ``$\kV$-sequence''
$$
0 \to \CO_X\to\CF_1\to\ldots \to \CF_{\ell_2-1}
\to\CO_X(\kU)^{\otimes(-\xb)}\otimes
\CO_X(\kV)^{\otimes\ell_2}\to 0
$$
where the sheaves $\CF_k$ are direct sums 
indexed by subsets $I\subseteq\{1,\ldots,\ell_2\}$, 
namely
$$
\textstyle
\CF_k=\oplus_{\sharp I=k\,}\CO_X(c^I\,\kU+k\,\kV)
\hspace{1em}\mbox{with}\hspace{0.5em}
c^I:=\sum_{i\in I}c^i\leq 0.
$$
\end{corollary}

\begin{remark}
Setting $\CF_0:=\CO_X$ and 
$\CF_{\ell_2}:=\CO_X(-\xb\kU+\ell_2\kV)$ 
we can extend this notation to 
$k\in\{0,\ldots,\ell_2\}$. Moreover, 
using the identification $\Pic X\cong\Z^2$ 
with $\kU=(1,0)$ and $\kV=(0,1)$ from 
\ssect{pic2:nef} we usually write 
$\CF_k=\oplus_{\sharp I=k\,}\CO_X(c^I,k)$.
\end{remark}

\begin{proof}
We deal with the primitive collection 
$S=\{\vN{1}, \ldots,\vN{\ell_2}\}$
giving rise to the $\kV$-sequence; the
case of the $\kU$-sequence works similarly.

\medskip

Proposition~\ref{prop:Resol} implies that
we obtain an exact sequence for 
$\CF_k=\oplus_{\sharp J=\ell_2-k\,}\CO_X(k_J)$ 
with $k_J=\ku{1}_{S\setminus J}$.
Renaming $I:=S\setminus J$ this becomes
$\CF_k=\oplus_{\sharp I=k\,}\CO_X(\ku{1}_I)$. 
On the other hand, assigning the Weil 
divisor $\ku{1}_I=\sum_{\rho\in I} D_\rho$ 
to its class means applying the map $\pi_c$ 
from \ssect{pic2:classMap}. Therefore, 
$\ku{1}_I$ becomes $(c^I,k)$ since 
$\pi_c(\fM{i})=(c^i,1)$.
\end{proof}

\begin{example}
\label{ex:resol}
One of the very first examples one comes
across when computing the cohomology of 
invertible sheaves on toric varieties is
$\gExt^1(2\kV,\kU)=\kk$ on the Hirzebruch
surface $X=\CH_1=(2,2;(0,1))$. This is 
represented by the exact sequence from 
Example~\ref{ex:HirzOne}. After twisting 
with $\CO(-\kU)$ this becomes the 
$\kV$-sequence of Corollary~\ref{coro:resol}.
\end{example}

\medskip

\subsection{{The co-immaculate locus}}
\label{subsec:ImmLoc}
From now on we will work with toric 
varieties of the form $X=(\ell_1,\ell_2;c)$ 
and identify $\Pic(X)$ with $\Z^2$ via 
the map $\pi_c$ from 
Subsection~\eqref{pic2:classMap}. 

\medskip

As pointed out in the introduction the 
locus of invertible sheaves with vanishing 
cohomology plays a crucial r\^ole in this 
paper. We call
$$
\nImm(X):=\{\CL\in\Pic(X)\mid
\gH^j(X,\CL^{-1})=0,\,j\in\Z_{\geq0}\}.
$$
the {{\em co-immaculate locus}} of $\Pic(X)$. 
This distinguished subset was referred to 
as the {negative} immaculate locus 
in~\cite{immaculate} and~\cite{p1p1p1}. 
In passing we remark that here and in the sequel we 
shall not distinguish between invertible 
sheaves and their isomorphism classes in 
the Picard group.

\medskip

Rewriting~\cite[Theorem VI.2]{immaculate} 
in terms of
{the co-immaculate} locus, we can describe 
$\nImm(\ell_1,\ell_2;\xa,\xb)=\cH\cup\cP$ 
as the union of the {\em horizontal strip}
$$
\cH:=\{(a,b)\in\Z^2\mid0<b<\ell_2\}
$$
and the {\em parallelogram}
\noGammaText{
\begin{align*}
\cP&\phantom{:}=
\{(a,b)\in\Z^2\mid  -\xb<a<\ell_1
\hspace{0.5em}\mbox{and}\hspace{0.5em}
0<\langle(a,b),\,(1,\xa)\rangle 
< \ell_1-\xb+\xa\ell_2\}.
\end{align*}
}
\gammaText{
\begin{align*}
\cP&\phantom{:}=
\{(a,b)\in\Z^2\mid  -\xb<a<\ell_1
\hspace{0.5em}\mbox{and}\hspace{0.5em}
0<\langle(a,b),\,(1,\xa)\rangle < \ell_1+\xc\}.
\end{align*}}
The {co-immaculate} locus is point symmetric with 
respect to $(\ell_1-\xb,\ell_2)/2$. In 
particular, the origin is point symmetric to 
the anti-canonical class $[-K]=(\ell_1-\xb,\ell_2)$. 
Figure~\ref{fig:ImmEx} 
illustrates the typical shape of the {co-immaculate} 
locus for $\ell_1=7$ and $\ell_2=4$. The 
parallelogram is indicated by the shaded area; 
lattice points in $\cH$ are in grey while lattice 
points in $\cP$ but not in $\cH$ are blue.

\begin{figure}[ht]
\newcommand{\spaceA}{\hspace*{1.5em}}
\newcommand{\scaleA}{0.4}
\begin{tikzpicture}[scale=\scaleA]
\draw[color=oliwkowy!40] (-1.3,-3.3) grid (8.3,7.3);
\fill[pattern color=parallelogram, pattern=north west lines]
  (0,-4.5) -- (0,8.5) -- (7,8.5) -- (7,-4.5) -- cycle;
\draw[thick, color=black]
  (0,0) -- (-2,0) (7,0) -- (9,0)
  (0,4) -- (-2,4) (7,4) -- (9,4);
\draw[thick, dotted, color=black]
  (-2,0) -- (-3,0) (9,0) -- (10,0)
  (-2,4) -- (-3,4) (9,4) -- (10,4);  
\draw[thick, color=black]
  (0,4) -- (0,7.5) (7,4) -- (7,7.5)
  (0,0) -- (0,-3.5) (7,0) -- (7,-3.5);
\draw[thick, dotted, color=black]
  (0,7.5) -- (0,8.5) (7,7.5) -- (7,8.5)
  (0,-3.5) -- (0,-4.5) (7,-3.5) -- (7,-4.5);
\fill[thick, color=origin]
  (0,0) circle (5pt);
\draw[thick, color=origin]
  (0,0) circle (7pt);
\draw[thick, color=origin]
  (-0.5,-0.5) node{$\scriptstyle0$};
\fill[thick, color=origin]
  (7,4) circle (5pt);
\draw[thick, color=origin]
  (7,4) circle (7pt);
\draw[thick, color=origin]
  (9.0,4.7) node{$\scriptstyle-[K]=(\ell_1,\,\ell_2)$};
\foreach \x in {1,2,...,6} 
{
    \foreach \y in {-3,-2,...,7}
    {
        \fill[thick, color=regii]   
        (\x,\y) circle (5pt);
    }
}
\foreach \y in {1,2,...,3} 
    \foreach \x in {9,8,...,-2} 
        \foreach \y in {1,2,...,3} 
        {
        \fill[thick, color=regi]
        (\x,\y) circle (5pt); 
        }
\draw[thick, color=black]
    (10.5,1) node{$\scriptstyle1$} 
    (10.5,3) node{$\scriptstyle\ell_2-1$}
    (10.5,2.2) node{$\scriptstyle\vdots$}
    (10.5,0) node{$\scriptstyle0$} 
    (10.5,4) node{$\scriptstyle\ell_2$};
\draw[thick, color=black]
    (0.0,9.0) node{$\scriptstyle0$}
    (7.2,9.0) node{$\scriptstyle\ell_1$} 
    (5.8,9.0) node{$\scriptstyle\ell_1-1$}
    (3.0,9.0) node{$\scriptstyle\ldots$} 
    (1.0,9.0) node{$\scriptstyle1$};
\draw[thick, color=black]
  (3.5,2) circle (7pt);
\fill[thick, color=black]
  (3.5,2) circle (5pt);
\draw[thick, color=black]
  (4.0,1.5) node{$\scriptstyle-K/2$};
\end{tikzpicture}
\spaceA
\begin{tikzpicture}[scale=\scaleA]
\draw[color=oliwkowy!40] (-3.3,-4.3) grid (7.3,8.3);
\fill[pattern color=parallelogram, pattern=north west lines]
  (-3,1.5) -- (-3,7.5) -- (7,2.5) -- (7,-3.5) -- cycle;
\draw[thick, color=black]
  (0,0) -- (7,-3.5) -- (7,0);
\fill[thick, color=origin]
  (0,0) circle (5pt);
\draw[thick, color=origin]
  (-0.5,-0.5) node{$\scriptstyle0$};
\draw[thick, color=black]
  (2,2) circle (9pt);
\draw[thick, color=black]
  (2.5,1.5) node{$\scriptstyle-K/2$};
\foreach \x in {1,2,...,6} {
  \fill[thick, color=regii]
  (\x,0) circle (5pt); }
\foreach \x in {5,6} {
  \fill[thick, color=regii]
    (\x,-2) circle (5pt); }
\foreach \x in {3,4,...,6} {
  \fill[thick, color=regii]
    (\x,-1) circle (5pt); }
\foreach \x in {10,9,...,-6} 
\foreach \y in {1,2,...,3} {
  \fill[thick, color=regi]
    (\x,\y) circle (5pt); }
\draw[thick, color=black]
  (11.5,1) node{$\scriptstyle1$} 
  (11.5,3) node{$\scriptstyle\ell_2-1$}
  (11.5,2.2) node{$\scriptstyle\vdots$}
  (11.5,0) node{$\scriptstyle0$} 
  (11.5,4) node{$\scriptstyle\ell_2$};
\draw[thick, color=black]
  (4,4) -- (-3,7.5) -- (-3,4);
\fill[thick, color=origin]
  (4,4) circle (5pt);
\draw[thick, color=origin]
  (6.0,4.5) node{$\scriptstyle-[K]=(\ell_1-\xb,\,\ell_2)$};
\draw[thick, color=black]
  (0.0,6.5) node{$\scriptstyle0$}
  (7.0,6.5) node{$\scriptstyle\ell_1$} 
  (4.0,6.5) node{$\scriptstyle\ldots$} 
  (1.0,6.5) node{$\scriptstyle1$};
\foreach \x in {3,2,...,-2} {
  \fill[thick, color=regii]
  (\x,4) circle (5pt); }
\foreach \x in {1,0,...,-2} {
  \fill[thick, color=regii]
    (\x,5) circle (5pt); }
\foreach \x in {-1,-2} {
  \fill[thick, color=regii]
    (\x,6) circle (5pt); }
\end{tikzpicture}
\caption{The {co-immaculate} locus of 
$(7,4;0)$ (left hand side) and
$\big(7,4;(0,0,-1,-2)\big)$ (right hand side).}
\label{fig:ImmEx}
\end{figure}

\begin{remark}\label{rem:xaxb}
\hfill
\begin{enumerate}
\item[(i)] The sequel of this article 
entirely depends on the combinatorics 
of the {co-immaculate} locus, not the 
underlying fan. The former depends 
solely on $\ell_1$, $\ell_2$, $\xa=-\cc$ 
and $\xb=-\sum_jc^j$. The {co-immaculate} 
locus will be therefore written as 
$\nImm(\ell_1,\ell_2;\xa,\xb)$ (or 
simply $\nImm$ depending on the context). 
\vspace{1ex}
\item[(ii)] The {co-immaculate} locus is 
{\em horizontally integral convex}, that 
is, if $s$, $t\in\nImm\cap\ell$ for any 
horizontal line $\ell\subseteq\Z^2$, then 
the segment $[s,t]\subseteq\ell$ is 
contained in $\nImm$, too.
\vspace{1ex}
\item[(iii)] Since $[-K]=(\ell_1-\xb,\ell_2)$ 
the anti-canonical class sits always to 
the left of the line $[x=\ell_1+1]$. By 
\ssect{pic2:can} it sits to the right of 
$[x=0]$ if and only if $(\ell_1,\ell_2;c)$ 
is Fano.
\end{enumerate}
\end{remark}

\medskip

\subsection{The associated lattice}
\label{subsec:AssLat}
The boundary points $(\ell_1,0)$ and 
$(-\xb,\ell_2)$ of the {co-immaculate} locus
$\nImm(\ell_1,\ell_2;\xa,\xb)$, 
define the {\em associated lattice}
$$
L:=\Z(\ell_1,0)\oplus\Z(-\xb,\ell_2).
$$
We denote $\CT=\Z^2/L$ the induced quotient 
and $\Phi:\Z^2\to\CT$ the projection map 
which sends $(a,b)$ to its class $[a,b]$ in $\CT$,
cf.\ also~\cite[Section 4]{p1p1p1}.

\begin{lemma}
\label{lem:LImm}
We have $L\cap\nImm(\ell_1,\ell_2;\xa,\xb)=\varnothing$. 
\end{lemma}

\begin{proof}
This is obvious for $c=0$ so assume that $c\not=0$.

\medskip

Suppose we could pick $(a,b)\in L\cap\nImm$ so that 
$(a,b)=(n\ell_1-m\xb,m\ell_2)\in\nImm$ for some $n$, 
$m\in\Z$. In particular, we necessarily have 
$(a,b)\in\cP$. If $m\leq0$, 
{then we exploit the inequalities
$\langle (a,b),(1,0)\rangle <\ell_1$ and
$0<\langle (a,b),(1,\xa)\rangle$; they imply
}
$$
0\leq -m\xa\ell_2<n\ell_1-m\xb<\ell_1
$$ 
and therefore, {by adding $m\xb$,} 
$$
{m\xb\leq
\big[}\gammaText{-m\xc=}
m(\xb-\xa\ell_2)\big]<n\ell_1<\ell_1 +m\xb \leq \ell_1.
$$
By the basic inequality~\eqref{eq:EffIneq} 
on page~\pageref{eq:EffIneq} we have {even}
\noGammaText{$0\leq m(\xb-\xa\ell_2)$}%
\gammaText{$0\leq -m\xc$}
whence a contradiction for both cases $n\geq 1$ 
and $n\leq 0$. 
{On the other hand, for $m\geq1$ the 
remaining two inequalities of $\cP$, namely
$-\xb<\langle (a,b),(1,0)\rangle$ and
$\langle (a,b),(1,\xa)\rangle<\ell_1+\xc$ imply
}
\noGammaText{
$$
-\xb<n\ell_1-m\xb<\ell_1-\xb-(m-1)\xa\ell_2
$$
and therefore 
$$
0\leq(m-1)\xb<n\ell_1<\ell_1+(m-1)(\xb-\ell_2\xa)
\leq\ell_1.
$$
}%
\gammaText{
$$
0\leq(m-1)\xb<n\ell_1<\ell_1-(m-1)\xc\leq\ell_1.
$$
}%
{Again, this leads to a contradiction.}
\end{proof}

\medskip

\section{Exceptional sequences of invertible sheaves}
\label{sec:ExcSeq}
\subsection{The exceptionality condition}
\label{subsubsec:ExcCon}
Recall from the introduction that a 
sequence $\CL_0,\ldots,\CL_N$ of 
invertible sheaves on a variety $X$ 
is said to be {\em exceptional} if 
all backward $\mathrm{Ext}$-groups 
vanish. Equivalently, 
$$
\gH^k(X,\CL_i\otimes\CL_j^{-1})=0\mbox{ if }i<j.
$$
Consequently, we can 
rephrase the exceptionality condition 
as follows. Denoting the isomorphism 
classes in $\Pic(X)$ by $\cl_i:=\CL_i$, 
the sequence
$\cl_0,\cl_1,\ldots,\cl_N\in\Pic(X)$ is 
exceptional on $X$ if and only if 
$$
\vect{\cl_i\,\cl_j}=\cl_j-\cl_i\in\nImm(X)
$$
for all $i<j$, or equivalently,
$$
\cl_j\in\bigcap_{i<j} \big(\cl_i+\nImm(X)\big)
\quad\mbox{for all }j\geq 1.
$$

\medskip

This condition persists under a 
simultaneous shift so that we 
may replace the original sequence by 
$\CL_i':=\CL_i\otimes\CL_0^{-1}$. 
Therefore, we may assume that 
$\CL_0=\CO_X$ is trivial whenever this 
is convenient. In particular, $\cl_0=0$ 
which implies $\cl_i\in\nImm(X)$ for 
all $i\geq 1$.

\begin{example}\label{ex:PnP1}
For $X=\PP^{\ell-1}$, Serre duality and 
the well-known vanishing theorems for 
invertible sheaves on projective space 
yield 
$$
\nImm(\PP^{\ell-1})=\{\CO_{\PP^{\ell-1}}(1),\ldots,
\CO_{\PP^{\ell-1}}(\ell-1)\}
\entspr\,\{1,2,\ldots,\ell-1\}\subseteq\Z.
$$
\end{example}

\begin{example}\label{ex:P1P1}
The previous example behaves well under 
products. Figure~\ref{fig:ImmP2P2} 
visualises the case of $X=\PP^1\times\PP^1$. 
Let us determine on $\PP^1\times\PP^1$ all 
possible types of exceptional sequences of 
maximal length which is $4$. The shape of 
the {co-immaculate} locus immediately implies 
that we have at most two elements on the 
lines $[x=1]$ and $[y=1]$, and any two 
elements on the same line must be 
consecutive. Computing the sets 
$\nImm\cap\big((a,1)+\nImm\big)$ and 
$\nImm\cap\big((1,b)+\nImm\big)$ for 
$a$, $b\in\Z$ shows that we have four 
families $\cl^1,\ldots,\cl^4$ of {\mes}{s}
given by
$$
\begin{array}{llll}
\cl^1&=\big((0,0),\,(1,0),\,(a,1),\,(a+1,1)\big)&
\cl^2&=\big((0,0),\,(0,1),\,(1,b),\,(1,b+1)\big)\\[5pt]
\cl^3&=\big((0,0),\,(a,1),\,(a+1,1),\,(1,2)\big)&
\cl^4&=\big((0,0),\,(1,b),\,(1,b+1),\,(2,1)\big).
\end{array}
$$
Examples are displayed in Figure~\ref{fig:fourEx}.

\begin{figure}[ht]
\begin{tikzpicture}[scale=0.35]
\draw[color=oliwkowy!40] (-4.3,-4.3) grid (4.3,4.3);
\fill[thick, color=origin]
  (0,0) circle (7pt);
\draw[thick, color=origin]
  (-0.5,-0.5) node{$0$};
\foreach \x in {-4,...,4} 
{
    \fill[thick, color=regi]
    (\x,1) circle (7pt);
}
\draw[thick, color=regi]
  (5,1) node{$\ldots$};
\draw[thick, color=regi]
  (-5,1) node{$\ldots$};
\foreach \y in {-4,...,4}
{
    \fill[thick, color=regi]
    (1,\y) circle (7pt);
}
\draw[thick, color=regi]
  (1,5) node{$\vdots$};
\draw[thick, color=regi]
  (1,-4.7) node{$\vdots$};
\end{tikzpicture}
\caption{The {co-immaculate} locus of
$\PP^1\times\PP^1$ is given by the 
grey points. Note that the origin 
does not belong to
$\nImm(\PP^1\times\PP^1)$.}
\label{fig:ImmP2P2}
\end{figure}
\end{example}

\begin{figure}[ht]
\newcommand{\spaceA}{\hspace*{2.5em}}
\newcommand{\scaleA}{0.3}
\begin{tikzpicture}[scale=\scaleA]
\draw[color=oliwkowy!40] (-4.3,-4.3) grid (4.3,4.3);
\fill[thick, color=origin]
  (0,0) circle (9pt);
\foreach \x in {-4,...,4}
{
    \draw[thick, color=regi]
    (\x,1) circle (5pt);
}
\foreach \y in {-4,...,4}
{
    \draw[thick, color=regi]
    (1,\y) circle (5pt);
}
\fill[thick, color=black]
  (-3,1) circle (9pt) (-2,1) circle (9pt) 
  (1,0) circle (9pt);
\end{tikzpicture}
\spaceA
\begin{tikzpicture}[scale=\scaleA]
\draw[color=oliwkowy!40] (-4.3,-4.3) grid (4.3,4.3);
\fill[thick, color=origin]
  (0,0) circle (9pt);
\foreach \x in {-4,...,4}
{
    \draw[thick, color=regi]
    (\x,1) circle (5pt);
}
\foreach \y in {-4,...,4}
{
    \draw[thick, color=regi]
    (1,\y) circle (5pt);
}
\fill[thick, color=black]
  (1,-3) circle (9pt) (1,-2) circle (9pt) 
  (0,1) circle (9pt);
\end{tikzpicture}
\spaceA
\begin{tikzpicture}[scale=\scaleA]
\draw[color=oliwkowy!40] (-4.3,-4.3) grid (4.3,4.3);
\fill[thick, color=origin]
  (0,0) circle (9pt);
\foreach \x in {-4,...,4}
{
    \draw[thick, color=regi]
    (\x,1) circle (5pt);
}
\foreach \y in {-4,...,4}
{
    \draw[thick, color=regi]
    (1,\y) circle (5pt);
}
\fill[thick, color=black]
  (-3,1) circle (9pt) (-2,1) circle (9pt) 
  (1,2) circle (9pt);
\end{tikzpicture}
\spaceA
\begin{tikzpicture}[scale=\scaleA]
\draw[color=oliwkowy!40] (-4.3,-4.3) grid (4.3,4.3);
\fill[thick, color=origin]
  (0,0) circle (9pt);
\foreach \x in {-4,...,4}
{
    \draw[thick, color=regi]
    (\x,1) circle (5pt);
}
\foreach \y in {-4,...,4}
{
    \draw[thick, color=regi]
    (1,\y) circle (5pt);
}
\fill[thick, color=black]
  (1,-3) circle (9pt) (1,-2) circle (9pt) 
  (2,1) circle (9pt);
\end{tikzpicture}
\caption{$\cl^1$, $\cl^2$, $\cl^3$, $\cl^4$ 
for $a=b=-3$.}
\label{fig:fourEx}
\end{figure}

\medskip

\subsection{The maximality condition}
\label{subsec:MaxCond}
Using the identification 
$\Pic(\toric(\Sigma))\cong\Z^2$
we think of exceptional sequences 
as subsets of $\Z^2$ together 
with a certain ordering. 

\begin{definition}
\label{def:MaxExc}
Consider the smooth toric 
variety of Picard rank two 
$(\ell_1,\ell_2;c)$ which
is of dimension $d:=\dim X=\ell_1+\ell_2-2$,
cf.~\ssect{pic2:KleinSchmidt}. The 
subset inside $\Z^2$ underlying 
an exceptional sequence 
$\cl=(\cl_0,\cl_1,\ldots,\cl_N)$ 
is called
\begin{enumerate}
\item[(i)] 
{\em \nex} if it is not strictly contained 
in some other subset of $\Z^2$ underlying 
an exceptional sequence.
\vspace{1ex}
\item[(ii)] 
{\em \tm} if $N+1$ equals 
$\sharp\Sigma(d)=\ell_1\ell_2$.
\end{enumerate}
By abuse of language, we refer to the 
sequence $\cl$ itself as {\nex} or {\tm} 
if the underlying set is {\nex} or {\tm}. 
\end{definition}

Note that ``\tm'' implies ``\nex'', 
the converse being false, see 
Example~\ref{exam:NexNmax}. A crucial 
property of maximal exceptional sequences 
is this. 

\begin{lemma}
\label{lem:PhiMap}
If $\cl$ is a maximal exceptional sequence, 
then the restriction of $\Phi:\Z^2\to\CT$ 
from \ssect{subsec:AssLat}
defines a bijection.
\end{lemma}

\begin{proof}
If $\Phi(\cl_i)=\Phi(\cl_j)$ for some pair 
$\cl_i<\cl_j$ in $\cl$, then $\cl_j-\cl_i
\in\nImm(\ell_1,\ell_2;\xa,\xb)\cap L$ which 
is empty by Lemma~\ref{lem:LImm}. Hence 
$\Phi|_\cl$ is injective.

\medskip

Since $\cl$ and $\CT$ have the same cardinality, 
$\Phi|_\cl$ is also surjective.
\end{proof}

\medskip

\subsection{The fullness condition}
\label{subsec:FullCond}
Recall from the introduction that an 
exceptional sequence $\cl$ is said to 
be {\em full} if the underlying set 
$\{\cl_0,\ldots,\cl_N\}$ generates 
$\CD(X)$. Since $X$ is smooth, 
$\CD(X)$ is generated by $\Pic(X)$. 
Regarding $\cl$ as a subset of $\Pic(X)$ 
it is therefore sufficient to show that 
we can generate any invertible sheaf.

\begin{example}\label{ex:Beilinson}
The celebrated Beilinson exact 
sequence~\cite{beilinson} yields
\begin{equation}\label{eq:BeilSeq}
0 \to \Lambda^\ell\kk^{\ell}\otimes
\CO_{\PP^{\ell-1}}(0)\to
\Lambda^{\ell-1}\kk^{\ell}\otimes
\CO_{\PP^{\ell-1}}(1)\to\ldots\to
\Lambda^{0}\kk^{\ell}\otimes
\CO_{\PP^{\ell-1}}(\ell)\to 0.
\end{equation}
In our language, this is just the 
exact sequence from 
Theorem~\ref{thm:S_Seq} for the 
standard $(\ell-1)$-simplex 
in $\R^{\ell}$. 

\medskip

At any rate, given the sequence 
$\CO_{\PP^{\ell-1}}(1),\ldots,
\CO_{\PP^{\ell-1}}(\ell-1)$ 
the bundle $\CO_{\PP^{\ell-1}}(0)$ 
generates $\CO_{\PP^{\ell-1}}(\ell)$ 
in $\CD(\PP^{\ell-1})$ and vice versa. 
Thus, any sequence of $\ell$ consecutive 
classes of invertible sheaves generates 
$\Pic(\PP^{\ell-1})\cong\Z$ and thus 
the derived category $\CD(\PP^{\ell-1})$. 
Note that any such sequence actually 
defines a maximal exceptional 
sequence.
\end{example}

\medskip

\subsection{The helix operator}
\label{subsec:helixOp}
Let $\cl=(\cl_0,\cl_1,\ldots,\cl_N)$ be a 
{\mes} on $(\ell_1,\ell_2;c)$. We define
$$  
\helix(\cl):=(\cl_1,\cl_2,\ldots,\cl_N,
-[K]+\cl_0)=(\cl_1,\cl_2,\ldots,\cl_N,\,
(\ell_1-\xb,\ell_2)+\cl_0)
$$  
and call $\helix$ the {\em helix operator}. 
It preserves exceptionality as follows 
directly from the point symmetry of $\nImm$ 
with respect to $-[K]/2$, cf.\ 
\ssect{subsec:ImmLoc}. 

\begin{remark}
\label{rem:HelixOperator}
{Helixing is a standard operation in 
the theory of exceptional sequences, cf.\ for 
instance~\cite{rudakov}, and it is well-known 
that the helix operator also preserves fullness. 
From our combinatorial point of view, this follows 
as a corollary to Theorem~{\thFull}; assuming the 
general theory, Theorem~{\thFull} is a simple 
corollary of Theorem~{\thHelex}.}
\end{remark}

\medskip

\section{Lexicographical order and chains}
\label{sec:Relax}
\subsection{Lexicographically orderable exceptional sequences}
\label{subsec:CriteriaRelax}

Let $\cl=(\cl_0,\ldots,\cl_N)$ 
be an exceptional sequence on 
$X=(\ell_1,\ell_2;c)$, that is, 
$\cl_j-\cl_i\in\nImm=
\nImm(\ell_1,\ell_2;\xa,\xb)$ 
for all $0\leq i<j\leq N$. The 
underlying set $\{\cl_0,\ldots,\cl_N\}$ 
can give rise to exceptional 
sequences for various orders.
Still, we necessarily have the

\begin{lemma}
\label{lem:HorOrd}
Let $\cl=(\cl_0,\ldots,\cl_N)$ 
be an exceptional sequence on 
$(\ell_1,\ell_2;c)$. Then for 
any $\cl_i=(a_i,b)$ and $\cl_j=(a_j,b)$ 
in $\cl\cap[y=b]$ we have $a_i<a_j$ 
if and only if $i<j$. Moreover, 
$|a_j-a_i|\leq\ell_1-1$.
\end{lemma}

\begin{proof}
By definition of exceptionality, 
$i<j$ implies
$$
\cl_j-\cl_i=(a_j-a_i,\,0)\in\nImm
$$
which holds if and only if 
$0<a_j-a_i<\ell_1$. 
\end{proof}

\begin{definition}
The subset  
underlying an exceptional sequence 
$\cl$ is called {\em vertically orderable} 
if it also defines an exceptional 
sequence with respect to the {\em vertical 
(lexicographical) order} 
\begin{center}
$(a_1,b_1)<(a_2,b_2)$ if and only if 
$b_1<b_2$ or $\big(b_1=b_2$ and $a_1<a_2\big)$.
\end{center}
Similarly, it is called {\em horizontally 
orderable} if it defines an exceptional 
sequence with respect to the 
{\em horizontal (lexicographical) order}
\begin{center}
$(a_1,b_1)<(a_2,b_2)$ if and only if 
$a_1<a_2$ or $\big(a_1=a_2$ and 
$b_1<b_2\big)$.
\end{center}
\end{definition}

In general, exceptionality cannot 
be expected to be preserved under 
lexicographic reordering. For 
instance, the sequence 
$\cl=(\cl_0,\,\cl_1,\,\cl_2)$ on 
$(\ell_1,\ell_2;c)$ with $c\not=0$ 
and $\xa\ell_2+1<\ell_1$ 
which is given by
$$
\cl_0=(0,0),\quad\cl_1=(\xa\ell_2+1,
-\ell_2),
\quad\cl_2= (\xa\ell_2+2-\xb,0)
\gammaText{\;=(\xc+2,0)}
$$ 
is certainly exceptional, but 
neither the horizontally ordered 
sequence $(\cl_0,\cl_2,\cl_1)$ nor 
the vertically ordered sequence 
$(\cl_1,\cl_0,\cl_2)$ are as
$\xb<\xa\ell_2$ in virtue of the
basic inequality~\eqref{eq:EffIneq}. 
However, these ``downward jumps'' are 
the only obstruction against vertical 
order. 

\begin{proposition}
\label{prop:RelCrit}
An exceptional sequence 
$\cl=(\cl_0,\ldots,\cl_N)$ on
$(\ell_1,\ell_2;c)$ with 
$\cl_i=(a_i,b_i)$ is vertically 
orderable if and only if 
\begin{equation}\label{eq:RelBound}
b_j-b_i>-\ell_2
\end{equation}
for all $0\leq i<j\leq N$. 
\end{proposition}

\begin{proof}
Assume that $\cl$ is vertically 
orderable and that we have a pair 
$\cl_i<\cl_j$ with
$$
-\ell_2\geq b:=b_j-b_i.
$$
In particular, 
{$\cl_j-\cl_i$ must lie in the parallelogram $\cP$, hence}
$$
-\xa b<a:=a_j-a_i.
$$
On the other hand, we also have 
$\cl_i-\cl_j\in\nImm$. Since 
$-b\geq\ell_2$, the difference 
$\cl_i-\cl_j=(-a,-b)$ must be  
in the parallelogram $\cP$ {as well whence} $-\xb<-a$. By the basic inequality~\eqref{eq:EffIneq}, 
$a>-\xa b>\xb$ {and thus} $-a<-\xb$, 
contradiction.

\medskip

Conversely, assume that the 
bound~\eqref{eq:RelBound} holds. 
We show that a permutation 
$\sigma:\{0,\ldots,N\}\to\{0,\ldots,N\}$ {exists} 
such that 
$\cl_{\sigma(0)},\ldots,\cl_{\sigma(N)}$ 
is an exceptional sequence with respect 
to lexicographical order.

\medskip

From Lemma~\ref{lem:HorOrd}, $i<j$ 
implies $a_i<a_j$ if $\cl_i=(a_i,b)$ 
and $\cl_j=(a_j,b)$. In addition, we 
show that we can interchange the order 
of two adjacent sequence elements 
$\cl_i=(a_i,b_i)$ and 
$\cl_{i+1}=(a_{i+1},b_{i+1})$ with 
$b_i>b_{i+1}$ while keeping the 
sequence exceptional.

\medskip

Indeed, the difference 
$\vect{\cl_i\,\cl_{i+1}}=(a_{i+1}-a_i,\; b_{i+1}-b_i)$ 
lies in $\nImm$. The assumptions 
imply
$$
0<b_i-b_{i+1}\leq\ell_2-1.
$$
Hence 
$\vect{\cl_{i+1}\,\cl_i}=-\vect{\cl_i\,\cl_{i+1}}$ 
belongs to the horizontal strip 
and therefore lies in $\nImm$, too.

\medskip

Replacing in the sequence the 
pair $\cl_i,\cl_{i+1}$ by $\cl_{i+1},\cl_i$ 
does not affect the remaining difference 
vectors of $\{\cl_i,\cl_{i+1}\}$ with 
elements from $\{\cl_0,\ldots,\cl_{i-1}\}$ 
and $\{\cl_{i+2},\ldots,\cl_N\}$ so that 
$\cl_0,\ldots,\cl_{i-1},\cl_{i+1},\cl_i,
\cl_{i+2},\ldots,\cl_N$ is still exceptional. 
Furthermore, it still satisfies~\eqref{eq:RelBound} 
since $b_{i+1}-b_i\leq0$. After a finite 
number of pairwise permutations which 
preserve~\eqref{eq:RelBound} and exceptionality, 
we obtain the desired permutation.
\end{proof}

\begin{corollary}[``no horizontal gaps'']
\label{coro:NoGaps}
Let $\cl=(\cl_0,\cl_1,\ldots,\cl_N)$ define an
exceptional sequence which is not extendable. 
If $\cl$ is vertically orderable, 
then the restriction $\cl\cap[y=b]$ to any horizontal 
line has no gaps. Namely, if $(a,b)$ and $(a',b)\in\cl$ 
with $a<a'$, then $(a'',b)\in\cl\cap[y=b]$ for all 
$a\leq a''\leq a'$.
\end{corollary}

\begin{proof}
Reordering if necessary we may assume 
that $\cl$ is vertically ordered. 
Suppose we have $\cl_i=(a_i,b)$, 
$\cl_{i+1}=(a_{i+1},b)$ with $a_i+2\leq a_{i+1}$. 
Then we can choose a point $(a,b)\in\Z^2$ and $a_i<a<a_j$.
We define a new sequence $\cl'$ by
$$
\cl'_k:=\left\{\begin{array}{ll}
\cl_k & \mbox{if }k\leq i\\
(a,b) & \mbox{if }k=i+1\\
\cl_{k-1} & \mbox{if }k\geq i+2
\end{array}\right..
$$

\medskip

We show that $\cl'$ is exceptional, contradicting the 
nonextendability of $\cl$. Consider the vectors 
$\vect{\cl'_k\cl'_{k+r}}=\cl'_{k+r}-\cl'_k$ 
for $k\geq 1$. If none of the indices $k$ or $k+r$ 
equals $i+1$, then $\cl'_k$ and $\cl'_{k+r}$ belong to 
the original sequence $\cl$, whence 
$\vect{\cl'_k\cl'_{k+r}}\in\nImm$.

\medskip

If $k\leq i$, consider $\cl'_{i+1}-\cl'_k=(a,b)-\cl_k$ 
which sits between $\cl_i-\cl_k$ and $\cl_{i+1}-\cl_k$ 
on the same horizontal line. Hence $(a,b)-\cl_k\in\nImm$ 
by horizontal integral convexity, cf.\ 
\ssect{subsec:ImmLoc}.

\medskip

Finally, if $k\geq i+1$, we can reason as before
and conclude that the difference $\cl'_k-\cl'_{i+1}$ 
belongs to $\nImm$.  
\end{proof}

{Since by Proposition~\ref{prop:RelCrit} any 
pair of adjacent elements with $b_i-b_{i+1}<\ell_2$ can be switched we 
obtain the following}

\begin{lemma}
\label{lem:ReorderJump}
Let $\cl=(\cl_0,\ldots,\cl_N)$ be an 
exceptional sequence on $(\ell_1,\ell_2;c)$ 
which is not vertically orderable. Then we can 
reorder $\cl$ in such a way that the new 
sequence is still exceptional and there is a 
consecutive pair $\cl_i<\cl_{i+1}$ with 
$b_i-b_{i+1}\geq\ell_2$. 
\end{lemma}

\begin{remark}
\label{rem:RelProdCase}
For $c=0$ the $\Z^2$-involution $(a,b)\mapsto(b,a)$ 
maps any $(\ell_1,\ell_2;0)$-exceptional sequence 
to an $(\ell_2,\ell_1;0)$-exceptional sequence. In 
particular, it follows that $\cl$ is horizontally 
orderable if and only if
\begin{equation}\label{eq:RelBound2}
a_j-a_i>-\ell_1
\end{equation}
for all $0\leq i<j\leq N$. Similarly,
Lemma~\ref{lem:HorOrd},~\ref{lem:ReorderJump}
and Corollary~\ref{coro:NoGaps}
hold mutatis mutandis.
\end{remark}

\medskip

\subsection{Orderable varieties}
\label{subsec:RelVar}
For $c\not=0$ we define the {\em integral depth} 
of $\nImm$ as the smallest integer $\hZ$ such that 
the line $[y=\hZ]$ meets $\nImm$, namely
\label{page:IntDepth}
\begin{equation}\label{eq:IntDepth}
\hZ=-\Big\lfloor{\frac{\ell_1-2}{\xa}}\Big\rfloor. 
\end{equation}
Proposition~\ref{prop:RelCrit} immediately 
implies the

\begin{corollary}\label{coro:RelVar}
If $\hZ\geq1-\ell_2$, then any exceptional 
sequence on $(\ell_1,\ell_2;c)$ with $c\not=0$ 
can be vertically ordered. 
\end{corollary}

We call the toric variety $(\ell_1,\ell_2;c)$ 
itself {\em vertically orderable} if every 
exceptional sequence is vertically orderable. 
\ssect{pic2:can} and the basic inequality~\eqref{eq:EffIneq} 
immediately imply the
 
\begin{proposition}
\label{prop:RelVar}
A toric variety $(\ell_1,\ell_2;c)$ with 
$c\not=0$ and $\xb\geq\ell_1-2$ is vertically 
orderable. 
\end{proposition}

This holds, for instance, for varieties of 
dimension less than or equal to three 
as well as for all non-Fano varieties. 

\medskip

\subsection{Existence}
\label{subsec:ExChain}
To generalise the generation strategy 
from \ssect{subsec:GenDerCat} we make 
the following

\begin{definition}
\label{def:HorChain}
A {\em horizontal chain} in $\Z^2$ is 
a subset of the form 
$$
(a,b)+\{(0,0),\ldots,(\ell_1-1,0)\}.
$$
Similarly, a {\em vertical chain} is 
a subset of the form 
$(a,b)+\{(0,0),\ldots,(0,\ell_2-1)\}$.
\end{definition}

\begin{remark}
Thinking of $\Z^2$ as the Picard group 
$\Pic(\ell_1,\ell_2;c)$ we can fill, that is, 
generate in the derived category any 
line $[y=b]$ which contains a horizontal 
chain via the Beilinson 
sequence~\eqref{eq:BeilSeq} 
in \ssect{subsec:FullCond}. The product
case also requires vertical chains.
\end{remark}

Given an exceptional sequence $\cl$ we 
quantify its ``spatial size'' as follows. 
We let 
$$
\underline a:= \underline a(\cl):=
\min\{a\mid (a,b)\in\cl\},\quad
\underline b:=\underline b(\cl):=
\min\{b\mid(a,b)\in\cl\}
$$
and
$$
\overline a:=\overline a(\cl):=
\max\{a\mid (a,b)\in\cl\},\quad
\overline b:=\overline b(\cl):=
\max\{b\mid(a,b)\in\cl\},
$$
The {\em height} and the {\em width} 
of $\cl$ are then defined by 
$$
H(\cl):=\overline b-\underline b+1
\quad\mbox{and}\quad W(\cl):=
\overline a-\underline a+1.
$$

\begin{proposition}
\label{prop:CompShape}
Let $\cl$ be a {\mes} on
$(\ell_1,\ell_2;c)$ and
$H(\cl)\leq2\ell_2$. For all integers $b$
with $0\leq b-\underline b<\ell_2$, the sets
$$
Y_{b+\ell_2}:=(-\xb,\ell_2)+([y=b]\cap\cl)
\hspace{0.8em}\mbox{and}\hspace{0.8em}
X_{b+\ell_2}:=[y=b+\ell_2]\cap\cl
$$
define the horizontal chain 
$$
S_{b+\ell_2}:=Y_{b+\ell_2}\cup X_{b+\ell_2}\subseteq\Z^2.
$$
In particular, $Y_{b+\ell_2}\cap X_{b+\ell_2}=\varnothing$ 
and $\sharp S_{b+\ell_2}=\ell_1$.
\end{proposition}

\begin{proof}
Consider the lattice $L$ with associated 
map $\Phi:\Z^2\to\CT$ from \ssect{subsec:AssLat}. By 
Lemma~\ref{lem:PhiMap} its restriction 
to $\cl$ is bijective. 
We denote by 
$[a,b]=\Phi(a,b)$ the equivalence class 
of $(a,b)\in\cl$ in $\CT$. 

\medskip

To ease notation we assume that 
$\underline b=0$. Since $H(\cl)\leq2\ell_2$, 
$0\leq b\leq2\ell_2-1$ whenever
$(a,b)\in\cl$. Hence, the $\ell_1$ 
classes $[a_1,b],\ldots,[a_{\ell_1},b]$ 
come either from $[y=b]\cap\cl$ or 
$[y=b+\ell_2]\cap\cl$. In particular, 
the union $Y_{b+\ell_2}\cup X_{b+\ell_2}$ 
is disjoint and has precisely $\ell_1$ elements. 

\medskip

Let $Y_b':=\{y\in\Z\mid (y,b+\ell_2)\in Y_{b+\ell_2}\}$ 
and $X_b':=\{x\in\Z\mid (x,b+\ell_2)\in X_{b+\ell_2}\}$. 
Since both sets are disjoint, we can define
$$
\cl(z):=\left\{\begin{array}{ll}
(y+\xb,b)& \mbox{if }z=y\in Y_b'
\\
(x,b+\ell_2)& \mbox{if }z=x\in X_b'
\end{array}\right.
$$
which is just the element of $\cl$ 
giving rise to $z$. By Lemma~\ref{lem:HorOrd},
$x$, $x'\in X_b'$ satisfy $\cl(x)<\cl(x')$ 
if and only if $x<x'$, and similarly for $Y_b'$. 
A more involved characterisation holds for
mixed pairs $(x,y)\in X_b'\times Y_b'$, namely
$$
\cl(y)<\cl(x) \ifff 0<x-y<\ell_1
\hspace{0.8em}\mbox{and}\hspace{0.8em}
\cl(x)<\cl(y) \ifff \xc<y-x<\ell_1-\xb.
$$
Indeed, $\cl(y)<\cl(x)$ if and only if 
$(x-y-\xb,\ell_2)\in\nImm(\ell_1,\ell_2;\xa,\xb)$.
Further, $\cl(x)<\cl(y)$ if and only if
$(y+\xb-x,-\ell_2)\in\nImm(\ell_1,\ell_2;\xa,\xb)$. 
The claim follows from the inequalities defining
the {co-immaculate} locus as well as $\xc\geq0$, cf.\ 
Remark~\ref{rem:effective}. In particular, 
$Y_b'\cup X_b'$ forms a sequence of $\ell_1$ 
consecutive integers.
\end{proof}

\begin{remark}
\label{rem:LR}
If in addition $c\not=0$, the proof of 
Proposition~\ref{prop:CompShape} 
also yields a horizontal no-gap lemma
for $Y_{b+\ell_2}\cup X_{b+\ell_2}$
without assuming that the {\mes} $\cl$ is orderable,
compare Corollary~\ref{coro:NoGaps}.
Moreover, it implies for this case that $Y_{b+\ell_2}$ lies 
to the left of $X_{b+\ell_2}$.
Indeed, if there exists $x\in X_b'$ 
with $y=x+1\in Y_b'$, then
$1=y-x>\xc\geq 1$, contradiction.
\end{remark}

{For the sequel we say that an element $\cl_i$ 
of $\cl$ is at {\em level $h$}, if $\cl_i=(a,\underline b+h)$.}

\begin{corollary}
\label{coro:ExHorChains}
Let $\cl$ be a {\mes} with 
$\nl:=2\ell_2-H(\cl)>0$. 
Then there {exist} 
horizontal chains in $\cl$ at level 
$\ell_2-\nl,\ldots,\ell_2-1$. 
\end{corollary}

\begin{proof}
{The assertion is invariant under shifts 
so we may assume that $\underline b=\underline b(\cl)=0$. So if 
$(a,b)\in\cl$, then $\underline b=0\leq b\leq\overline b=
2\ell_2-\nl-1$. Therefore,  
$[y=2\ell_2-\nl]\cap\cl,\ldots,[y=2\ell_2-1]
\cap\cl=\varnothing$. By 
Proposition~\ref{prop:CompShape} we must have 
horizontal chains in $[y=\ell_2-\nl]\cap\cl,
\ldots,[y=\ell_2-1]\cap\cl$.}
\end{proof}

The converse statement also holds.
For this, let
$$
\DeltaUp:=\{(a,b)\in\cP\kst b\geq\ell_2\}
=\{(a,b)\in\Z^2\mid b\geq\ell_2,\; -\xb<a
\mbox{ and }a+b\xa<\ell_1+\xc\}.
$$
Note that the inequalities $b\geq\ell_2$ 
and $-\xb<a$ alone already imply that 
$a+b\xa>\xc$. The subsequent lemma states, 
roughly speaking, that a point of $\cl$ 
at sufficiently high level prevents horizontal 
chains in $\cl$.

\begin{lemma}%
\label{lem:HeightConstraint}
Let $\cl$ be an exceptional sequence 
on $(\ell_1,\ell_2;c)$ which is
vertically orderable. If $(a'',b'')\in\cl$, 
then for all pairs $(a,b)<(a',b)$ 
with $b''-b\geq\ell_2$ we have 
$a'-a<\ell_1-1$ (instead of $\leq\ell_1-1$
as asserted in Lemma~\ref{lem:HorOrd}).
\end{lemma}

\begin{proof}
From the definition of the {co-immaculate} 
locus it follows that 
$$
\DeltaUp\subseteq\{(a,b)\in\Z^2\mid-\xb
<a<\ell_1-\xb\}.
$$
Reordering vertically if necessary 
we have $(a,b)<(a'',b'')$ for any 
$(a,b)\in\cl$ hence $[y=b]\cap\cl$ 
is contained in
$$
(a'',b'')-\DeltaUp\subseteq\{(a,b)\in
\cP\mid\xb+a''-\ell_1<a<\xb+a''\}.
$$
Consequently, $a'-a<\ell_1-1$.
\end{proof}

If an exceptional sequence $\cl$ 
starts at $\cl_0=0$, any horizontal chain 
in $\cl$ must be necessarily located at 
level $h$ with $0\leq h\leq\ell_2-1$. 
We therefore immediately deduce the

\begin{corollary}
\label{coro:SlimTwisted}
Let $\cl$ be an exceptional sequence 
on $(\ell_1,\ell_2;c)$ which is
vertically orderable. 
If there exists a horizontal chain in $\cl$ 
at level $h$, then $H(\cl)\leq\ell_2+h\leq
2\ell_2-1$.
\end{corollary}

\begin{corollary}
\label{coro:VertOcc}
{Let $\cl$ be a {\mes} 
on $(\ell_1,\ell_2;c)$ which is
vertically orderable. If 
$H(\cl)\leq2\ell_2$, 
then $[y=b]\cap\cl\not=\varnothing$ 
for all $\underline b\leq b\leq\overline b$.
In particular, $H(\cl)=\sharp$ of 
rows occupied by $\cl$.}
\end{corollary}

\begin{proof}
{Since the assertion concerns only the
underlying set we may suppose that 
$\cl$ is vertically ordered and $\cl_0=0$.
Assume that
$[y=b]\cap\cl=\varnothing$. If $0\leq b<\ell_2$, 
then $[y=b+\ell_2]\cap\cl$ must have $\ell_1$ 
elements by Proposition~\ref{prop:CompShape}. 
But $\ell_2\leq b+\ell_2$ whence 
$([y=b+\ell_2]\cap\cl)\subseteq\DeltaUp$ which is 
impossible.}

\medskip

{On the other hand, 
$b\geq\ell_2$ implies that 
$[y=b-\ell_2]\cap\cl$ has $\ell_1$ 
elements. By Corollary~\ref{coro:SlimTwisted},
$H(\cl)\leq b$ which contradicts $b\leq H(\cl)-1$.}
\end{proof}

\begin{remark}
For $c=0$ Proposition~\ref{prop:CompShape} and 
Corollaries~\ref{coro:NoGaps}
~\ref{coro:ExHorChains},
~\ref{coro:SlimTwisted} and
~\ref{coro:VertOcc} hold mutatis 
mutandis for the {horizontal} case. 
\end{remark}

\medskip

\subsection{The trivial {\mes}s}
\label{subsec:TrivMes}
In \ssect{subsubsec:HeLexing}
we introduced the standard rectangle
$$
\Rect_{\ell_1,\ell_2}=\{(a,b)\in\Z^2
\mid0\leq a<\ell_1,\,0\leq b<\ell_2\}.
$$
{If the pair $(\ell_1,\ell_2)$ is clear from 
the context we simply write $\Rect$. With respect to 
the vertical lexicographical order there is the maximal exceptional 
sequence given by}
$$
\cl_{a+b\ell_1}:=(a,b)\in\Rect.
$$
Indeed, $\sharp\Rect=\ell_1\ell_2$, 
and the difference of $\cl_i$ and 
$\cl_{i+k}$ sitting in a common row is 
$(k,0)\in\nImm$. In all other cases, the 
difference is of the form $(\cdot,b)$ with 
$b\in\{1,2,\ldots,\ell_2-1\}$. It is thus 
contained in $\nImm(\ell_1,\ell_2;\xa,\xb)$, too.

\medskip

{Furthermore, we obtain {\mes}s by 
\begin{enumerate}
\item[(i)] applying an overall shift $(a,b)$ 
to the entire sequence.
\vspace{1ex}
\item[(ii)] shifting any of the rows at level 
$b=1,\,2,\ldots,\ell_1-1$ by {some} $(a_b,0)\in\Z^2$ 
depending on the level $b$.
\end{enumerate}
The resulting {\mes}s will be referred to as 
{\em vertically trivial}. They are always
vertically orderable and have the maximal number
$z=\ell_2$ of horizontal chains. In 
particular, a {\mes} $\cl$ is vertically 
trivial if and only if $H(\cl)=\ell_2$.}

\medskip

Mutatis mutandis we also have horizontally 
trivial {\mes}s in the product case $c=0$.

\medskip

\section{The dichotomy of the product case}
\label{sec:ProdCase}
Choose two integers $\ell_1$, $\ell_2\geq2$. 
We set out to tackle the Theorems~\thLex-\thFull\ from 
the introduction for the product 
case $(\ell_1,\ell_2;0)=\PP^{\ell_1-1}
\times\PP^{\ell_2-1}$. The reader which
is solely interested in the twisted case
can continue with Section~\ref{sec:MES}.
For the remainder of this section {let}
$\nImm:=\nImm(\ell_1,\ell_2;0,0)$. 

\medskip

\subsection{Exceptional sequences 
are semi-bounded}
\label{subsec:ThmBProdCase}

\begin{theorem}[Theorem {\thHeight}, 
product version]
\label{thm:SlimPaPb}
An exceptional sequence $\cl$ on 
$(\ell_1,\ell_2;0)$ is semi-bounded, 
that is, we have either $H(s)\leq2\ell_2-1$ 
or $W(s)\leq2\ell_1-1$.
\end{theorem}

\begin{proof}
Suppose to the contrary that both 
$H(\cl)\geq2\ell_2$ and $W(\cl)\geq2\ell_1$.

\medskip

We consider (a priori not necessarily 
distinct points) $A=(\underline a,b_A)$, 
$B=(a_B,\underline b)$, $C=(\overline a,b_C)$, 
$D=(a_D,\overline b)\in\cl$, see 
Figure~\ref{fig:SubdivisionR}. We decompose 
$$
R:=\{(a,b)\in\Z^2\mid\underline a\leq a\leq
\overline a,\,\underline b\leq b\leq
\overline b\}
$$ 
into pairs of horizontal and vertical strips, 
namely
$$
\lowerR:=\{(a,b)\in R\mid  b<\underline b
+\ell_2\},\quad\upperR:=\{(a,b)\in R\mid  
b\geq\underline b+\ell_2\}
$$
and 
$$
\leftR:=\{(a,b)\in R\mid  a<\underline a
+\ell_1\},\quad\rightR:=\{(a,b)\in R\mid  
a\geq\underline a+\ell_1\}.
$$
see Figure~\ref{fig:SubdivisionR}. In 
particular, $A\in\leftR$ and $B\in\lowerR$. 
We distinguish two cases.

\begin{figure}[ht]
\newcommand{\spaceA}{\hspace*{2em}}
\newcommand{\scaleA}{0.5}
\begin{tikzpicture}[scale=\scaleA]
\draw[thick, color=black]
  (0,0) -- (5,0) -- (5,5) -- (0,5) -- cycle;
\fill[pattern color=blue!40, pattern=north west lines]
  (-0.5,-0.5) -- (-0.5,2.3) -- (5.5,2.3) -- 
  (5.5,-0.5) -- cycle;
\fill[pattern color=intOrange!40, pattern=north west lines]
  (-0.5,2.7) -- (-0.5,5.5) -- (5.5,5.5) -- 
  (5.5,2.7) -- cycle;
\fill[thick, color=black]
  (0,3) circle (4pt)
  (2,0) circle (4pt)
  (5,1) circle (4pt)
  (4,5) circle (4pt);
\draw[thick, color=black]
  (-0.5,3) node{$A$}
  (2,-0.5) node{$B$}
  (5.5,1) node{$C$}
  (4,5.5) node{$D$}
  (2.5,3.7) node{$\upperR$}
  (2.5,1.3) node{$\lowerR$};
\end{tikzpicture}
\spaceA
\begin{tikzpicture}[scale=\scaleA]
\draw[thick, color=black]
  (0,0) -- (5,0) -- (5,5) -- (0,5) -- cycle;
\fill[pattern color=green!40, pattern=north west lines]
  (-0.5,-0.5) -- (2.3,-0.5) -- (2.3,5.5) -- 
  (-0.5,5.5) -- cycle;
\fill[pattern color=red!40, pattern=north west lines]
  (2.7,-0.5) -- (5.5,-0.5) -- (5.5,5.5) -- 
  (2.7,5.5) -- cycle;
\fill[thick, color=black]
  (0,3) circle (4pt)
  (2,0) circle (4pt)
  (5,1) circle (4pt)
  (4,5) circle (4pt);
\draw[thick, color=black]
  (-0.5,3) node{$A$}
  (2,-0.5) node{$B$}
  (5.5,1) node{$C$}
  (4,5.5) node{$D$}
  (3.7,2.5) node{$\rightR$}
  (1.3,2.5) node{$\leftR$};
\end{tikzpicture}
\caption{The horizontal and the vertical 
subdivision of $R$.}
\label{fig:SubdivisionR}
\end{figure}

\medskip

{\em Case 1:} $A\in\upperR$. Then 
$A\not= B$ and the exceptionality 
condition imply that either $A-B$ 
or $B-A$ lies in $\nImm$. However, 
$A-B=(\underline a-a_B,b_A-
\underline b)\in\nImm$ is impossible 
for $b_A-\underline b\geq\ell_2$ and 
$\underline a-a_B\leq0$. Hence $A<B$ 
and in particular $B\in\leftR$ as 
$\underline b-b_A<0$, so 
$a_B-\underline a<\ell_1$. We conclude 
in a similar way that necessarily 
$B<C$ and $C\in\lowerR$. It follows 
that $A<C$, but this is impossible 
since 
$\overline a-\underline a\geq W(\cl)-1\geq\ell_1$ 
while $b_C-b_A<0$.

\medskip

{\em Case 2:} $A\in\lowerR$. Since 
$\underline a-a_D\leq0$ and 
$b_A-\overline b<0$ we necessarily 
have $A<D$ and $D\in\leftR$. We 
conclude in a similar way that $C<D$ 
and $C\in\upperR$, thus $B<C$ and 
$B\in\rightR$. But then $B<D$ which 
is impossible for $a_D-a_B<0$ and 
$\overline b-\underline b\geq 
H(\cl)-1\geq\ell_2$.
\end{proof}

\medskip

\subsection{An inductive argument}
\label{subsec:indArg}
Our mainstream development for proving fullness
will pursue a rather algorithmic approach based 
on the lexicographical order from \ssect{subsec:ReLExProdCase} 
and the classification of {\mes}s in \ssect{subsec:ClassProdCase}. 
As an aside, we briefly sketch an inductive approach to 
Theorem~{\thFull} which is based {on the following}
{\em collapsing procedure}.

\begin{lemma}
\label{lem:ProductColl}
For $\ell_2\geq3$ let $\cl$ be an exceptional 
sequence on $(\ell_1,\ell_2;0)$ with $\cl_0=0$
and $H(\cl)\leq2\ell_2-1$. Then we obtain a 
sequence $\cl'$ on 
$$
(\ell_1',\ell_2';0)=(\ell_1,\ell_2-1;0)
$$ 
via the following procedure:
\begin{enumerate}
\item[(i)] 
Remove the horizontal line at level $\ell_2-1$ 
from $\cl$.
\vspace{1ex}
\item[(ii)] 
For every $\cl_i=(a_i,b_i)$ with $b_i\geq\ell_2$
put $\cl'_i:=\cl_i-(0,1)$.
\vspace{1ex}
\item[(iii)] 
For all remaining $\cl_i$ put $\cl_i':=\cl_i$.
\end{enumerate}
If we endow $\cl'$ with the order induced by $\cl$,
then $\cl'$ defines an exceptional sequence, too. 
\end{lemma}

\begin{proof}
Consider $\cl_i=(a_i,b_i)<\cl_j=(a_j,b_j)$ and 
assume that $b_i$, $b_j\neq\ell_2-1$. We denote 
the {co-immaculate} locus of $(\ell_1',\ell_2';0)$
by $\nImm'$.

\medskip

{\em Case 1.} If $b_i\geq b_j$ or $b_i+2\leq b_j
\leq b_i+\ell_2-2$, then $\cl_1'<\cl_2'$ is 
immediate.

\medskip

{\em Case 2.} If $b_j=b_i+1$, then both $\cl_i$ 
and $\cl_j$ belong to the same side either above 
or below the removed line. Hence, $\cl'_j-\cl_i'=
\cl_j-\cl_i=(a_j-a_i,1)\in\nImm'$.

\medskip

{\em Case 3.} If $b_j\geq b_i+\ell_2-1$, then 
$\cl_i$ sits below and $\cl_j$ sits above the 
line at level $\ell_2-1$. Thus,
$\cl'_j-\cl_i'=\cl_j-\cl_i-(0,1)\in\nImm'$.
\end{proof}

We say that $\cl'$ is obtained from $\cl$ by 
{\em collapsing along $\ell_2$}. Similarly, we can
collapse along $\ell_1$ if $W(\cl)\leq2\ell_1-1$, 
cf.\ also Theorem~\ref{thm:Dichotomy} and
Lemma~\ref{lem:L1Coll}.

\medskip

As an immediate consequence of the collapsing
procedure we obtain that
for a {\mes} where $\sharp\cl=\ell_1\ell_2$, the
inequalities $\sharp\cl'\leq \ell_1(\ell_2-1)$ and 
$\sharp([y=\ell_2-1]\cap\cl)\leq\ell_1$ imply that
$$
\sharp\cl'= \ell_1(\ell_2-1)
\quad\mbox{and}\quad
\sharp(\cl\cap[y=\ell_2-1])=\ell_1.
$$
Therefore, $\cl'$ is maximal, too, 
and we actually removed a horizontal chain in
the sense of Definition~\ref{def:HorChain}.

\medskip

This allows to prove fullness in a rather 
implicit way:

\begin{corollary}[Theorem {\thFull}, product version]
\label{cor:FullnessTheoremPaPb}
On $(\ell_1,\ell_2;0)$ every maximal 
exceptional sequence $\cl$ with 
$H(\cl)\leq2\ell_2-1$ is full.
\end{corollary}

Of course, the same result holds for 
$W(\cl)\leq2\ell_1-1$.

\medskip

{\em Sketch of the proof.}
Proceeding by induction we may assume 
that the collapsed sequence $\cl'$ 
generates $\Z^2$ by filling horizontal and 
vertical lines whenever there are
$\ell_1$ or $\ell_2$ consecutive points, 
respectively.

\medskip

Given $\cl$ we begin by using the horizontal chain 
at level $\ell_2-1$ to fill the entire horizontal 
line on which it lies. Afterwards, we may lift all 
line fillings from $\cl'$ to $\cl$ since any vertical
chain of $\cl'$ must reach height $\ell_2-1$.
\hfill$\square$

\medskip

\subsection{Maximal exceptional 
sequences are orderable}
\label{subsec:ReLExProdCase}

\begin{theorem}[Theorem {\thLex}, 
product version]
\label{thm:MESRelProdCase}
Let $\cl=(\cl_0,\ldots,\cl_N)$ 
be a maximal exceptional sequence on 
$(\ell_1,\ell_2;0)$. Then $\cl$ is 
vertically or horizontally orderable. 
\end{theorem}

\begin{proof}
{Assume that $\cl$ is a maximal 
exceptional sequence which is {not vertically} 
orderable.}

\medskip

{Reordering if necessary, Lemma~\ref{lem:ReorderJump}
implies that there exists a pair $\cl_i<\cl_{i+1}$
with $b_i-b_{i+1}\geq\ell_2$. Upon
applying $i$-times the helix operator from 
\ssect{subsec:helixOp} we may replace $\cl$
by a sequence where $i=0$. Choosing suitable
coordinates we may therefore suppose
without loss of generality that
$$
\cl_0=(0,0),\quad\cl_1=(\kappa,\lambda)
$$
for $\lambda\leq-\ell_2$ and 
$0<\kappa<\ell_1$.}

\medskip

{Next, there exists precisely 
one $i_0>0$ such that $\Phi(\cl_{i_0})=[1,0]$
by Lemma~\ref{lem:PhiMap}. Hence, 
$\cl_{i_0}\in\big((1,0)+L\big)\cap\nImm$, 
where we recall that
$$
L=\Z(\ell_1,0)\oplus\Z(0,\ell_2).
$$
Consequently, 
$$
\cl_{i_0}=(1,m\ell_2)\in\cP
$$
for some $m\in\Z$, where $\cP$ is the parallelogram of the 
{co-immaculate} locus, cf.\ Subsection~\eqref{subsec:ImmLoc}.}

\medskip

{On the other hand, $\cl_{i_0}$ is equal to or 
a successor of $\cl_1=(\kappa,\lambda)$, hence
$$
\cl_{i_0}\in\big((\kappa,\lambda)+\nImm\big)
\cup\{(\kappa,\lambda)\}.
$$
But $\lambda\leq-\ell_2$ implies 
$\cP\cap\big((\kappa,\lambda)+\CP\big)=\varnothing$, 
so that $\cl_{i_0}\in\nImm$ yields the contradiction 
$0<\kappa<1$ unless $i_0=1$, that is, $\kappa=1$ and $m\leq-1$.}

\medskip

{Now any point $\cl_j=(a_j,b_j)$, $j\geq2$ must be in 
\[
\nImm\cap\big((1,m\ell_2)+\nImm\big)=\{(a,b)\in\Z^2\mid1<a<\ell_1\}. 
\]
By Remark~\ref{rem:RelProdCase} it follows that $\cl$ 
must be horizontally orderable.}
\end{proof}

\begin{remark}
As $\cl^1$ and $\cl^2$ in Example~\ref{ex:P1P1} 
show for suitable $a$ and $b$, a maximal 
exceptional sequence is not necessarily both 
vertically and horizontally orderable. In particular,
there are \mes s examples of height equal to or less than 
$2\ell_2-1$ which are either vertically or
horizontally orderable, but not both. 
Similarly for constrained width.
\end{remark}

\medskip

\subsection{The dichotomy of {\mes}s}
\label{subsec:DichoMESPC}
As a first step towards the classification
of {\mes} we want to combine a lexicographical
order with a spatial constraint.

\begin{lemma}
\label{prop:HorVer}
Let $\cl$ be a {\mes} with $H(\cl)\geq2\ell_2$.
Then $\cl$ is horizontally orderable.
\end{lemma}

\begin{proof}
{By Theorem~\ref{thm:MESRelProdCase}, it is enough
to show that vertical orderability
implies horizontal orderability. So if
$<$ denotes the vertical order
on $\cl$, we will show that $a_j-a_i>-\ell_1$ 
whenever $\cl_i<\cl_j$, and appeal again to 
Remark~\ref{rem:RelProdCase}.}

\medskip

By definition, 
$\cl_N-\cl_i=(a_N-a_i,b_N-b_i)$
is in the {co-immaculate} locus for all $i<N$. 
If $b_N-b_i\geq\ell_2$, then  
$$
0<a_N-a_i<\ell_1
$$
by the equations defining $\nImm$. 
If $b_N-b_i<\ell_2$, then 
$H(\cl)\geq2\ell_2$ implies that $b_i$ is 
at least at level $\ell_2$ whence 
$$
0<a_i-a_0,\,a_N-a_0<\ell_1.
$$
At any rate, $-\ell_1<a_N-a_i<\ell_1$.

\medskip

Next consider $\cl_i=(a_i,b_i)<
\cl_j=(a_j,b_j)$ for $j<N$. By the above,
$a_j-a_N>-\ell_1$. Since $b_i\leq b_j$,
$b_N-b_i<\ell_2$ implies $b_N-b_j<\ell_2$
and therefore $0<a_i-a_0,\,a_j-a_0<\ell_1$. 
In particular, $a_j-a_i>-\ell_1$. On the
other hand, $b_N-b_i\geq\ell_2$ yields
also
$$
a_j-a_i=a_j-a_N+a_N-a_i>-\ell_1.
$$
\end{proof}

This gives rise to the following dichotomy. 
If $\cl$ satisfies
$H(\cl)\geq2\ell_2$, then Lemma~\ref{prop:HorVer} 
and Theorem~\ref{thm:SlimPaPb} imply 
horizontal orderability and $W(\cl)\leq
2\ell_1-1$. Similarly, 
$W(\cl)\geq2\ell_1$ implies 
vertical orderability and $H(\cl)\leq2\ell_2-1$.
Moreover, independently on $H(\cl)$ or $W(\cl)$,
$\cl$ is either vertically 
or horizontally orderable by 
Theorem~\ref{thm:MESRelProdCase}. This 
implies the

\begin{theorem}[Dichotomy of \mes s]
\label{thm:Dichotomy}
To any {\mes} $\cl$ on $(\ell_1,\ell_2;0)$
at least one of the following two items applies:
\begin{enumerate}
\item[(i)] $\cl$ is vertically orderable with 
$H(\cl)\leq2\ell_2-1$
\vspace{1ex}
\item[(ii)] $\cl$ is horizontally orderable 
with $W(\cl)\leq2\ell_1-1$.
\end{enumerate}
\end{theorem}

For sake of concreteness {\em we will 
concentrate on the first case for the remainder 
of this paper} since this fits into the twisted case; 
mutatis mutandis everything which follows also 
applies to the second case.

\medskip

\section{The classification for the 
product case}
\label{sec:ClassPC}
\subsection{HeLexing}
\label{subsec:HeLexPC}
Recall from \ssect{subsec:helixOp}
the definition of the helix operator 
$\helix$ which sends a {\mes} $\cl$ on 
$(\ell_1,\ell_2;0)$ to
$$
\helix(\cl)=\big(\cl_1,\ldots,\cl_N,\cl_0+
(\ell_1,\ell_2)\big).
$$

\begin{lemma}
Let $\cl$ be a vertically ordered {\mes} on 
$(\ell_1,\ell_2;0)$ with $\cl_0=0$ and 
$\ell_2<H(\cl)\leq2\ell_2-1$. Then $\helix(\cl)$ 
is vertically orderable with $H(\helix(\cl))\leq2\ell_2-1$.
\end{lemma}

\begin{proof}
Let
$$
\cl'=(\cl_0',\ldots,\cl_N'):=\helix(\cl).
$$
Since $H(\cl)>\ell_2$ there is a point 
$\cl_i=(a_i,b_i)$ with $b_i\geq\ell_2$ for 
some $i=1,\ldots,N$. Therefore, $H(\helix(\cl))
\leq H(\cl)\leq2\ell_2-1$.

\medskip

Next assume that $\cl'=\helix(\cl)$ is not
vertically orderable. By Proposition
~\ref{prop:RelCrit}, there exists a pair 
$\cl_i'<\cl_j'$ with $b_i'-b_j'\geq\ell_2$. 
Since $\cl_i'=\cl_{i+1}$ and $\cl_j'=\cl_{j+1}$ 
if $j<N$, vertical orderability of $\cl$ 
implies that $j=N$. However, 
$b_{i+1}-\ell_2\geq\ell_2$ implies 
$b_{i+1}\geq2\ell_2$, contradicting 
$H(\cl)\leq2\ell_2-1$.
\end{proof}

It follows that we can vertically reorder 
$\helix(\cl)$. Shifting yields an exceptional
sequence $\heLex(\cl)$ starting at the origin.
We call $\heLex$ the {\em heLex operator}.

\begin{theorem}[Theorem~\thHelex, product version]
\label{thm:HeLexingPC}
On $(\ell_1,\ell_2;0)$ every vertically 
orderable {\mes} $\cl$ of height 
$H(\cl)\leq 2\ell_2-1$ can be transformed 
into a vertically trivial sequence by 
vertically reordering and successively 
applying $\heLex$ 
at most $\ell_1\ell_2$ times.
\end{theorem}

\begin{proof}
If $\cl$ is vertically ordered with $\cl_0=0$, 
the helix operator $\helix$ sends $\cl_0$ to
$(\ell_1,\ell_2)$. Unless $H(\cl)\leq \ell_2$, 
that is, $\cl$ is vertically trivial, either 
\begin{enumerate}
\item[(i)]
the lowest row becomes empty, whence $H(\cl)$ 
decreases at least by one, or
\vspace{1ex}
\item[(ii)]
the height $H(\cl)$ remains unchanged, so that 
we are reducing the number of elements in 
$[y=0]\cap\cl$ by one.
\end{enumerate}
By induction we eventually arrive at a {\mes} 
of height $H(\cl)\leq\ell_2$, i.e., a 
vertically trivial one.
\end{proof}

\medskip

\subsection{The classification of {\mes}s}
\label{subsec:ClassProdCase}
Next we discuss the classification of the
sets underlying a {\mes} on 
$(\ell_1,\ell_2;0)$ by giving an explicit 
algorithm for their construction. We recall
that we tacitly assume to work with
vertically orderable sequences $\cl$ of height
$H(\cl)\leq2\ell_2-1$.

\medskip

For any subset $X\subseteq\DeltaUp=
\{(a,b)\in\Z^2\mid b\geq\ell_2,\; 
0<a<\ell_1\}$ 
or, more generally, for any
$X\subseteq\Z^2$ we let
$$
X_k:=[y=k]\cap X=\{(a,b)\in X\kst b=k\}.
$$

\begin{definition}
\label{def:AdmSetPC}
We call a non-empty set $X\subseteq
\DeltaUp$ {\em admissible} if
\begin{enumerate}
\item[(Ai)]
$X_k=\varnothing$ for $k\geq 2\ell_2-1$.
\vspace{1ex}
\item[(Aii)]
the layers $X_k\not=\varnothing$ consist 
of successive points $(x,k),\,(x+1,k),
\ldots,(x+q_k,k)$.
\vspace{1ex}
\item[(Aiii)]
for each $k\geq\ell_2$ we have 
$$
(0,-1)+X_{k+1}\subseteq X_k.
$$
\item[(Aiv)]
the bottom layer $X_{\ell_2}$ is 
right-aligned, i.e., $(\ell_1-1,\ell_2)
\in X_{\ell_2}$.
\end{enumerate}

\medskip

By convention, the empty set will be 
admissible, too.
\end{definition}

In addition to an admissible set
$X\subseteq \DeltaUp$ we need 
a further set to completely 
classify {\mes}s.

\begin{definition}
\label{def:SuppPartnerPC}
Let $\varnothing\not=X\subseteq\DeltaUp$ 
be admissible. Then $Y\subseteq\Z^2$ 
is called a {\em complementing partner} 
of $X$ if
\begin{enumerate}
\item[(Cv)]
$Y_k=\varnothing$ for $k\geq 2\ell_2$.
\vspace{1ex}
\item[(Cvi)]
$Y_k<X_k$, meaning that $(y,k)\in Y_k$ 
and $(x,k)\in X_k$ imply $y<x$.
\vspace{1ex}
\item[(Cvii)]
for each $k\in\{\ell_2,\ldots,2\ell_2-1\}$, 
the union $Y_k\cup X_k$ forms a horizontal 
chain.
\end{enumerate}

\medskip

For $X=\varnothing$ a complementing partner 
will be any set $Y$ consisting of $\ell_2$ 
horizontal chains $Y_{\ell_2},\ldots,
Y_{2\ell_2-1}$ such that $Y_{\ell_2}$ 
starts at $(0,\ell_2)$.
\end{definition}

\begin{remark}
Whenever $X_k\neq\varnothing$, the 
complementary set $Y_k$ is uniquely 
determined. In contrast, $X_k=\varnothing$ 
for $k\in\{\ell_2,\ldots,2\ell_2-1\}$ 
implies that $Y_k$ defines a horizontal 
chain whose horizontal position is 
unrestricted unless $k=\ell_2$.
\end{remark}

\begin{example}
\label{ex:XY}
Figure~\ref{fig:XY} displays a typical 
{\mes} together with its 
admissible set and complementing partner.
\begin{figure}[ht]
\newcommand{\spaceA}{\hspace*{2em}}
\newcommand{\scaleA}{0.4}
\begin{tikzpicture}[scale=\scaleA]
\draw[color=oliwkowy!40] (-2.3,-1.3) grid (6.3,8.3);
\fill[pattern color=parallelogram, pattern=north west lines]
  (0,-1.3) -- (0,8.3) -- (5,8.3) -- 
  (5,-1.3) -- cycle;
\fill[pattern color=parallelogram, pattern=north west lines]
  (-2.3,0) -- (0,0) -- (0,4) -- 
  (-2.3,4) -- cycle;
\fill[pattern color=parallelogram, pattern=north west lines]
  (5,0) -- (6.3,0) -- (6.3,4) -- 
  (5,4) -- cycle;
\draw[thick, color=black]
  (0,0) -- (-1,0) (5,0) -- (5.5,0)
  (0,4) -- (-1,4) (5,4) -- (5.5,4);
\draw[thick, color=black]
  (0,4) -- (0,6.5) (5,4) -- (5,6.5)
  (0,0) -- (0,-0.5) (5,0) -- (5,-0.5);
\draw[thick, dotted, color=black]
  (-1,0) -- (-2,0) (5.5,0) -- (6.3,0)
  (-1,4) -- (-2,4) (5.5,4) -- (6.3,4);  
\draw[thick, dotted, color=black]
  (0,6.5) -- (0,7.3) (5,5.5) -- (5,7.3)
  (0,-0.5) -- (0,-1.3) (5,-0.5) -- (5,-1.3);
\fill[thick, color=origin]
    (0,0) circle (7pt);
\fill[thick, color=black]
    (1,0) circle (7pt) (0,1) circle (7pt) 
    (1,1) circle (7pt)
    (-1,2) circle (7pt) (0,2) circle (7pt) 
    (1,2) circle (7pt)
    (2,2) circle (7pt) (1,3) circle (7pt) 
    (2,3) circle (7pt) 
    (3,3) circle (7pt) (4,3) circle (7pt) 
    (5,3) circle (7pt)
    (2,4) circle (7pt) (3,4) circle (7pt) 
    (4,4) circle (7pt)
    (2,5) circle (7pt) (3,5) circle (7pt) 
    (4,5) circle (7pt) 
    (3,6) circle (7pt);
\draw[thin, color=red]
    (2,4) circle (7pt) (3,4) circle (7pt)
    (4,4) circle (7pt) (2,5) circle (7pt)
    (3,5) circle (7pt) (4,5) circle (7pt)
    (3,6) circle (7pt);
\draw[thin, color=blue]
    (0,0) circle (7pt)
    (1,0) circle (7pt) (0,1) circle (7pt)
    (1,1) circle (7pt) (-1,2) circle (7pt)
    (0,2) circle (7pt) (1,2) circle (7pt)
    (2,2) circle (7pt) (1,3) circle (7pt)
    (2,3) circle (7pt) (3,3) circle (7pt)
    (4,3) circle (7pt) (5,3) circle (7pt);
\end{tikzpicture}
\spaceA
\begin{tikzpicture}[scale=\scaleA]
\draw[color=oliwkowy!40] (-2.3,-1.3) grid (6.3,8.3);
\fill[thick, color=origin]
    (0,0) circle (7pt);
\fill[thick, color=blue]
    (0,4) circle (7pt);
\fill[thick, color=blue]
    (1,4) circle (7pt) (0,5) circle (7pt) 
    (1,5) circle (7pt) (-1,6) circle (7pt) 
    (0,6) circle (7pt) (1,6) circle (7pt)
    (2,6) circle (7pt) (1,7) circle (7pt) 
    (2,7) circle (7pt) (3,7) circle (7pt) 
    (4,7) circle (7pt) (5,7) circle (7pt);
\fill[thick, color=red]
    (2,4) circle (7pt) (3,4) circle (7pt) 
    (4,4) circle (7pt) (2,5) circle (7pt) 
    (3,5) circle (7pt) (4,5) circle (7pt) 
    (3,6) circle (7pt);
\end{tikzpicture}
\caption{A {\mes} on $(5,4;0)$ with starting 
point at the origin in green. The right hand 
side displays its corresponding admissible 
set $X$ in red with complementing partner 
$Y$ in blue.}
\label{fig:XY}
\end{figure}
\end{example}

\begin{theorem}[Theorem \thClass, product 
version]
\label{thm:mesAfterAdShiftPC}
Let $X\subseteq \DeltaUp$ be 
admissible and $Y\subseteq\Z^2$ a 
complementing partner. Then the 
union of
$$
\clDown:=Y+(0,-\ell_2)\subseteq\cH
\cup[y=0]\quad\mbox{and}\quad 
\clUp:=X\subseteq\DeltaUp
$$
together with vertical order yields 
a {\mes} $\cl$ with $H(\cl)\leq2\ell_2-1$
and $\cl_0=0$. {Moreover, any
vertically ordered {\mes} starting at the origin 
arises this way.}
\end{theorem}

\begin{proof}
If $X=\varnothing$, then $Y$ consists of 
$\ell_2$ consecutive horizontal chains.
Shifting down by $(0,-\ell_2)$ yields 
a vertically trivial sequence starting 
at the origin. 

\medskip

We therefore assume that $X\not=\varnothing$.
We order the set $\clDown\cup\clUp$ 
vertically to obtain the sequence $\cl$. 
From Definition~\ref{def:HorChain},
Definition~\ref{def:AdmSetPC}~(iv), and
Definition~\ref{def:SuppPartnerPC}~(Cvii), 
it is clear that $Y_{\ell_2}\cup X_{\ell_2}$ 
forms a horizontal chain which ends at 
$(\ell_1-1,\ell_2)$. Consequently, 
$Y_{\ell_2}$ starts at $(0,\ell_2)$ and
$\cl_0=0$.

\medskip

For $\cl_i<\cl_j$ we have to show that 
$\cl_j-\cl_i\in\nImm$. If $\cl_i$, 
$\cl_j\in[y=b]$, then this follows from
$\cl_i,\cl_j\in Y_{b+\ell_2}+(0,-\ell_2)$ 
or $\cl_i,\cl_j\in X_b$,
(Cvi) and (Cvii). If they are at different levels 
$0\leq b_i<b_j\leq 2\ell_2-1$, the only
critical case arises from $b_j-b_i\geq \ell_2$
which implies $0\leq b_i\leq\ell_2-1$ and 
$\ell_2\leq b_j\leq2\ell_2-1$. This means
$$
\cl_i\in[y=b_i]\cap\clDown=Y_{b_i+\ell_2}+
(0,-\ell_2)\quad\mbox{and}\quad\cl_j\in[y=b_j]
\cap\clUp=X_{b_j}.
$$
We proceed via induction over $m:=b_j-b_i-
\ell_2\geq 0$. 

\medskip

If $m=0$, then $\cl_i\in Y_{b_j}+(0,-\ell_2)$ 
whence
$$
\cl_j-\big(\cl_i+(0,\ell_2)\big)\in X_{b_j}-Y_{b_j}
\subseteq\{(1,0),\,\ldots, (\ell_1-1,0)\}.
$$
In particular,  
$\cl_j-\cl_i\in (0,\ell_2)+\{(1,0),\,\ldots,
(\ell_1-1,0)\}\subseteq\DeltaUp\subseteq\nImm$.

\medskip

Next let $m\geq 1$. By (Aiii) we know that
$\cl_j-(0,1)\in\clUp$. On the other hand,
the induction hypothesis implies $\cl_j-(0,1)
-\cl_i\in\nImm$, and since $b_j-1-b_i-\ell_2
\geq0$ we even have $\cl_j-(0,1)-\cl_i\in
\DeltaUp$. Further, $\cl_j-\cl_i\in\DeltaUp$ 
by definition of $\DeltaUp$ whence $\cl_j-\cl_i
\in\nImm$.

\medskip

Finally, we want to show that any vertically
ordered {\mes} $\cl$ with $\cl_0=0$ and $H(\cl)
\leq2\ell_2-1$ arises this way. For this, we 
let
$$
Y:=(0,\ell_2)+\big(\cl\cap(\cH\cup[y=0])\big)
\quad\mbox{and}\quad X:=\cl\cap\DeltaUp
$$ 
where we identify the sequence $\cl$ with its
underlying set. 

\medskip

If $\cl$ is vertically trivial, 
then $Y=(0,\ell_2)+\cl$ and $X=\varnothing$. 
We therefore assume that $\cl$ is not 
trivial and check that $X$ is admissible 
with complementing partner $Y$. Properties 
(Ai) and (Cv) follow from $H(\cl)\leq2\ell_2-1$ 
and the definition of $Y$, respectively. 
Furthermore, (Aii) follows from
Corollary~\ref{coro:NoGaps}, while (Cvii) 
is a consequence of 
Proposition~\ref{prop:CompShape}.

\medskip

By Corollary~\ref{coro:VertOcc}, 
$[y=b]\cap\cl\not=\varnothing$ for 
$0\leq b\leq H(\cl)-1$. Furthermore,
we have $\nl=2\ell_2-H(\cl)$ horizontal lines 
by Corollary~\ref{coro:ExHorChains}. {For 
$b=0,\ldots,\ell_2-\nl-1$, we put here and in the sequel
\begin{equation}\label{eq:rb}
0\leq r_b:=\sharp[y=b]\cap\cl-1<\ell_1-1
\end{equation}
so that $[y=b]\cap\cl=\{(a_b,b),\ldots,(a_b+r_b,b)\}$.} 

\medskip

In particular, we have 
$\{(r_0+1,\ell_2),\ldots,(\ell_1-1,\ell_2)\}
=[y=\ell_2]\cap\cl$ which implies (Aiv). 
For all other $0<b<\ell_2-\nl-1$, 
$Y_{b+\ell_2}=(0,\ell_2)+[y=b]\cap\cl$ is either 
to the left or to the right of 
$X_{b+\ell_2}$. If $Y_{b+\ell_2}$ is to the right, 
then $(a_b-1,b+\ell_2)\in X_{b+\ell_2}$ whence 
$$
(a_b-1,b+\ell_2)-(a_b+r_b,b)=
(-1-r_b,\ell_2)\in\nImm,
$$
contradiction. This implies (Cvi).

\medskip

Finally, Proposition~\ref{prop:CompShape}
implies $(a_b+r_b+1,b+\ell_2)
\in X=\cl\cap\DeltaUp$. It follows that 
$(a_b+r_b+1,b+\ell_2)-(a_{b-1}
+r_{b-1},b-1)\in\nImm$ whence 
$a_b+r_b\geq a_{b-1}+r_{b-1}$. 
Similarly, we have $a_b\leq a_{b-1}$. 
In particular, $(0,1)+Y_{b+\ell_2}
\subseteq Y_{b+\ell_2+1}$
and thus $(0,-1)+X_{b+\ell_2+1}
\subseteq X_{b+\ell_2}$
which yields (Aiii).
\end{proof}

\medskip

\subsection{Generating the derived category}
\label{subsec:ProofPaPb}

We can use Theorem~\ref{thm:mesAfterAdShiftPC} 
to prove fullness of any {\mes} on 
$(\ell_1,\ell_2;0)$. By the dichotomy {principle}, it 
suffices to consider the case of a vertically 
orderable sequence $\cl$ of height 
$H(\cl)\leq2\ell_2-1$. The case of horizontally
orderable sequences of width $W(\cl)\leq2\ell_1-1$
follows analogously.

\medskip

Let $\cl$ be a {\mes}. Since the standard 
rectangle $\Rect_{\ell_1,\ell_2}$ generates 
the Picard group via the horizontal and vertical
Beilinson sequence it suffices to show that 
$\Rect_{\ell_1,\ell_2}\subseteq\langle\cl\rangle$.

\begin{example}
We consider again the sequence $\cl$ from 
Example~\ref{ex:XY}. Figure~\ref{fig:Fullness} 
displays our strategy to fill all of $\Z^2$ 
starting from $\cl$. After filling horizontal 
and vertical lines in Steps (a)-(e) we see that 
$\Rect_{5,4}\subseteq\langle\cl\rangle$.

\begin{figure}[ht]
\newcommand{\spaceA}{\hspace*{1.5em}}
\newcommand{\scaleA}{0.3}
\begin{tikzpicture}[scale=\scaleA]
\draw[color=oliwkowy!40] (-2.3,-1.3) grid (5.3,6.3);
\draw[thick, color=red]
  (-2,3) -- (5,3);
\fill[thick, color=origin]
    (0,0) circle (9pt);
\fill[thick, color=black]
    (1,0) circle (9pt) (0,1) circle (9pt)  
    (1,1) circle (9pt) (-1,2) circle (9pt) 
    (0,2) circle (9pt) (1,2) circle (9pt)
    (2,2) circle (9pt) (1,3) circle (9pt) 
    (2,3) circle (9pt) (3,3) circle (9pt) 
    (4,3) circle (9pt) (5,3) circle (9pt)
    (2,4) circle (9pt) (3,4) circle (9pt) 
    (4,4) circle (9pt) (2,5) circle (9pt) 
    (3,5) circle (9pt) (4,5) circle (9pt) 
    (3,6) circle (9pt);
\draw[thick, color=black]
    (1.5,-3) node{(a)};
\end{tikzpicture}
\spaceA
\begin{tikzpicture}[scale=\scaleA]
\draw[color=oliwkowy!40] (-2.3,-1.3) grid (5.3,6.3);
\draw[thick, color=blue]
  (-2,3) -- (5,3);
\draw[thick, color=red]
  (3,6) -- (3,-1);
\fill[thick, color=origin]
    (0,0) circle (9pt);
\fill[thick, color=black]
    (1,0) circle (9pt) (0,1) circle (9pt) 
    (1,1) circle (9pt) (-1,2) circle (9pt) 
    (0,2) circle (9pt) (1,2) circle (9pt)
    (2,2) circle (9pt) (1,3) circle (9pt) 
    (2,3) circle (9pt) (3,3) circle (9pt) 
    (4,3) circle (9pt) (5,3) circle (9pt)
    (2,4) circle (9pt) (3,4) circle (9pt) 
    (4,4) circle (9pt) (2,5) circle (9pt) 
    (3,5) circle (9pt) (4,5) circle (9pt) 
    (3,6) circle (9pt);
\draw[thick, color=black]
    (1.5,-3) node{(b)};
\end{tikzpicture}
\spaceA
\begin{tikzpicture}[scale=\scaleA]
\draw[color=oliwkowy!40] (-2.3,-1.3) grid (5.3,6.3);
\draw[thick, color=blue]
  (-2,3) -- (5,3);
\draw[thick, color=blue]
  (3,6) -- (3,-1);
\draw[thick, color=red]
  (-2,2) -- (5,2);
\fill[thick, color=origin]
    (0,0) circle (9pt);
\fill[thick, color=red]
    (3,2) circle (9pt);
\fill[thick, color=black]
    (1,0) circle (9pt) (0,1) circle (9pt) 
    (1,1) circle (9pt) (-1,2) circle (9pt) 
    (0,2) circle (9pt) (1,2) circle (9pt)
    (2,2) circle (9pt) (1,3) circle (9pt) 
    (2,3) circle (9pt) (3,3) circle (9pt) 
    (4,3) circle (9pt) (5,3) circle (9pt)
    (2,4) circle (9pt) (3,4) circle (9pt) 
    (4,4) circle (9pt) (2,5) circle (9pt) 
    (3,5) circle (9pt) (4,5) circle (9pt) 
    (3,6) circle (9pt);
\draw[thick, color=black]
    (1.5,-3) node{(c)};
\end{tikzpicture}
\spaceA
\begin{tikzpicture}[scale=\scaleA]
\draw[color=oliwkowy!40] (-2.3,-1.3) grid (5.3,6.3);
\draw[thick, color=blue]
  (-2,3) -- (5,3);
\draw[thick, color=blue]
  (3,6) -- (3,-1);
\draw[thick, color=blue]
  (-2,2) -- (5,2);
\draw[thick, color=red]
  (2,6) -- (2,-1);
\draw[thick, color=red]
  (4,6) -- (4,-1);
\fill[thick, color=origin]
    (0,0) circle (9pt);
\fill[thick, color=red]
    (3,2) circle (9pt) (4,2) circle (9pt);
\fill[thick, color=black]
    (1,0) circle (9pt) (0,1) circle (9pt) 
    (1,1) circle (9pt) (-1,2) circle (9pt) 
    (0,2) circle (9pt) (1,2) circle (9pt) 
    (2,2) circle (9pt) (1,3) circle (9pt) 
    (2,3) circle (9pt) (3,3) circle (9pt) 
    (4,3) circle (9pt) (5,3) circle (9pt)
    (2,4) circle (9pt) (3,4) circle (9pt) 
    (4,4) circle (9pt) (2,5) circle (9pt) 
    (3,5) circle (9pt) (4,5) circle (9pt) 
    (3,6) circle (9pt);
\draw[thick, color=black]
    (1.5,-3) node{(d)};
\end{tikzpicture}
\spaceA
\begin{tikzpicture}[scale=\scaleA]
\draw[color=oliwkowy!40] (-2.3,-1.3) grid (5.3,6.3);
\draw[thick, color=blue]
  (-2,3) -- (5,3);
\draw[thick, color=blue]
  (3,6) -- (3,-1);
\draw[thick, color=blue]
  (-2,2) -- (5,2);
\draw[thick, color=blue]
  (2,6) -- (2,-1);
\draw[thick, color=blue]
  (4,6) -- (4,-1);
\draw[thick, color=red]
  (-2,1) -- (5,1);
\draw[thick, color=red]
  (-2,0) -- (5,0);
\fill[thick, color=origin]
    (0,0) circle (9pt);
\fill[thick, color=red]
    (3,2) circle (9pt) (4,2) circle (9pt);
\fill[thick, color=red]
    (2,1) circle (9pt) (3,1) circle (9pt) 
    (4,1) circle (9pt);
\fill[thick, color=red]
    (2,0) circle (9pt) (3,0) circle (9pt) 
    (4,0) circle (9pt);
\fill[thick, color=black]
    (1,0) circle (9pt) (0,1) circle (9pt) 
    (1,1) circle (9pt) (-1,2) circle (9pt) 
    (0,2) circle (9pt) (1,2) circle (9pt)
    (2,2) circle (9pt) (1,3) circle (9pt) 
    (2,3) circle (9pt) (3,3) circle (9pt) 
    (4,3) circle (9pt) (5,3) circle (9pt)
    (2,4) circle (9pt) (3,4) circle (9pt) 
    (4,4) circle (9pt) (2,5) circle (9pt) 
    (3,5) circle (9pt) (4,5) circle (9pt) 
    (3,6) circle (9pt);
\draw[thick, color=black]
    (1.5,-3) node{(e)};
\end{tikzpicture}
\caption{Filling $\Pic(5,4;0,0)$ from $\cl$. 
The green dot marks the origin. The red 
dots are generated by $\cl$ and are used to
fill further lines.}
\label{fig:Fullness}
\end{figure}
\end{example}

\begin{theorem}[Theorem {\thFull}, 
product version]
\label{thm:FullnessTheoremPaPb}
On $(\ell_1,\ell_2;0)$ every maximal 
exceptional sequence $\cl$ is full.
\end{theorem}

\begin{proof}
{We are indebted to the referee for pointing out to us the following much more elegant version of the proof.

\medskip

We may assume that $\cl$ is of the form
given in Theorem~\ref{thm:mesAfterAdShiftPC} and proceed 
by induction on $\#X$.
If the admissible set 
$X\subseteq \DeltaUp$ is empty, we have already observed in the proof
of Theorem~\ref{thm:mesAfterAdShiftPC} that $\cl$ is
vertically trivial. 

\medskip

Next consider $\#X>0$. First, we fill the lines $[y=b]$ for
$b\in\{0,\ldots,\ell_2-1\}$
with $X_{b+\ell_2}=\emptyset$, that is,
$(\clDown)_b = Y_{b+\ell_2}+(0,-\ell_2)$ is a horizontal chain.
Second, we pick the left-most element $(A,B)$ of the top row of $X$.
Because of (Aiii), the vertical line $[x=A]$ contains a vertical chain built
from the elements of $X$ and the horizontal lines just filled. 
In particular, filling this vertical line yields
$(A,B-\ell_2)\in\langle\cl\rangle$.

\medskip

Now we may consider $\cl'$ built from $X':=X\setminus\{(A,B)\}$.
This is still an admissible set with $\#X'=\#X-1$. Moreover,
$\cl'\subseteq\cl\cup\{(A,B-\ell_2)\}\subseteq\langle\cl\rangle$,
hence $\Z^2=\langle\cl'\rangle\subseteq\langle\cl\rangle$.
}
\end{proof}

\medskip

\section{Maximal exceptional sequences 
in the twisted case}
\label{sec:MES}
For the rest of this article we assume 
$c\not=0$ and set out to prove the main 
theorems \thLex-{\thFull} in the subsequent sections.
We start with some examples.

\subsection{Maximal exceptional sequences 
with $\mathbf{\ell_1=2}$}\label{subsec:ExamL1}
The integral depth was defined as
$$
\hZ=-\Big\lfloor{\frac{\ell_1-2}{\xa}}\Big\rfloor,
$$
cf.\ \ssect{subsec:RelVar}. 
If $\ell_1=2$, then $\hZ=0$. Therefore, the 
{co-immaculate} locus of $(2,\ell_2;c)$ 
is given by
$$
\nImm(2,\ell_2;\xa,\xb)=\cH\cup
\{0,\,(1,0),\,(-\xb+1,\ell_2)\}
$$
and the variety is orderable. 
Hence, the only nontrivial family of 
{\mes}s up to shifts is given by
$$
\cl=\{0,\,(a_1,1),\,(a_1+1,1),\ldots,
(a_{\ell_2-1},\,\ell_2-1),\,
(a_{\ell_2-1}+1,\ell_2-1),(-\xb+1,\ell_2)\}
$$
with $a_i\in\Z$. We have $H(\cl)=\ell_2+1$.
In the language of admissible sets and 
complementing partners which will be developed 
for the twisted case in 
\ssect{subsec:mesClassification},
we have 
$$
X=\{(-\xb+1,\ell_2)\}\;\mbox{and}\; 
Y=\{(-\xb,\ell_2),\,(a_1-\xb,1+\ell_2),\ldots,
(a_{\ell_2-1}+1-\xb,\,2\ell_2-1)\}.
$$

\medskip

For instance, {we obtain the family 
of Hirzebruch surfaces 
$\CH_{\xa}=\big(2,2;(0,-\xa)\big)$ by setting $\ell_2=2$}. 
Among these, $\CH_1=\PP\big(\CO_{\PP^1}\oplus
\CO_{\PP^1}(1)\big)$ is the only Fano variety. 
Figure~\ref{fig:Hirzebruch} represents 
$\nImm(\CH_{\xa})=\nImm(2,2;\xa,\xa)$ for 
$\xa=1$, $2$ and $3$.
\begin{figure}
\newcommand{\scaleA}{0.4}
\newcommand{\spaceA}{\hspace*{3em}}
\begin{tikzpicture}[scale=\scaleA]
\immreg{2}{2}{1}{1}
\end{tikzpicture}
\spaceA
\begin{tikzpicture}[scale=\scaleA]
\immreg{2}{2}{2}{2}
\end{tikzpicture}
\spaceA
\begin{tikzpicture}[scale=\scaleA]
\immreg{2}{2}{3}{3}
\end{tikzpicture}
\caption{The Hirzebruch surfaces $\CH_1$, 
$\CH_2$ and $\CH_3$.}
\label{fig:Hirzebruch}
\end{figure}

\medskip

\subsection{Maximal exceptional sequences 
with $\ell_1=\ell_2=3$}\label{subsec:Exam33}
Since $\hZ\geq -1$ any variety $(3,3;c)$ 
is necessarily vertically orderable. In 
order to determine the nontrivial {\mes}s 
$\cl$ up to shift we may therefore assume 
that $\cl$ is vertically lexicographically 
ordered. We distinguish two cases:

\medskip

{\em Case 1:} $\xa\geq2$. Then $\hZ=0$. 
As in \ssect{subsec:ExamL1} we find the 
{\mes}s
\begin{align*}
\cl^1=\big(& 0,\,(1,0),\,(a_1,1),
\,(a_1+1,1),\,(a_1+2,1),\\
&(a_2,2),\,(a_2+1,2),\,(a_2+2,2),
\,(-\xb+2,3)\big)\\\cl^2=\big(& 0,\,(a_1,1),
\,(a_1+1,1),\,(a_1+2,1),\\
&(a_2,2),\,(a_2+1,2),\,(a_2+2,2),\,(-\xb+1,3),
\,(-\xb+2,3)\big)
\end{align*}
for $a_1$, $a_2\in\Z$. We have $H(\cl^1)=H(\cl^2)=4$.

\medskip

{\em Case 2:} $\xa=1$. By the basic 
inequality~\eqref{eq:EffIneq}, $\xb=1$ or $2$. 
In addition to the {\mes}s from the previous 
case we find
\begin{align*}
\cl^3=\big(&0,\,(-1,1),\,(0,1),\,(a_2,2),
\,(a_2+1,2),\,(a_2+2,2),\\
&(-\xb+1,3),\,(-\xb+2,3),\,(-\xb+1,4)\big)
\end{align*}
for $a_2\in\Z$. We have $H(\cl^3)=5$.

\medskip

Figure~\ref{fig:ExamXY} displays the admissible 
sets and complementing partners for $\xb=2$.

\begin{figure}[ht]
\newcommand{\spaceA}{\hspace*{1.5em}}
\newcommand{\scaleA}{0.3}
\begin{tikzpicture}[scale=\scaleA]
\draw[color=oliwkowy!40] (-4.3,-1.3) grid (3.3,6.3);
\fill[thick, color=origin]
    (0,0) circle (9pt);
\fill[thick, color=blue]
    (-2,3) circle (9pt) (-1,3) circle (9pt)  
    (-3,4) circle (9pt) (-2,4) circle (9pt) 
    (-1,4) circle (9pt) (0,5) circle (9pt)
    (1,5) circle (9pt) (2,5) circle (9pt);
\fill[thick, color=red]
    (0,3) circle (9pt);  
\draw[thick, color=black]
    (-0.5,-3) node{(a)};
\end{tikzpicture}
\spaceA
\begin{tikzpicture}[scale=\scaleA]
\draw[color=oliwkowy!40] (-4.3,-1.3) grid (3.3,6.3);
\fill[thick, color=origin]
    (0,0) circle (9pt);
\fill[thick, color=blue]
    (-2,3) circle (9pt)   
    (-3,4) circle (9pt) (-2,4) circle (9pt) 
    (-1,4) circle (9pt) (0,5) circle (9pt)
    (1,5) circle (9pt) (2,5) circle (9pt);
\fill[thick, color=red]
    (0,3) circle (9pt) (-1,3) circle (9pt);
\draw[thick, color=black]
    (-0.5,-3) node{(b)};
\end{tikzpicture}
\spaceA
\begin{tikzpicture}[scale=\scaleA]
\draw[color=oliwkowy!40] (-4.3,-1.3) grid (3.3,6.3);
\fill[thick, color=origin]
    (0,0) circle (9pt);
\fill[thick, color=blue]
    (-2,3) circle (9pt) (-3,4) circle (9pt) 
    (-2,4) circle (9pt) (0,5) circle (9pt) 
    (1,5) circle (9pt) (2,5) circle (9pt);
\fill[thick, color=red]
    (-1,3) circle (9pt) (0,3) circle (9pt)
    (-1,4) circle (9pt);
\draw[thick, color=black]
    (-0.5,-3) node{(c)};
\end{tikzpicture}
\caption{The admissible sets (red) and 
complementing partners (blue) for 
(a) $\cl^1$ (b) $\cl^2$ and 
(c) $\cl^3$ with $a_1=-1$, $a_2=2$ 
and $\xb=2$. The origin is marked in
green.}
\label{fig:ExamXY}
\end{figure}

\medskip

\subsection{Maximal exceptional sequences 
are always vertically orderable}
\label{subsec:MESRel}
General {\nex} sequences might not be 
vertically orderable.

\begin{example}
\label{exam:NexNmax}
On $(4,2;(0,1))$ consider the exceptional 
sequence
$$
\cl=\big(0,\,(3,-2),\,(2,-1),\,(3,-1),
\,(3,0),\,(3,1),\,(4,1)\big),
$$
cf.\ Figure~\ref{fig:SharpEstimate}.
\begin{figure}[ht]
\begin{tikzpicture}[scale=0.4]
\immreg{4}{2}{1}{1}
\foreach \y in {-1,0,1} {
  \draw[thick, color=red]
  (3,\y) circle (7pt); }
\draw[thick, color=red]
  (4,1) circle (7pt); 
\draw[thick, color=red]
  (2,-1) circle (7pt);
\fill[thick, color=origin]
  (0,0) circle (5pt);
\draw[thick, color=red]
  (0,0) circle (7pt);
\draw[thick, color=red]
  (3,-2) circle (7pt);
\draw[thick, color=black]
  (3.6,-2.5) node{$\cl_1$};
\end{tikzpicture}
\caption{The points encircled in 
red define the sequence $\cl$ in 
$\nImm(4,2;1,1)$.}
\label{fig:SharpEstimate}
\end{figure}
Since $\cl_1-\cl_0=(3,-2)$ the 
order of these two elements cannot 
be switched, cf.\ Proposition
~\ref{prop:RelCrit}. 

\medskip

On the other hand, the sequence 
$\cl$ is not maximal for 
$\sharp\cl=7<8=\ell_1\ell_2$ and eventually 
not extendable. An easy computation 
shows that $\nImm\cap\big((3,-2)+\nImm\big)$,
the set of possible common successors of 
$\cl_0$ and $\cl_1$, is
$$
\{(2,-1),\,(3,-1),\,(3,0),\,(3,1),\,(4,1)\}.
$$
Alternatively, one may apply
Theorem~\ref{thm:MESRel} below
to show that $\cl$ is the largest 
choice for extending $\big(0,\,(3,-2)\big)$ 
to an exceptional sequence.
\end{example}

\begin{theorem}[Theorem {\thLex}, twisted version]
\label{thm:MESRel}
Let $\cl=(\cl_0,\ldots,\cl_{\ell_1\ell_2-1})$ be 
a {\mes} on $(\ell_1,\ell_2;c)$, $c\not=0$. 
Then $\cl$ is vertically orderable.
\end{theorem}

\begin{proof}
Assume to the contrary that $\cl$ is a maximal 
exceptional sequence which is not vertically 
orderable. 

\medskip

By Lemma~\ref{lem:ReorderJump} we may assume 
that there exists a pair $\cl_i<\cl_{i+1}$
with $b_i-b_{i+1}\geq\ell_2$. Furthermore, upon
applying $i$-times the helix operator from 
\ssect{subsec:helixOp} we may replace $\cl$
by a sequence where $i=0$. Choosing suitable
coordinates we therefore suppose
without loss of generality that
$$
\cl_0=(0,0),\quad\cl_1=(\kappa,\lambda)
$$
for $\lambda\leq-\ell_2$ and 
$0<-\lambda\xa<\kappa<\ell_1$. 

\medskip

By Lemma~\ref{lem:PhiMap}, there exists precisely 
one $i_0>0$ such that $\Phi(\cl_{i_0})=[1,0]$. 
Hence, $\cl_{i_0}\in\big((1,0)+L\big)\cap\nImm$. 
Reasoning as in the proof of Lemma~\ref{lem:LImm} 
{shows that} 
$$
\cl_{i_0}=(n\ell_1-m\xb+1,m\ell_2)\in\cP. 
$$
{Comparing with the proof of
Lemma~\ref{lem:LImm} we have added $1$ in the first argument. 
Instead of a contradiction, we now obtain a unique solution:
$m=n=0$ providing $\cl_{i_0}=(1,0)$ for $m\leq0$ and 
$m=1$, $n=0$ providing $\cl_{i_0}=(-\xb+1,\ell_2)$ for $m\geq1$.}

\medskip

On the other hand, $\cl_{i_0}$ is equal to or 
a successor of $\cl_1=(\kappa,\lambda)$, hence
\begin{equation}\label{eq:SI0}
\cl_{i_0}\in\big((\kappa,\lambda)+\nImm\big)
\cup\{(\kappa,\lambda)\}.
\end{equation}

\medskip

However, $(1,0)$, 
$(-\xb+1,\ell_2)\not\in(\kappa,\lambda)+\nImm$, 
contradicting~\eqref{eq:SI0}. Indeed, we have 
$\kappa>-\lambda\xa\geq\xa\ell_2>\xb$.
Therefore, a point 
$(a,b)\in\Z^2$ with $a\leq1$ and $b\geq0$ cannot 
lie in $(\kappa,\lambda)+\nImm$ as
$$
a-\kappa\leq1-\kappa\leq-\xa\ell_2<-\xb,
$$
while $b-\lambda\geq\ell_2$ implies 
$-\xb<a-\kappa$. 
\end{proof}

\medskip

\section{Vertical {\wbox}s}
\label{sec:HCVB}
{We continue with the twisted case $c\neq 0$.
In particular, $\xa,\xb,\xc\geq 1$.}
\subsection{Replacing vertical chains by 
ensembles}
\label{subsec:vertBoxes}
Rather than to vertical chains, the $\kV$-sequence 
from Corollary~\ref{coro:resol} gives rise to 
a complicated shape for the involved locus 
in $\Pic X=\Z^2$. The sheaves
$\CF_k$ occurring in \ssect{subsec:appPicTwo}
suggest the following  

\begin{definition}
\label{def:ShapeBox}
We denote by 
$$
V=V(\ell_2;\xa,\xb)\subseteq\Z^2
$$ 
the set of lattice points 
$\sum_{j\in J}(f^j,1)=(\sum_{j\in J}f^j,\sharp J)$
where $J\subseteq\{1,\ldots,\ell_2\}$ is an arbitrary
subset of cardinality $1\leq\sharp J\leq\ell_2-1$ and 
$f=(f^1,f^2,\ldots,f^{\ell_2})$ runs 
through all integral $\ell_2$-tuples satisfying 
$$
0=f^1\geq f^2\geq\ldots\geq f^{\ell_2-1}\geq f^{\ell_2}=
-\xa\quad\mbox{and}
\hspace{0.5em}\sum_{j=1}^{\ell_2}f^j=-\xb.
$$

\medskip

We call the subset $V+(a,b)$ the 
{\bf V-{\wbox} based at $(a,b)$}.
\end{definition}

Note that $V$ fits into the rectangular 
box bounded by $-\xb\leq x\leq 0$ and $0\leq y\leq \ell_2$ 
and whose diagonal of negative slope ends in $(0,0)$ and 
$(-\xb,\ell_2)$. {In particular}, $V\subseteq\cH$. Further, 
in the extreme case $\xa=\xb$ which necessarily implies 
$c^1=c^2=\ldots=c^{\ell_2-1}=0$, the V-{\wbox} degenerates 
to 
$$
V(\ell_2;\xb,\xb)=\{(-\xb,b),\,(0,b)\mid b=1,\ldots,\ell_2-1\}.
$$
See also Figure~\ref{fig:VBox} for an illustration.

\begin{figure}[ht]
\begin{tikzpicture}[scale=0.5]
\draw[color=oliwkowy!40] (-3.3,-3.3) grid (3.3,3.3);
\fill[pattern color=darkgreen!100, pattern=north west lines]
 (2,-2) -- (-2,0) -- (-2,2) -- (2,0) -- cycle;
\draw[thick, color=redSeg]
 (2,-2)--(2,2)--(-2,2)--(-2,-2)--(2,-2);
\fill[thick, color=origin]
    (2,-2) circle (5pt) (-2,2) circle (5pt);
\draw[thick, color=origin]
    (2,-2) circle (7pt) (-2,2) circle (7pt);
\draw[thick, color=origin]
  (2.4,-2.7) node{$0$};
\draw[thick, color=origin]
  (-2.5,2.7) node{$(-\xb,\ell_2)$};
\foreach \y in {1,2,3}
    {
    \foreach \x in {1,2,3}
        {
        \fill[thick, color=redSeg]
        (\x-\y,\y-2) circle (5pt);
        }
    }
\fill[thick, color=red]
  (-2,0) circle (5pt) (2,0) circle (5pt);
\draw[color=oliwkowy!40] (5.3,-3.3) grid (12.3,3.3);
\fill[pattern color=darkgreen!100, pattern=north west lines]
 (11,-2) -- (7,-1) -- (7,2) -- (11,1) -- cycle;
\draw[thick, color=redSeg]
 (11,-2)--(11,2)--(7,2)--(7,-2)--(11,-2);
\fill[thick, color=origin]
    (11,-2) circle (5pt) (7,2) circle (5pt);
\draw[thick, color=origin]
    (11,-2) circle (7pt) (7,2) circle (7pt);
\draw[thick, color=origin]
  (11.4,-2.7) node{$0$};
\draw[thick, color=origin]
  (6.5,2.7) node{$(-\xb,\ell_2)$};
\foreach \y in {1,2,3}
    {
    \fill[thick, color=redSeg]
    (11,\y-2) circle (5pt) (7,\y-2) circle (5pt);
    }
\end{tikzpicture}
\caption{The sets $V(4;2,4)$ and $V(4;4,4)$. On the left hand 
side the dark red points come from the vector $c=(0,-1-1,-2)$. 
Additional points from $c=(0,0,-2,-2)$ which are not dark red 
yet are marked in light red.}
\label{fig:VBox}
\end{figure}

\medskip

We turn to existence of V-{\wbox}s inside exceptional 
sequences next. To simplify the shape we introduce the
larger {\bf $W$-{\wbox} 
of $(\ell_1,\ell_2;c)$} by 
$$
W=W(\ell_2;\xa):=\{(a,b)\in\Z^2\mid0<b<\ell_2,\,-b\xa\leq a\leq0\}.
$$

\begin{lemma}
\label{lem:comparisonVW}
$\;V(\ell_2;\xa,\xb)$ is symmetric under the transformation
$(x,y)\mapsto (-\xb,\ell_2)-(x,y)$. Moreover, $V(\ell_2;\xa,\xb)
\subseteq W(\ell_2;\xa)$. 
\end{lemma}

\begin{proof}
For the symmetry note that
$\sum_{j=1}^{\ell_2}(f^j,1)=(-\xb,\ell_2)$. {Therefore,}
$$
(-\xb,\ell_2)-\sum_{j\in J}(f^j,1)=\sum_{i\in I}(f^i,1)
$$ 
with $I=\{1,\ldots,\ell_2\}\setminus J$.

\medskip

Further, $(a,b)\in V$ implies 
$(a,b)=(\sum_{j\in J_b}f^j,b)$
for some suitable $0\geq f^j\geq-\xa$ and 
$\sharp J_b=b$. Hence $a\geq-b\xa$. 
\end{proof}

Let us define
\noGammaText{
\begin{align*}
\ko{\cP}_V&\phantom{:}=
\{(a,b)\in\Z^2\mid  -\xb\leq a\leq 0
\hspace{0.5em}\mbox{and}\hspace{0.5em}
0\leq\langle(a,b),\,(1,\xa)\rangle \leq -\xb+\xa\ell_2\}
\end{align*}
}
\gammaText{
\begin{align*}
\ko{\cP}_V&\phantom{:}=
\{(a,b)\in\Z^2\mid  -\xb\leq a\leq 0
\hspace{0.5em}\mbox{and}\hspace{0.5em}
0\leq\langle(a,b),\,(1,\xa)\rangle \leq \xc\}
\end{align*}
}
\!\!as a smaller (and closed) version 
of the {co-immaculate} parallelogram
$\cP$ from \ssect{subsec:ImmLoc}, 
cf.~Figure~\ref{fig:scenarioHighSeq}.
The $(-\xb,\ell_2)$-symmetry of $V$
implies the following observation.

\begin{figure}[ht]
\newcommand{\scaleA}{0.3}
\newcommand{\spaceA}{\hspace*{2em}}
\begin{tikzpicture}[scale=\scaleA]
\def\la{16} %
\def\lb{7} %
\def\a{2} %
\def\b{6} %
\tikzmath{real \slopeBox; \slopeBox=\b/\a;}
\tikzmath{real \heightBox; \heightBox=\lb-\slopeBox;}
\tikzmath{real \slopeTriangle; \slopeTriangle=\la/\a;}
\tikzmath{real \WBoxLeft; \WBoxLeft=\b*\lb/\slopeBox;}
\fill[color=blue!3]
 (-\b,\lb-\heightBox) -- (0,0) -- (\la,-\slopeTriangle) -- (\la,\heightBox)
   -- (-\b,\lb+\slopeTriangle) -- (-\b,\lb) -- cycle;
\shade[left color=white!80!yellow, right color=white]
 (-\b-3,1-0.3) -- (\la+5,1-0.3) -- (\la+5,\lb-1+0.3) -- (-\b-3,\lb-1+0.3)
                                                                  -- cycle;
\draw[color=oliwkowy!10] (-\b-3.3,-\slopeTriangle) grid
(\la+5.3,\lb+\slopeTriangle); %
\draw[thin, color=yellow!50!orange]
 (-\b-3,1) -- (\la+5,1)  (-\b-3,2) -- (\la+5,2) (-\b-3,3) -- (\la+5,3)
 (-\b-3,\lb-1) -- (\la+5,\lb-1)  (-\b-3,\lb-2) -- (\la+5,\lb-2)
 (-\b-3,\lb-3) -- (\la+5,\lb-3);
\fill[pattern color=darkgreen!100, pattern=north west lines]
 (0,0) -- (0,\heightBox) -- (-\b,\lb) -- (-\b,\slopeBox) -- cycle;
\draw[dashed, color=blue]
 (0,0) -- (\la,-\slopeTriangle) -- (\la,\heightBox)
   -- (-\b,\lb+\slopeTriangle) -- (-\b,\lb);
\draw[dashed, color=purple]
 (0,0) -- (\la,0) (-\b,\lb) -- (\la-\b,\lb)
 (-\b,2*\lb-1) -- (\la-\b,2*\lb-1);
\draw[thick, color=black]
 (0,0) -- (0,\heightBox) -- (-\b,\lb) -- (-\b,\slopeBox) -- cycle;
\fill[pattern color=red!100, pattern=north east lines]
 (-\b,\lb) -- (-\b,\slopeBox) -- (-\WBoxLeft,\lb) -- cycle
 (0,\heightBox) -- (0,\lb) -- (-\b,\lb) -- cycle;
\fill[pattern color=lightblue!100, pattern=north east lines]
 (0,0) -- (0,\lb-\slopeBox) -- (-\b+\WBoxLeft,0) -- cycle
 (-\b,\lb-\heightBox) -- (-\b,0) -- (0,0) -- cycle;
\fill[color = purple]
  (\la,0) circle (5pt);
\draw[thick, color=black]
  (-1.0,-1.3) node{$0$} (-\b-3.3,\lb+1.3) node{$(-\xb,\ell_2)$}
  (\la-\b+4.5,\lb) node{$(\ell_1-\xb,\ell_2)$} (\la+3.0,0) node{$(\ell_1,0)$}
  (\la-\b+4.5,2*\lb-1) node{$(2\ell_2-1)$-line}
  (\la+1.5,-6) node{$\cP$} (-0.5*\b,0.5*\lb) node{$\bbox$}
  (-1.8*\b,0.7*\lb) node{$W$}%
  (0.8*\b,0.2*\lb) node{$(-\xb,\ell_2)-W$} %
  (\la+3,0.5*\lb) node{$\CH$};
\fill[color = red]
 (0,0) circle (5pt);
\fill[color = blue]
 (-\b,\lb) circle (5pt) (\la-\b,\lb) circle (5pt);
\end{tikzpicture}
\caption{$\bbox$ and $\cP$ and $\CH$ for
$\nImm(\ell_1,\ell_2,\xa,\xb)=\nImm(16,7;2,6)$}
\label{fig:scenarioHighSeq}
\end{figure}

\begin{lemma}
\label{lem:comparisonPW}
We have
\begin{enumerate}
\item[(i)] $W(\ell_2;\xa)\cap\big((-\xb,\ell_2) - W(\ell_2;\xa)\big)
= \ko{\cP}_V\setminus \big\{0,\;(-\xb,\ell_2)\big\}$.
\vspace{1ex}
\item[(ii)] $V(\ell_2;\xa,\xb)\subseteq \ko{\cP}_V\setminus
\big\{0,\;(-\xb,\ell_2)\big\}$.
\end{enumerate}
\end{lemma}

\begin{remark}
\label{rem:shapeV}
The fact that $f^1=0$ and $f^{\ell_2}=-\xa$ for the sequences 
defining $V$ allows actually an even more
refined description. It turns out that 
$V\cup\big\{0,\;(-\xb,\ell_2)\big\}$ equals the union 
of four smaller $\ko{\cP}_V$-like parallelograms
located at the four corners of the ambient $\ko{\cP}_V$. However, 
the rather coarse relationship %
$V\subseteq W$ from Lemma~\ref{lem:comparisonVW} will be
sufficient for our purposes.
\end{remark}

\subsection{Finding $V$- and $W$-{\wbox}s inside \mes s}
\label{subsec:BoxesAndMES}
Our goal is to {locate} sufficiently 
many $V$-{\wbox}s inside any given {\mes} $\cl$.

\begin{lemma}
\label{lem:InsertWBox}
Let $\cl$ be a vertically ordered exceptional sequence on 
$(\ell_1,\ell_2;c)$ with $\cl_0=0$ and $H(\cl)\geq2\ell_2$. 
Then $\cl\cup W(\ell_2;\alpha)$ is also exceptional with 
respect to the vertical order.
\end{lemma}

\begin{proof}
As suggested by 
Figure~\ref{fig:WBoxTriangle}, 
$(a,b)-W\subseteq\nImm$ for all 
$(a,b)\in\DeltaUp$. Indeed, let $(a',b')\in W$. 
By \ssect{subsec:ImmLoc},
$$
\ell_2\leq b,\quad-\xb<a<
\noGammaText{\ell_1-\xb-(b-\ell_2)\xa}
\gammaText{\ell_1 + \xc -b\xa}.
$$
Since $b'<\ell_2$ it follows that $0<b-b'$. 
If $b-b'<\ell_2$, then $(a-a',b-b')\in\cH$, 
and we are done. If $\ell_2\leq b-b'$, then 
$0\leq-a'\leq b'\xa$ 
implies 
$$
-\xb<a-a'<
\gammaText{\ell_1+\xc-(b-b')\xa}
=\ell_1-\xb-(b-b'-\ell_2)\xa<\ell_1
$$
whence $(a-a',b-b')\in\nImm$.

\begin{figure}[ht]
\begin{tikzpicture}[scale=0.4]
\immreg{4}{3}{1}{1}
\fill[pattern color=red!30, pattern=north west lines]
  (0,0) -- (-2,2) -- (0,2) -- cycle;
\fill[pattern color=red, pattern=north west lines]
  (1,4) -- (1,2) -- (3,2) -- cycle;
\end{tikzpicture}
\caption{$\nImm(4,3;1,1)$ with the $W$-{\wbox} in 
light red and the set $(1,4)-W$ in red.}
\label{fig:WBoxTriangle}
\end{figure}

\medskip

Furthermore, $H(\cl)\geq2\ell_2$ implies 
existence of a point $(a,b)\in\cl$ with 
${b\geq2\ell_2-1}$. Therefore, $(a,b)-W$ and
$(a,b)-(\cl\cap\cH)$ are contained in 
$\DeltaUp$. In virtue of 
Lemma~\ref{lem:HeightConstraint} 
it follows that $|a'-a''|\leq\ell_1-2$ 
for any $(a',b')\in W$ and $(a'',b')\in(\cl\cap\cH)\cup W$. 
Hence $\cl\cup W$ is exceptional.
\end{proof}

\begin{corollary}\label{cor:WBox}
Let $\cl$ be a vertically ordered {\mes} with 
$H(\cl)\geq2\ell_2$. Then $\cl_0+W\subseteq\cl$. 
In particular, $\cl_0+(V\cup\{(0,b)\mid0\leq b<\ell_2\})$
is contained in $\cl$.
\end{corollary}

\begin{proof}
Since $\cl\cup(\cl_0+W)$ is an exceptional extension 
of $\cl$, maximality of $\cl$ implies 
$\cl\cup(\cl_0+W)\subseteq\cl$. The second claim 
follows from
$$
\cl_0+\{(0,b)\mid0\leq b<\ell_2\}\subseteq
(\cl_0+W)\cup\{\cl_0\}
$$
and Lemma~\ref{lem:comparisonVW} which asserts that 
$V\subseteq W$.
\end{proof}

\subsection{Bounding the height}
\label{subsec:BoundHeight}
Our final goal in this section is to establish 
Theorem {\thHeight} in the twisted case. We first define
$$
\cH^-:=\{(a,b)\in\Z^2\mid a<0,\,0<b<\ell_2\}\subseteq\CH.
$$

\begin{lemma}
\label{lem:L1Coll}
Let $\cl$ be a {\mes} on $(\ell_1,\ell_2;c)$ with 
$c\not=0$, $\ell_1\geq3$ and $H(\cl)\geq2\ell_2$. 
If $\cl$ is vertically ordered with $\cl_0=0$, 
then we obtain a sequence $\cl'$ on 
$$
(\ell_1',\ell_2';c')=(\ell_1-1,\ell_2;c)
$$ 
via the following procedure:
\begin{enumerate}
    \item[(i)] Remove the set 
    $\{(0,b)\mid0\leq b<\ell_2\}$ from $\cl$.
    \item[(ii)] For every $\cl_i=(a_i,b_i)\in\CH^-$ 
    put $\cl'_i:=\cl_i+(1,0)$.
    \item[(iii)] For all remaining $\cl_i$ put 
    $\cl_i':=\cl_i$.
\end{enumerate}
If we endow $\cl'$ with the order induced by $\cl$, 
then $\cl'$ defines a {\mes}. We say that $\cl'$ is 
obtained from $\cl$ by {\em collapsing along $\ell_1$}.
\end{lemma}

\begin{proof}
We denote the {co-immaculate} locus of 
$(\ell_1',\ell_2';c')$ by $\nImm'$.
Let $\cl_i'<\cl_j'$ be a pair of elements in $\cl'$ 
coming from $\cl_i=(a_i,b_i)<\cl_j=(a_j,b_j)$ in 
$\cl$. We need to show that $\cl_j'-\cl_i'\in\nImm'$. 
Figure~\ref{fig:L1Collapse} sketches how the 
{co-immaculate} locus adjusts $\nImm'$ under passing 
from $(\ell_1,\ell_2;c)$ to $(\ell_1-1,\ell_2;c)$.

\begin{figure}[ht]
\begin{tikzpicture}[scale=0.4]
\draw[color=oliwkowy!40] (-1.3,-5.0) grid (5.3,8.0);
\immreg{4}{3}{1}{1}
\draw[thick, color=red]
  (-1,3) -- (-1,7);
\draw[thick, color=red]
  (-1,7) -- (3,3);
\draw[thick,black,dotted]
  (4,-4) -- (5,-5) -- (5,0);
\draw[thick,black,dotted]
  (4,3) -- (-1,8) -- (-1,7);
\fill[pattern color=parallelogram!50, pattern=north west lines]
  (-1,7) -- (-1,8) -- (5,2) -- (5,-5) -- (4,-4) -- (4,2) -- cycle;
\fill[thick, color=regii!50]
    (0,6) circle (5pt) (1,5) circle (5pt) (2,4) circle (5pt) 
    (3,3) circle (5pt);
\fill[thick, color=regii!50]
    (4,0) circle (5pt) (4,-1) circle (5pt) (4,-2) circle (5pt) 
    (4,-3) circle (5pt);
\draw[thick, color=origin]
    (4,3) circle (5pt);
\end{tikzpicture}
\caption{$\nImm(5,3;1,1)$ and $\nImm'=\nImm(4,3;1,1)$. 
The additional points in $\nImm(5,3;1,1)$ are marked in 
light blue. The upper left boundary and
the upper right boundary are marked in red.}
\label{fig:L1Collapse}
\end{figure}

\medskip

{\em Case 1:} $\cl_i,\,\cl_j\not\in\CH^-$. Then $\cl_j'-\cl_i'=
\cl_j-\cl_i$. If $\ell_2>b_j-b_i$, then $\cl_j'-\cl_i'\in\CH=\CH'$ 
or $b_j=b_i$. In the latter case we have $a_j-a_i<\ell_1-1$ for 
otherwise, $\cl$ would have a horizontal chain and thus 
$H(\cl)\leq2\ell_2-1$ by Corollary~\ref{coro:SlimTwisted}.

\medskip

If, on the other hand, $b_j-b_i\geq\ell_2$, then $\cl_j'-\cl_i'
\in\nImm'$ unless
$$
a_j-a_i+(b_j-b_i)\xa=\ell_1-1-\xb+\ell_2\xa=\ell_1'-\xb+\ell_2\xa,
$$
that is, $\cl_j-\cl_i$ sits in the upper right boundary of $\CP'$,
cf.\ Figure~\ref{fig:L1Collapse}. Since $\ell_2\leq b_j$ we have 
$a_j+b_j\xa\leq\ell_1-1-\xb+\ell_2\xa$. But $0< a_i+b_i\xa$ -- 
this is clear for $\cl_i\in\cH$ and follows from the defining 
inequalities of the parallelogram if $\cl_i\in\cP$. Hence
$$
a_j-a_i+(b_j-b_i)\xa<a_j+b_j\xa\leq\ell_1-1-\xb+\ell_2\xa,
$$
contradiction.

\medskip

{\em Case 2:} $\cl_i,\,\cl_j\in\CH^-$. Again $\cl_j'-\cl_i'=
\cl_j-\cl_i$, and $\cl_j-\cl_i\in\CH$ or 
$b_j=b_i$. As for Case~1 we find that $\cl_j'-\cl_i'\in\CH'
\subseteq\nImm'$.

\medskip

{\em Case 3:} $\cl_i\in\CH^-$, $\cl_j\not\in\CH^-$. Then 
$\cl_j'-\cl_i'=\cl_j-\cl_i-(1,0)$. If $\ell_2>b_j-b_i$, we 
conclude as in Case 1. 

\medskip

If, on the other hand, $b_j-b_i\geq\ell_2$, then $\cl_j'-\cl_i'
\in\nImm'$ unless
$$
\cl_j-\cl_i=(a_j-a_i,b_j-b_i)=(-\xb+1,b_j-b_i),
$$ 
that is, $\cl_j-\cl_i$ sits in the upper left boundary of $\CP'$,
cf.~Figure~\ref{fig:L1Collapse}. Now $\ell_2<b_j$ implies
$a_j-a_i=-\xb+1\leq a_j$, but $a_i<0$.

\medskip

{\em Case 4:} $\cl_i\not\in\CH^-$, $\cl_j\in\CH^-$. Then 
$\cl_j'-\cl_i'=\cl_j+(1,0)-\cl_i$ and $\ell_2>b_j>b_i\geq0$,
where the middle inequality follows from $a_j<a_i$. 
We conclude again as in Case~1. 
\end{proof}

\begin{corollary}[Theorem {\thHeight}, twisted version]
\label{cor:HeightConstraint}
Let $\cl$ be a {\mes}. Then $H(\cl)\leq2\ell_2$. 
\end{corollary}

\begin{proof}
We assume that $\cl$ is vertically ordered and starts at 
$\cl_0=0$. We proceed by induction on $\ell_1\geq2$.

\medskip

If $\ell_1=2$, then $H(\cl)\leq\ell_2+1<2\ell_2$, 
cf.\ \ssect{subsec:ExamL1}.

\medskip

Next assume that $\ell_1\geq3$. Let $\cl$ be a {\mes} with 
$H(\cl)>2\ell_2$. By Lemma~\ref{lem:L1Coll} we can 
collapse $\cl$ along $\ell_1$ and obtain the {\mes} $\cl'$ 
on $(\ell_1-1,\ell_2;c)$ with $2\ell_2\leq H(\cl)-1\leq H(\cl')$.
Here, the latter inequality is a consequence of 
$W\subseteq\cl$ from Corollary~\ref{cor:WBox}. 
Further, $H(\cl')\leq2\ell_2$ by induction hypothesis so that 
$H(\cl')=2\ell_2<H(\cl)$. In particular, the collapsed 
sequence $\cl'$ starts at $\cl_0'=(a,1)$ with $a\leq0$.

\medskip

By Corollary~\ref{cor:WBox}, $\cl_0'+W=(a,1)+W\subseteq\cl'$. 
Therefore
$$
(a,1)+\big(-(\ell_2-1)\xa,\ell_2-1\big)=\big(-(\ell_2-1)\xa+a,\ell_2\big)\in\cl'.
$$
By design of the collapsing procedure, $(-(\ell_2-1)\xa+a,\ell_2)$ 
is also in $\cl$ whence 
$$
-\xb<-(\ell_2-1)\xa+a\leq-(\ell_2-1)\xa.
$$
But this contradicts the basic inequality~\eqref{eq:EffIneq}.
\end{proof}

\medskip

\section{The classification for the twisted case}
\label{sec:completeClass}
We now discuss the twisted analogues of
heLexing (cf.\ \ssect{subsec:HeLexPC}) and 
the structure of {\mes}s (cf.\ \ssect{subsec:ClassProdCase}).

\medskip

\subsection{HeLexing}
\label{subsec:HeLexing}
{Let $\cl$ be a \mes\ starting at 
the origin which by Theorem~\ref{thm:MESRel} we may 
take to be vertically ordered.} 
The helix operator $\helix$ sends $\cl_0$, the 
leftmost element of the lowest row, to the point
$(\ell_1-\xb,\ell_2)$ at level $\ell_2$,
using the terminology of \ssect{subsec:ExChain}.

\medskip

In \ssect{subsec:HeLexPC} we considered
$\heLex$ which was the helix operator $\helix$ followed by 
vertically lexicographic reordering and a shift sending the 
resulting $\cl_0$ back to the origin. The proof of
Proposition~\ref{thm:HeLexingPC} applies verbatim 
and yields the

\begin{theorem}[Theorem~\thHelex, twisted version]
\label{thm:HeLexing}
Every \mes\ on $(\ell_1,\ell_2;c)$, $c\not=0$, can be 
transformed into a trivial sequence by successively applying 
$\heLex$ at most $\ell_1\ell_2$ times.
\end{theorem}

\medskip

\subsection{The classification}
\label{subsec:mesClassification}
Again we can establish an algorithmic recipe 
for the construction of \mes s.

\medskip

The definition of admissible sets and 
complementing partners carries over from 
\ssect{subsec:ClassProdCase} except for 
{the modified shape of $\DeltaUp$, namely
$$
\DeltaUp=\{(a,b)\in\cP\kst b\geq\ell_2\}
=\{(a,b)\in\Z^2\mid b\geq\ell_2,\; -\xb<a
\mbox{ and }a+b\xa<\ell_1+\xc\}
$$
and}
(Aiii) which gets replaced by
\begin{enumerate}
\item[(Aiii')]
for each $k\geq\ell_2$ and $(x,k+1)\in X_{k+1}$ 
the points $(x,k),\,(x+1,k),\,\ldots,\,(x+\xa,k)$
belong to $X_k$.
\end{enumerate}
See \ssects{subsec:ExamL1} and 
~\eqref{subsec:Exam33} for examples. 
{Note that $X_{\ell_2}$ being right-aligned 
means now that $(\ell_1-1-\xb,\ell_2)\in X_{\ell_2}$.}
Then we obtain the

\begin{theorem}[Theorem {\thClass}, twisted version]
\label{thm:mesAfterAdShift}
If $X\subseteq \DeltaUp$ is 
admissible and $Y\subseteq \Z^2$ a 
complementing partner, then the union of
$$
\clDown:=Y+(\xb,-\ell_2)\subseteq\cH\cup[y=0] 
\quad\mbox{and}\quad 
\clUp:=X\subseteq\DeltaUp
$$
together with vertical order yields a 
{\mes} $\cl$ with $\cl_0=0$. {Moreover, any
vertically ordered {\mes} starting at the origin 
arises this way.}
\end{theorem}

\begin{proof}
The proof goes along the lines of the
proof of Theorem~\ref{thm:mesAfterAdShiftPC}.

\medskip

If $X=\varnothing$, then $Y$ consists of 
$\ell_2$ consecutive horizontal chains.
Shifting down by $(\xb,-\ell_2)$ yields 
a vertically trivial sequence starting 
at the origin. 

\medskip

We therefore assume that $X\not=\varnothing$.
We order the set $\clDown\cup\clUp$ 
vertically to obtain the sequence $\cl$. 
From Definition~\ref{def:HorChain},
Definition~\ref{def:AdmSetPC}~(iv), and
Definition~\ref{def:SuppPartnerPC}~(Cvii), 
it is clear that $Y_{\ell_2}\cup X_{\ell_2}$ 
forms a horizontal chain which ends at 
$(\ell_1-1-\xb,\ell_2)$. Consequently, 
$Y_{\ell_2}$ starts at $(-\xb,\ell_2)$ 
and $\cl_0=0$.

\medskip

For $\cl_i<\cl_j$ we have to show that 
$\cl_j-\cl_i\in\nImm$. 
If $\cl_i$, 
$\cl_j\in[y=b]$, this follows from
$\cl_i,\cl_j\in Y_{b+\ell_2}+(\xb,-\ell_2)$
or $\cl_i,\cl_j\in X_b$, (Cvi) 
and (Cvii). If they are at different levels 
$0\leq b_i<b_j\leq 2\ell_2-1$, the only
critical case arises from $b_j-b_i\geq \ell_2$
which implies $0\leq b_i\leq\ell_2-1$ and 
$\ell_2\leq b_j\leq2\ell_2-1$. This means that
$$
\cl_i\in[y=b_i]\cap\clDown=Y_{b_i+\ell_2}+(\xb,-\ell_2)
\quad\mbox{and}\quad\cl_j\in[y=b_j]\cap\clUp=X_{b_j}.
$$
We proceed via induction over $m:=b_j-b_i-
\ell_2\geq 0$. 

\medskip

If $m=0$, then $\cl_i\in Y_{b_j}+(\xb,-\ell_2)$ 
whence
$$
\cl_j-\big(\cl_i+(-\xb,\ell_2)\big)\in X_{b_j}-Y_{b_j}\subseteq
\{(1,0),\,\ldots, (\ell_1-1,0)\}.
$$
In particular,  
$\cl_j-\cl_i\in (-\xb,\ell_2)+\{(1,0),\,\ldots,
(\ell_1-1,0)\}\subseteq\DeltaUp$.

\medskip

Next let $m\geq 1$. Writing 
$B:=\{(0,-1),(1,-1),\ldots,(\xa,-1)\}$,
(Aiii') implies that
$\cl_j+B\in\clUp$. On the other hand,
the induction hypothesis implies $\cl_j-\cl_i+B
\in\nImm$, and since $b_j-1-b_i-\ell_2\geq0$
we even have $\cl_j-\cl_i+B\in\DeltaUp$.
By definition of $\DeltaUp$ we also have
$\cl_j-\cl_i\in\DeltaUp$, whence $\cl_j-\cl_i
\in\nImm$.

\medskip

Finally, we want to show that any vertically
ordered {\mes} $\cl$ with $\cl_0=0$ arises 
this way. For this, we let
$$
Y:=(-\xb,\ell_2)+\big(\cl\cap(\cH\cup[y=0])\big)
\quad\mbox{and}\quad X:=\cl\cap\DeltaUp,
$$ 
where we identify the sequence $\cl$ with its
underlying set. 

\medskip

If $\cl$ is vertically trivial, 
then $Y=(-\xb,\ell_2)+\cl$ and $X=\varnothing$. 
We therefore assume that $\cl$ is not trivial
and check that $X$ is admissible with complementing
partner $Y$. By Lemma~\ref{lem:HeightConstraint},
$H(\cl)\leq2\ell_2$. From this and the definition 
of $Y$, Properties (Ai) and (Cv) follow. Furthermore, 
(Aii) follows from Corollary~\ref{coro:NoGaps}, while 
(Cvii) is a consequence of Proposition~\ref{prop:CompShape}.

\medskip

By Corollary~\ref{coro:VertOcc}, $[y=b]\cap\cl\not=
\varnothing$ for $0\leq b\leq H(\cl)-1$. Furthermore,
we have $\nl=2\ell_2-H(\cl)$ horizontal lines by 
Corollary~\ref{coro:ExHorChains}. 

\medskip

{From \ssect{subsec:ClassProdCase} and in particular 
Inequality~\eqref{eq:rb} on Page~\pageref{eq:rb} we recall 
the following notation:} For $b=0,\ldots,
\ell_2-\nl -1$, we let $(a_b,b)$ and $(a_b+r_b,b)\in\cl$ 
be the minimal and maximal element of $[y=b]\cap\cl$,
that is, $[y=b]\cap\cl=\{(a_b,b),\ldots,
(a_b+r_b,b)\}$ for some $0\leq r_b<\ell_1-1$. {It follows that} $(r_0+1-\xb,\ell_2),\ldots,(\ell_1-1-\xb,\ell_2)
=[y=\ell_2]\cap\cl$ which implies (Aiv). 

\medskip

For all other $0<b<\ell_2-\nl -1$, 
$Y_{b+\ell_2}=(-\xb,\ell_2)+[y=b]\cap\cl$ is to the 
left, cf.\ Remark~\ref{rem:LR}. This implies (Cvi).

\medskip

It remains to check (Aiii'). Let $(x,b+\ell_2+1)\in 
X_{b+\ell_2+1}$ for some $0\leq b$. We need to show 
that $(x,b+\ell_2)$ and $(x+\xa,b+\ell_2)$ belong to 
$\cl$. We first note that $X_{b+\ell_2}\neq\varnothing$ 
for otherwise, $Y_{b+\ell_2}$ and thus
$[y=b]\cap\cl$ would consist of $\ell_1$ 
consecutive points. Hence $X_{b+\ell_2+1}$ 
would be empty by Corollary~\ref{coro:SlimTwisted}, 
which is absurd.

\medskip

Next, we show that $(x+\xa,b+\ell_2)\in X_{b+\ell_2}$. 
Assume otherwise. Since $X_{b+\ell_2}\neq\varnothing$ 
and $\sharp(Y_{b+\ell_2}\cup X_{b+\ell_2})=\ell_1$ 
this would imply $(x+\xa-\ell_1,b+\ell_2)\in Y_{b+\ell_2}$, 
or equivalently, $(x+\xa-\ell_1+\xb,b)\in\cl$. {This implies 
that}
$$
(x,b+\ell_2+1)-(x+\xa-\ell_1+\xb,b)=
(\ell_1-\xa-\xb,\ell_2+1)\in\DeltaUp.
$$
However, the rightmost element of $\DeltaUp$ is 
$(\ell_1-\xa-1-\xb,\ell_2+1)$ whence again 
a contradiction.

\medskip

Finally, we show that $(x,b+\ell_2)\in X_{b+\ell_2}$.
Again, assume otherwise. Since $(x+\xa,b+\ell_2)
\in X_{b+\ell_2}$ 
and $\sharp(Y_{b+\ell_2}\cup X_{b+\ell_2})=\ell_1$, 
this would imply $(x,b+\ell_2)\in Y_{b+\ell_2}$ 
whence $(x+\xb, b)\in\cl$.
However, this means that
$$
(x,b+\ell_2+1)-(x+\xb,b)=(-\xb,\ell_2+1)\in\DeltaUp.
$$
But $[x=-\xb]\cap\DeltaUp$ lies in the boundary of
$\DeltaUp$ whence a contradiction.
\end{proof}

{As} $(x,k+1)\in X_{k+1}$
implies $(x,k)\in X_k$ by (Aiii'), 
Theorem~\ref{thm:mesAfterAdShift}
immediately yields the following 

\begin{corollary}
\label{coro:ExtWBox}
On $(\ell_1,\ell_2;c)$ let $\cl$ be a vertically ordered {\mes} 
with $\cl_0=0$, and let $\nl=2\ell_2-H(\cl)$. Then
$$
a_{i-1}+r_{i-1}\leq a_i+r_i\quad\mbox{and}\quad 
a_i+\xa\leq a_{i-1}
$$
for $i=1,\ldots,\ell_2-1-\nl$.
\end{corollary}

\begin{proof}
By (Cvi) and (Cvii) it follows for $i=0,\ldots,\ell_2-1-\nl$ that
$$
X_{i+\ell_2}=\{(a_i+r_i+1-\xb,i+\ell_2),\ldots,
(a_i+\ell_1-1-\xb,i+\ell_2)\}.
$$
Now if $i>0$, then 
$$
\{(a_i+r_i+1-\xb,i+\ell_2-1),\ldots,
(a_i+\ell_1-1-\xb+\xa,i+\ell_2-1)\}\subseteq X_{i+\ell_2-1}
$$
by (Aiii'). In particular, $a_{i-1}+r_{i-1}\leq a_i+r_i$ and
$a_i+\xa\leq a_{i-1}$.
\end{proof}

\medskip

\section{Generating the derived category}
\label{sec:GenDerCat}
Finally, we set out to prove fullness of any nontrivial 
{\mes} $\cl$ on $(\ell_1,\ell_2;c)$ with $c\not=0$. 

\medskip

{Since $V\subseteq W$, 
a $W$-{\wbox} based at $(a,b)\in\cl$ generates 
the point $(a-\xb,b+\ell_2)$ in $\CD(\ell_1,\ell_2;c)$
by Corollary~\ref{coro:resol}.} In particular, 
the standard rectangle $\Rect_{\ell_1,\ell_2}$ generates
the Picard group in the twisted case, too.

\medskip 

Furthermore, Corollary~\ref{coro:ExtWBox} implies
that $\cl$ contains the set
\begin{equation}
\label{eq:ExtWbox}
\bigcup_{i=1}^\nl C_i
\cup\{(a,b)\in\Z^2\mid0<b<\ell_2-\nl,\,
-b\xa\leq a\leq a_b+r_b\}
\end{equation}

\medskip

{where again $\nl=2\ell_2-H(\cl)$, $C_i$ is a horizontal 
chain in $[y=\ell_2-1-i]\cap\cl$ if $\nl\geq1$, and $r_b$ was defined in Inequality~\eqref{eq:rb} on Page~\pageref{eq:rb}.}
Filling these lines via the $C_i$ shows that $\cl$
contains the $W$-{\wbox}s centered at $[y=0]\cap\cl$.

\medskip

\begin{example}
\label{exam:Exam33}
We illustrate our generation procedure on 
$(\ell_1,\ell_2;\xa,\xb)=(3,3;1,1)$ for the {\mes} 
$$
\cl=\big(0,\,(-1,1),\,(0,1),\,(1,2),\,(2,2),\,
(3,2),\,(0,3),\,(1,3),\,(0,4)\big).
$$
First, 
$$
W=W(3;1)=\{(-1,1),\,(0,1),\,(-2,2),\,(-1,2),\,(0,2)\}.
$$
Filling the line $[y=2]$ in (a) shows that $\langle\cl\rangle$ 
contains $W$ based at the origin. Hence we can 
generate in (b) the point $(-1,3)$ which we use to 
fill the line $[y=3]$ in (c). Therefore, the 
$W$-boxes based at $[y=1]\cap\cl$ are contained 
in $\langle\cl\rangle$. They generate the points 
$(-2,4)$ and $(-1,4)$ in (d) so that together with 
$(0,4)\in\cl$ we fill the line $[y=4]$ in (e), too. 
It follows that $(-2,2)+\Rect_{3,3}
\subseteq\langle\cl\rangle$ whence $\cl$ is full.

\begin{figure}[ht]
\newcommand{\spaceA}{\hspace*{1.5em}}
\newcommand{\scaleA}{0.3}
\begin{tikzpicture}[scale=\scaleA]
\draw[color=oliwkowy!40] (-3.3,-1.3) grid (4.3,5.3);
\fill[color=green!20]
  (-2,2) -- (0,2) -- (0,0);
\draw[thick, color=red]
  (-3,2) -- (4,2);
\fill[thick, color=origin]
  (0,0) circle (9pt);
\fill[thick, color=black]
  (-1,1) circle (9pt) (0,1) circle (9pt) (1,2) circle (9pt)
  (2,2) circle (9pt) (3,2) circle (9pt) (0,3) circle (9pt)
  (1,3) circle (9pt) (0,4) circle (9pt);
\draw[thick, color=black]
   (1,-5) node{(a)};
\end{tikzpicture}
\spaceA
\begin{tikzpicture}[scale=\scaleA]
\draw[color=oliwkowy!40] (-3.3,-1.3) grid (4.3,5.3);
\draw[thick, color=blue]
  (-3,2) -- (5,2);
\draw[thick, color=red]
  (0,0) -- (-1,1) -- (0,1) -- (-1,2) -- (0,2) -- (-1,3);
\fill[thick, color=origin]
  (0,0) circle (9pt);
\fill[thick, color=black]
  (-1,1) circle (9pt) (0,1) circle (9pt) (1,2) circle (9pt)
  (2,2) circle (9pt) (3,2) circle (9pt) (0,3) circle (9pt)
  (1,3) circle (9pt) (0,4) circle (9pt);
\fill[thick, color=red]
  (-2,2) circle (9pt) (-1,2) circle (9pt) (0,2) circle (9pt);
\draw[thick, color=black]
   (1,-5) node{(b)};
\end{tikzpicture}
\spaceA
\begin{tikzpicture}[scale=\scaleA]
\draw[color=oliwkowy!40] (-3.3,-1.3) grid (4.3,5.3);
\draw[thick, color=blue]
  (-3,2) -- (4,2);
\draw[thick, color=blue]
  (0,0) -- (-1,1) -- (0,1) -- (-1,2) -- (0,2) -- (-1,3);
\draw[thick, color=red]
  (-3,3) -- (4,3);
\fill[thick, color=origin]
  (0,0) circle (9pt);
\fill[thick, color=black]
  (-1,1) circle (9pt) (0,1) circle (9pt) (1,2) circle (9pt)
  (2,2) circle (9pt) (3,2) circle (9pt) (0,3) circle (9pt)
  (1,3) circle (9pt) (0,4) circle (9pt);
\fill[thick, color=red]
  (-2,2) circle (9pt) (-1,2) circle (9pt) (0,2) circle (9pt)
  (-1,3) circle (9pt);
\draw[thick, color=black]
   (1,-5) node{(c)};
\end{tikzpicture}
\spaceA
\begin{tikzpicture}[scale=\scaleA]
\draw[color=oliwkowy!40] (-3.3,-1.3) grid (4.3,5.3);
\draw[thick, color=blue]
  (-3,2) -- (4,2);
\draw[thick, color=blue]
  (0,0) -- (-1,1) -- (0,1) -- (-1,2) -- (0,2) -- (-1,3);
\draw[thick, color=blue]
  (-3,3) -- (4,3);
\draw[thick, color=red]
  (-1,1) -- (-2,2) -- (-1,2) -- (-2,3) -- (-1,3) -- (-2,4);
\draw[thick, color=red]
  (0,1) -- (-1,2) -- (0,2) -- (-1,3) -- (0,3) -- (-1,4);
\fill[thick, color=origin]
  (0,0) circle (9pt);
\fill[thick, color=black]
  (-1,1) circle (9pt) (0,1) circle (9pt) (1,2) circle (9pt)
  (2,2) circle (9pt) (3,2) circle (9pt) (0,3) circle (9pt)
  (1,3) circle (9pt) (0,4) circle (9pt);
\fill[thick, color=red]
  (-2,2) circle (9pt) (-1,2) circle (9pt) (0,2) circle (9pt)
  (-1,3) circle (9pt) (-2,3) circle (9pt) (-3,3) circle (9pt);
\draw[thick, color=black]
   (1,-5) node{(d)};
\end{tikzpicture}
\spaceA
\begin{tikzpicture}[scale=\scaleA]
\draw[color=oliwkowy!40] (-3.3,-1.3) grid (4.3,5.3);
\draw[thick, color=blue]
  (-3,2) -- (4,2);
\draw[thick, color=blue]
  (0,0) -- (-1,1) -- (0,1) -- (-1,2) -- (0,2) -- (-1,3);
\draw[thick, color=blue]
  (-1,1) -- (-2,2) -- (-1,2) -- (-2,3) -- (-1,3) -- (-2,4);
\draw[thick, color=blue]
  (0,1) -- (-1,2) -- (0,2) -- (-1,3) -- (0,3) -- (-1,4);
\draw[thick, color=blue]
  (-3,3) -- (4,3);
\draw[thick, color=red]
  (-3,4) -- (4,4);
\fill[thick, color=origin]
  (0,0) circle (9pt);
\fill[thick, color=black]
  (-1,1) circle (9pt) (0,1) circle (9pt) (1,2) circle (9pt)
  (2,2) circle (9pt) (3,2) circle (9pt) (0,3) circle (9pt)
  (1,3) circle (9pt) (0,4) circle (9pt);
\fill[thick, color=red]
  (-2,2) circle (9pt) (-1,2) circle (9pt) (0,2) circle (9pt)
  (-1,3) circle (9pt) (-2,3) circle (9pt) (-3,3) circle (9pt)
  (-1,4) circle (9pt) (-2,4) circle (9pt);
\draw[thick, color=black]
   (1,-5) node{(e)};
\end{tikzpicture}
\caption{Filling $\Pic(3,3;1,1)$ from $\cl$. The green point 
marks $\cl_0=0$. The shaded area in (a) is the $W$-{\wbox}
centered at $[y=0]\cap\cl=\{0\}$.}
\label{fig:InfecExS1}
\end{figure}
\end{example}

\begin{theorem}[Theorem {\thFull}, twisted version]
\label{thm:FullnessTheorem}
On $(\ell_1,\ell_2;c)$, $c\not=0$, any maximal exceptional 
sequence is full. 
\end{theorem}

\begin{proof}
We may assume that $\cl$ is vertically ordered and 
starts at the origin. We continue to use the notation from
Corollary~\ref{coro:ExtWBox}.

\medskip

First, we fill all lines containing a horizontal chain in $\cl$
(if any). Consequently, the lines $[y=\ell_2-1],\ldots,[y=\ell_2-\nl ]$ 
belong to $\langle\cl\rangle$. From 
Equation~\eqref{eq:ExtWbox} on Page~\pageref{eq:ExtWbox}
we conclude that $\langle\cl\rangle$ contains all the $W$-boxes 
centered at the points in $[y=0]\cap\cl$. 
Hence we can generate the points $(-\xb,\ell_2)+([y=0]\cap\cl)$ 
and fill the line $[y=\ell_2]$ which therefore also belongs
to $\langle\cl\rangle$. We can again appeal to 
Equation~\eqref{eq:ExtWbox} to infer that the $W$-boxes based at
$[y=1]\cap\cl$ are contained in $\langle\cl\rangle$. Thus
we can generate the points $(-\xb,\ell_2)+([y=1]\cap\cl)$ 
and fill the line $[y=\ell_2+1]$. After at most $\ell_2-\nl$ 
repetitions we conclude that $(0,\ell_2-\nl)+\Rect_{\ell_1,\ell_2}
\subseteq\langle\cl\rangle$. Hence $\cl$ is full. 
\end{proof}

\bigskip

{\em Acknowledgement}. We would like to thank 
Lutz Hille for stimulating discussions.
Special thanks go to Martin Altmann for his detailed 
comments which considerably improved the paper.
{Last but not least, we are very grateful for the 
comprehensive review of our paper by the anonymous referee.
In particular, it led to a far more elegant proof 
of Theorem~\ref{thm:FullnessTheoremPaPb}.}

\bibliographystyle{alpha}
\bibliography{fesFano}
\end{document}